\documentclass[10pt,reqno,twoside]{amsart}

\usepackage[
  paperwidth=6.73228in,
  paperheight=9.68504in,
  inner=0.748in,
  outer=0.667in,
  top=0.743in,
  bottom=0.725in,
  includehead,
  headheight=10pt,
  headsep=0.19in
]{geometry}

\usepackage{amsmath,amssymb,mathtools}
\usepackage{mathrsfs}
\usepackage[T1]{fontenc}
\usepackage{lmodern}
\usepackage{microtype}
\usepackage{enumitem}
\usepackage{fancyhdr}
\usepackage[colorlinks=true,citecolor=blue,linkcolor=blue,urlcolor=blue]{hyperref}

\setlist{topsep=4pt,itemsep=1pt,parsep=0pt,partopsep=0pt}
\makeatletter
\renewcommand{\section}{%
  \@startsection{section}{1}{\z@}%
    {-2.75ex plus -.5ex minus -.2ex}%
    {1.35ex plus .2ex}%
    {\normalfont\normalsize\bfseries}}
\renewcommand{\subsection}{%
  \@startsection{subsection}{2}{\z@}%
    {-2.35ex plus -.4ex minus -.2ex}%
    {1.05ex plus .2ex}%
    {\normalfont\normalsize\itshape}}
\renewcommand{\subsubsection}{%
  \@startsection{subsubsection}{3}{\z@}%
    {-2ex plus -.4ex minus -.2ex}%
    {.8ex plus .2ex}%
    {\normalfont\normalsize\itshape}}
\makeatother

\fancypagestyle{plain}{%
  \fancyhf{}
  }

\newcommand{\N}{\mathbb N}

\newcommand{\dd}{\mathop{}\!\mathrm d}
\newcommand{\ind}{\mathbf 1}
\renewcommand{\Re}{\operatorname{Re}}
\renewcommand{\Im}{\operatorname{Im}}
\newcommand{\A}{\mathcal A}

\title{On the Spectra of  Multiplicative Hilbert Operator on Bergman Spaces of Dirichlet Series
}

\author{Tengfei Ma}
\address{School of Mathematical Sciences, Dalian University of Technology, Dalian, Liaoning 116024, P. R. China}
\email{503304283@qq.com}

\author{Yufeng Lu}
\address{School of Mathematical Sciences, Dalian University of Technology, Dalian, Liaoning 116024, P. R. China}
\email{lyfdlut@dlut.edu.cn}

\author{Chao Zu\textsuperscript{*}}
\address{School of Mathematical Sciences, Dalian University of Technology, Dalian, Liaoning 116024, P. R. China}
\email{zuchao@dlut.edu.cn}
\thanks{\textsuperscript{*}Corresponding author.}


\date{}

\newtheorem{theorem}{Theorem}[section]
\newtheorem{lemma}[theorem]{Lemma}
\newtheorem{corollary}[theorem]{Corollary}
\theoremstyle{definition}
\newtheorem{example}[theorem]{Example}
\theoremstyle{remark}

\begin{document}

\begin{abstract}
For $\alpha>1$, we study the weighted multiplicative Hilbert operator
\[
    (T_\alpha f)(s)
    =
    \int_{1/2}^{\infty}
        f(t)\bigl(\zeta(s+t)^\alpha-1\bigr)
        \left(t-\frac12\right)^{\alpha-1}\,dt
\]
on the Hilbertian Bergman space $A_{\alpha,0}^2$ of Dirichlet
series with zero constant term.
In the natural orthonormal basis, $T_\alpha$ is represented
by the weighted multiplicative Hilbert matrix
\[
    \left(
        \frac{\Gamma(\alpha)\sqrt{d_\alpha(m)d_\alpha(n)}}
             {\sqrt{mn}\,(\log(mn))^\alpha}
    \right)_{m,n\ge2},
\]
where $d_\alpha$ is the generalized divisor function defined by
$\zeta(s)^\alpha=\sum_{n\ge1}d_\alpha(n)n^{-s}$ for $\Re s>1$.
Set $\Lambda_\alpha=\Gamma(\alpha/2)^2/\Gamma(\alpha)$.
We prove that the essential spectrum of \(T_\alpha\) is \([0,\Lambda_\alpha]\) and that the spectrum in this interval is purely absolutely continuous with multiplicity one.  Above \(\Lambda_\alpha\),
there are at most finitely many eigenvalues, all of which are simple. Moreover, there is a unique \(\alpha_*\in(1,2)\) such that
\(T_\alpha\) has eigenvalues if and only if \(\alpha>\alpha_*\). Consequently, the associated sharp embedding inequality
admits a non-zero extremizer precisely when $\alpha>\alpha_*$,
and this extremizer is unique up to scalar multiplication. We also establish the asymptotic formula
$\|T_\alpha\|\sim\Gamma(\alpha+1)/(2(2\log 2)^\alpha)$
as $\alpha\to\infty$. A key step in the proof combines Mellin-transform asymptotics, strict total positivity, and eigenvalue counting to exclude
embedded eigenvalues. The proof includes a general result on the essential and singular continuous spectra of weighted integral Hankel operators.

\end{abstract}

\maketitle\setcounter{tocdepth}{2}\tableofcontents

\section{Introduction}

\subsection{Background and motivation}

The Hilbert matrix is a classical non-compact Hankel operator. Magnus proved
that its spectrum is \([0,\pi]\) and that it has no eigenvalues, while
Rosenblum gave an explicit diagonalization and showed that its spectrum is
purely absolutely continuous with multiplicity one; see
\cite{Magnus1950,Rosenblum1958I,Rosenblum1958II}. The Bergman--Hilbert matrix was studied in
\cite{DavisGhatage1990,Ghatage1987}. For general background on Hankel operators,
see \cite{Peller2003}.

A multiplicative analogue appears in the theory of Dirichlet series.
Matrices whose entries depend on the product \(mn\), rather than
the sum \(m+n\), are called Helson matrices, or multiplicative
Hankel matrices. They arise naturally on Hardy spaces of Dirichlet
series; see \cite{HedenmalmLindqvistSeip1997} for the Hilbert space
setting and \cite{Helson2006,Helson2010} for Helson forms.
Perfekt and Pushnitski \cite{PerfektPushnitski2018Moments}
studied non-negativity and finite rank of Helson matrices,
obtained moment representations for positive Helson matrices,
and related their boundedness to Carleson measures.
The corresponding Bohr--Bergman spaces of Dirichlet series
were studied in \cite{BailleulBrevig2016}.
Connections between Helson matrices and continuous integral
operators were developed in \cite{MiheisiPushnitski2018}.

The multiplicative Hilbert matrix was introduced in
\cite{BrevigEtAl2016}. On the Hilbert space of Dirichlet series
\[
  \mathscr H_0^2
  =
  \left\{
    f(s)=\sum_{n\geq2}a_nn^{-s}:
    \sum_{n\geq2}|a_n|^2<\infty
  \right\},
\]
it is represented by the integral operator
\[
  (Hf)(s)
  =
  \int_{1/2}^{\infty}
  f(t)\bigl(\zeta(s+t)-1\bigr)\dd t,
  \qquad \Re s>\frac12.
\]
With respect to the orthonormal basis \((n^{-s})_{n\geq2}\) of
\(\mathscr H_0^2\), its matrix is
\[
  \left(
    \frac{1}{\sqrt{mn}\,\log(mn)}
  \right)_{m,n\geq2}.
\]
Combining the results of
\cite{BrevigEtAl2016,PerfektPushnitski2018Spectrum}
gives the following description of its spectrum.

\medskip
\noindent\textbf{Theorem A.}
The operator \(H\) is bounded and strictly positive on \(\mathscr H_0^2\),
and
\[
  \|H\|=\pi.
\]
Moreover,
\[
  \sigma(H)=\sigma_{\mathrm{ac}}(H)=[0,\pi],
  \qquad
  \sigma_{\mathrm{sc}}(H)=\varnothing,
  \qquad
  \sigma_{\mathrm p}(H)=\varnothing,
\]
and the absolutely continuous spectrum has multiplicity one.
\medskip

We now turn to weighted Bergman spaces of Dirichlet series. For
\(\alpha>1\), let \(d_\alpha\) be defined by
\[
  \zeta(s)^\alpha
  =
  \sum_{n=1}^{\infty}d_\alpha(n)n^{-s},
  \qquad \operatorname{Re}s>1.
\]
 The corresponding Hilbertian Bergman space \(\A_\alpha^2\) is given by
\[
  \left\|\sum_{n\geq1}a_nn^{-s}\right\|_{\A_\alpha^2}^2
  =
  \sum_{n\geq1}\frac{|a_n|^2}{d_\alpha(n)},
\]
and its reproducing kernel is
\[
  K_w(s)=\zeta(s+\overline w)^\alpha.
\]
Weighted multiplicative Hankel forms on Bergman spaces of Dirichlet series
were studied in \cite{BayartEtAl2019}. In the case \(\alpha=2\), one of the
examples considered there is the bounded non-compact form
\[
  (f,g)
  \longmapsto
  \int_{1/2}^{\infty}
  f(t)g(t)\left(t-\frac12\right)\dd t,
\]
whose symbol has coefficients
\[
  \frac{d(n)}{\sqrt n\,(\log n)^2}.
\]
The relation between multiplicative Hankel forms and Volterra operators on
weighted Hilbert spaces of Dirichlet series was further studied by
Bommier-Hato \cite{BommierHato2020}.

\subsection{Main results}
Let \(\A_{\alpha,0}^2\) denote the subspace of \(\A_\alpha^2\) consisting
of functions with zero constant term. Its reproducing kernel is
\(\zeta(s+\overline w)^\alpha-1\). For Dirichlet polynomials \(f\in\A_{\alpha,0}^2\), define
\[
  (T_\alpha f)(s)
  =
  \int_{1/2}^{\infty}
  f(t)\bigl(\zeta(s+t)^\alpha-1\bigr)
  \left(t-\frac12\right)^{\alpha-1}\dd t.
\]
Olsen's local embedding theorem \cite{Olsen2011}, combined with
point-evaluation estimates away from \(\Re s=1/2\), shows that
the measure
\[
  \dd\mu_\alpha(t)
  =
  \left(t-\frac12\right)^{\alpha-1}\dd t,
  \qquad t>\frac12,
\]
is a Carleson measure for \(\A_{\alpha,0}^2\).
The boundedness of \(T_\alpha\) then follows by the argument
used in \cite{BrevigEtAl2016}. By the reproducing property, we obtain that, for every Dirichlet polynomial \(f\) with zero constant term,
\[
  \langle T_\alpha f,f\rangle_{\A_\alpha^2}
  =
  \int_{1/2}^{\infty}
  |f(t)|^2\left(t-\frac12\right)^{\alpha-1}\dd t.
\]
Hence \(T_\alpha\) extends to a bounded non-negative operator on
\(\A_{\alpha,0}^2\).

If
\[
  e_n(s)=\sqrt{d_\alpha(n)}\,n^{-s},
  \qquad n\geq2,
\]
then \((e_n)_{n\geq2}\) is an orthonormal basis of
\(\A_{\alpha,0}^2\), and the matrix of \(T_\alpha\) with respect to this
basis is
\[
  \left(
    \frac{\Gamma(\alpha)\sqrt{d_\alpha(m)d_\alpha(n)}}
    {\sqrt{mn}\,(\log(mn))^\alpha}
  \right)_{m,n\geq2}.
\]
Apart from the factor \(\Gamma(\alpha)\), this is the weighted
multiplicative Hankel matrix associated with the sequence
\[
  \frac{d_\alpha(n)}
  {\sqrt n\,(\log n)^\alpha},
  \qquad n\geq2.
\]
At the formal endpoint \(\alpha=1\), the matrix reduces to the
multiplicative Hilbert matrix. For \(\alpha=2\), the corresponding
bilinear form is exactly the bounded non-compact multiplicative Hankel
form considered in \cite{BayartEtAl2019}.

Thus \(T_\alpha\) may be viewed as a weighted Bergman analogue of the
multiplicative Hilbert operator in Theorem A. We study the spectrum of \(T_\alpha\) as a function of \(\alpha\).  The principal difficulty is to determine whether eigenvalues can
occur in the essential spectrum or above its supremum.
The first question requires estimates for solutions of the
eigenvalue equation that are not provided by the analysis of the
continuous spectrum. The second requires a comparison between
the operator norm and the supremum of the essential spectrum
as \(\alpha\) varies.

Our main result is the following.

\begin{theorem}\label{thm:complete-spectrum-Ta}
Let \(\alpha>1\) and define
\[
  \Lambda_\alpha
  =
  \frac{\Gamma(\alpha/2)^2}{\Gamma(\alpha)}.
\]
Then the following assertions hold.
\begin{enumerate}
\item \(T_\alpha\) is bounded and non-negative, and
\[
  \ker T_\alpha=\{0\}.
\]
\item Its essential, absolutely continuous, and singular continuous spectra
satisfy
\[
  \sigma_{\mathrm{ess}}(T_\alpha)
  =
  \sigma_{\mathrm{ac}}(T_\alpha)
  =
  [0,\Lambda_\alpha],
  \qquad
  \sigma_{\mathrm{sc}}(T_\alpha)=\varnothing,
\]
and the absolutely continuous spectrum has multiplicity one.
\item $T_\alpha$ has no eigenvalues in its essential spectrum:
\[
  \sigma_{\mathrm p}(T_\alpha)
  \cap[0,\Lambda_\alpha]=\varnothing.
\]
\item The discrete spectrum above the essential spectrum is finite:
\[
  N\bigl((\Lambda_\alpha,\infty);T_\alpha\bigr)<\infty.
\]
Every eigenvalue of \(T_\alpha\) in
\((\Lambda_\alpha,\infty)\) is simple.
\item There exists a unique $\alpha_*\in(1,2)$ such that the function
\[
  \alpha\longmapsto\frac{\|T_\alpha\|}{\Lambda_\alpha}
\]
is equal to $1$ on $(1,\alpha_*]$ and is strictly increasing on
$[\alpha_*,\infty)$. In particular,
\[
  \sigma_{\mathrm p}(T_\alpha)
  \cap(\Lambda_\alpha,\infty)\neq\varnothing
  \quad\Longleftrightarrow\quad
  \alpha>\alpha_*.
\]
\end{enumerate}
\end{theorem}

Theorem~\ref{thm:complete-spectrum-Ta} shows a change in spectral type as \(\alpha\) increases. For \(1<\alpha\leq\alpha_*\), the spectrum is purely absolutely continuous of multiplicity one, as for the multiplicative Hilbert matrix. For \(\alpha>\alpha_*\), a non-empty finite set of simple eigenvalues appears above \(\Lambda_\alpha\), while the spectrum in \([0,\Lambda_\alpha]\) remains purely absolutely continuous. Thus the passage to Bergman spaces does not merely change the interval of continuous spectrum: it also produces discrete spectrum, which is absent in the Hardy-space case. The critical parameter and the strict increase of \(\|T_\alpha\|/\Lambda_\alpha\) describe this distinction quantitatively.

The spectral transition also determines whether the associated sharp
embedding inequality admits a non-zero extremizer.

\begin{corollary}\label{cor:sharp-embedding}
Let $\alpha>1$. The optimal constant $C_\alpha$ in the inequality
\[
    \int_{1/2}^{\infty}
        |f(t)|^2\left(t-\frac12\right)^{\alpha-1}\,dt
    \le C_\alpha \|f\|_{A_\alpha^2}^{\,2},
    \qquad f\in A_{\alpha,0}^2,
\]
is
\[
    C_\alpha=\|T_\alpha\|.
\]
Moreover, with $\alpha_*$ and $\Lambda_\alpha$ as in
Theorem~\ref{thm:complete-spectrum-Ta}, the following assertions hold.
\begin{enumerate}
    \item If $1<\alpha\le\alpha_*$, then
    $C_\alpha=\Lambda_\alpha$, and equality holds only for $f=0$.

    \item If $\alpha>\alpha_*$, then $C_\alpha>\Lambda_\alpha$,
    and a non-zero extremizer exists and is unique up to
    multiplication by a non-zero scalar.
    The extremizers are precisely the non-zero elements of
    \[
        \ker(T_\alpha-C_\alpha I).
    \]
\end{enumerate}
\end{corollary}

\begin{proof}
For every $f\in A_{\alpha,0}^2$, we have
\[
    \int_{1/2}^{\infty}
        |f(t)|^2\left(t-\frac12\right)^{\alpha-1}\,dt
    =\langle T_\alpha f,f\rangle_{A_\alpha^2}.
\]
Since $T_\alpha$ is non-negative and self-adjoint,
the optimal constant is $C_\alpha=\|T_\alpha\|$.
Furthermore,
\[
    C_\alpha\|f\|_{A_\alpha^2}^{\,2}
        -\langle T_\alpha f,f\rangle_{A_\alpha^2}
    =
    \bigl\|(C_\alpha I-T_\alpha)^{1/2}f
    \bigr\|_{A_\alpha^2}^{\,2}.
\]
Thus equality holds if and only if $T_\alpha f=C_\alpha f$.
For $1<\alpha\le\alpha_*$, Theorem~\ref{thm:complete-spectrum-Ta} gives
$C_\alpha=\Lambda_\alpha$ and shows that $\Lambda_\alpha$
is not an eigenvalue.
For $\alpha>\alpha_*$, the same theorem implies that
$C_\alpha$ lies above the essential spectrum and is a
simple eigenvalue. Both assertions follow.
\end{proof}

\subsection{Proof strategy for Theorem~\ref{thm:complete-spectrum-Ta}}
We reduce the spectral problem for \(T_\alpha\) to a weighted
integral Hankel operator on \(L^2(\mathbb R_+)\). Put
\(\rho=(\alpha-1)/2\) and define
\(N_\alpha:\mathcal A_{\alpha,0}^2\to L^2(\mathbb R_+)\) by
\[
  (N_\alpha f)(x)=x^\rho f(1/2+x),\qquad x>0.
\]
Then \(T_\alpha=N_\alpha^*N_\alpha\), and \(K_\alpha=N_\alpha N_\alpha^*\)
has the representation
\[
  (K_\alpha g)(x)=\int_0^\infty
  (xy)^\rho\bigl(\zeta(1+x+y)^\alpha-1\bigr)g(y)\dd y.
\]
Since \(N_\alpha\) is bounded and injective, its polar decomposition
gives a unitary equivalence between \(T_\alpha\) and
\(K_\alpha|_{(\ker K_\alpha)^\perp}\).

For the absolutely continuous spectrum, the operator \(K_\alpha\) satisfies the hypotheses of
\cite[Theorem~1.3]{FedelePushnitski2017}. Its absolutely continuous
spectrum is therefore \([0,\Lambda_\alpha]\), with multiplicity one.
By unitary equivalence, the same holds for \(T_\alpha\).
To exclude singular continuous spectrum, we use Mourre's theorem
\cite{Mourre1981} and Howland's construction \cite{Howland1992}.
Let \((Jg)(x)=x^{-1}g(x^{-1})\), and choose
\(R\in C^\infty((0,\infty);U(2))\) such that
\[
  R(x)=I\quad(0<x\le1/2),\qquad
  R(x)=\begin{pmatrix}0&-1\\1&0\end{pmatrix}\quad(x\ge2).
\]
On \(\mathcal H=L^2(\mathbb R_+)\oplus L^2(\mathbb R_+)\), define
\[
  (\mathcal RF)(x)=R(x)F(x),\qquad
  T=\mathcal R(K_\alpha\oplus JK_\alpha J^*)\mathcal R^*.
\]
Thus \(T\) is unitarily equivalent to \(K_\alpha\oplus K_\alpha\).
The conjugate operator is \(\mathbf A=A\oplus A\), where \(A\) is
Howland's logarithmic dilation generator
\cite[Eq.~(2.4), p.~418]{Howland1992}:
\[
  A=\frac12(x\log x\,p+p\,x\log x),\qquad
  p=-i\frac{\dd}{\dd x}.
\]
We establish the Mourre estimate
\eqref{eq:transferred-Mourre} and the required commutator regularity.
These imply \(\sigma_{\mathrm{sc}}(T_\alpha)=\varnothing\).
Theorem~\ref{thm:weighted-Hankel-no-sc} below proves the same conclusion for general weighted
integral Hankel operators under the stated assumptions.

To exclude eigenvalues \(\lambda\in(0,\Lambda_\alpha]\), we continue
the Mellin transform of an assumed eigenfunction meromorphically
into \(\Re z<1/2\). There are two cases: either all possible poles
in this half-plane are removable, or a non-removable pole occurs.
In the first case, weighted \(L^2\) estimates and the eigenvalue
equation imply that the eigenfunction is zero.
In the second case, the non-removable poles determine the
oscillatory terms of the eigenfunction near zero. We construct
approximating functions \(v_L\) for which these oscillations give
a lower bound for the number of sign changes \(S(v_L)\).
On the other hand, \(v_L\) belongs to a finite-dimensional spectral
subspace of the restricted operator, so strict total positivity
bounds \(S(v_L)\) by the corresponding eigenvalue count.
The asymptotic formula for this count contradicts the lower bound.

To prove finiteness of the eigenvalues above \(\Lambda_\alpha\), we write
\[
  T=H_0+K,\qquad K=LWL,\qquad W=W^*\in\mathbf S_2,
\]
where \(H_0\) is the homogeneous comparison operator and
\[
  L=L_0\oplus L_0,\qquad (L_0g)(x)=(1+|\log x|)^{-2}g(x).
\]
Replacing \(W\) by its positive part gives
\[
  T\le T_+=H_0+LW_+L,\qquad W_+=(|W|+W)/2.
\]
For \(T_+\), we derive the Birman--Schwinger counting relation
\eqref{eq:Birman-Schwinger-count} using Glazman's lemma
\cite[Theorem~1.25, p.~38]{FrankLaptevWeidl2023}.
The corresponding Birman--Schwinger operator has a rank-one singular
term at \(\Lambda_\alpha\); after subtracting this term, it converges
in operator norm to a compact operator.
This gives a uniform bound on the number of eigenvalues of \(T_+\)
in \((\lambda,\infty)\) as \(\lambda\downarrow\Lambda_\alpha\),
and hence finiteness for \(T_\alpha\).

To prove simplicity, we establish strict total positivity of the kernel
\[
k_\alpha(x,y)
=(xy)^\rho\bigl(\zeta(1+x+y)^\alpha-1\bigr),
\qquad x,y>0,
\]
of $K_\alpha$. The exterior powers of \(K_\alpha\) are therefore
positivity improving. Applying the Perron--Frobenius argument
to suitable exterior powers proves simplicity of each eigenvalue
of \(K_\alpha\) above \(\Lambda_\alpha\).

Finally, an integral identity involving the beta function,
together with a zeta-function inequality derived from
\cite[Theorems~3.1 and~4.1]{Alzer2012}, gives
\[
  \frac{\|K_\alpha\|}{\Lambda_\alpha}
  \le \frac{\|K_\beta\|}{\Lambda_\beta},
  \qquad 1<\alpha<\beta.
\]
The inequality is strict when \(\|K_\alpha\|>\Lambda_\alpha\).
Schur's test gives \(\|K_\alpha\|=\Lambda_\alpha\) for \(\alpha\)
near \(1\), while a test function gives \(\|K_2\|>\Lambda_2\).
The comparison and continuity of the quadratic forms show that
the parameters with \(\|K_\alpha\|>\Lambda_\alpha\) form an open
interval \((\alpha_*,\infty)\), where \(\alpha_*\in(1,2)\).

\subsection{Organization of the paper}

The paper is organized as follows. Section~2 introduces the divisor weights,
the Bergman spaces of Dirichlet series, and weighted Helson forms. Section~3
passes from \(T_\alpha\) to the integral Hankel model and determines the
essential, absolutely continuous, and singular continuous spectra. Section~4
proves the absence of embedded eigenvalues. Section~5 studies eigenvalues
above the essential spectrum, proving their finiteness and simplicity and
establishing the existence and uniqueness of the critical parameter
\(\alpha_*\). Section~6 proves the norm asymptotics and eventual
monotonicity, and presents the questions described above.

\section{Preliminaries}

\subsection{Divisor weights and Bergman spaces of Dirichlet series}

For \(\alpha>0\), define \(d_\alpha(n)\) by
\[
  \zeta(s)^\alpha
  =
  \sum_{n=1}^{\infty}d_\alpha(n)n^{-s},
  \qquad \operatorname{Re}s>1.
\]
For every prime \(p\) and \(k\geq0\),
\[
  d_\alpha(p^k)
  =
  \frac{\Gamma(k+\alpha)}{\Gamma(\alpha)\Gamma(k+1)}.
\]
In particular, \(d_1(n)=1\) and \(d_2(n)=d(n)\), the ordinary divisor
function.

For \(\alpha>1\), the Hilbertian Bergman space of Dirichlet series
\(\A_\alpha^2\) may be described as the completion of Dirichlet polynomials
for the norm
\[
  \left\|\sum_{n\geq1}a_nn^{-s}\right\|_{\A_\alpha^2}^2
  =
  \sum_{n\geq1}\frac{|a_n|^2}{d_\alpha(n)}.
\]
This is the Hilbert-space realization obtained from the Bergman--Bohr
construction; see \cite{BayartEtAl2019,BommierHato2020}. It follows from the
definition of \(d_\alpha\) and the Cauchy--Schwarz inequality that point
evaluation is bounded in the half-plane \(\operatorname{Re}s>1/2\). More
precisely,
\[
  |f(s)|^2
  \leq
  \|f\|_{\A_\alpha^2}^2
  \sum_{n\geq1}d_\alpha(n)n^{-2\operatorname{Re}s}
  =
  \|f\|_{\A_\alpha^2}^2
  \zeta(2\operatorname{Re}s)^\alpha.
\]
Consequently the reproducing kernel is
\[
  K_w(s)=\zeta(s+\overline w)^\alpha.
\]
We write
\[
  \A_{\alpha,0}^2
  =
  \{f\in\A_\alpha^2:f(+\infty)=0\}
\]
for the closed subspace of Dirichlet series with zero constant coefficient.
Its reproducing kernel is
\[
  K_w^0(s)=\zeta(s+\overline w)^\alpha-1,
\]
and
\[
  e_n(s)=\sqrt{d_\alpha(n)}\,n^{-s},
  \qquad n\geq2,
\]
is an orthonormal basis of \(\A_{\alpha,0}^2\).

\subsection{Spectral preliminaries}

For a bounded self-adjoint operator \(S\), we write
\[
  \sigma(S),\quad
  \sigma_{\mathrm{ess}}(S),\quad
  \sigma_{\mathrm{ac}}(S),\quad
  \sigma_{\mathrm{sc}}(S),\quad
  \sigma_{\mathrm p}(S)
\]
for its spectrum, essential spectrum, absolutely continuous spectrum,
singular continuous spectrum, and point spectrum, respectively. If
\(I\subset\mathbb R\) is an interval, then \(N(I;S)\) denotes the number of
eigenvalues of \(S\) in \(I\), counted with multiplicities.

We use \(U\lesssim V\) when \(U\leq CV\) for a constant \(C>0\) independent
of the variables under consideration. The notation \(U\gtrsim V\) has the
reverse meaning, and \(U\asymp V\) means that both \(U\lesssim V\) and
\(U\gtrsim V\) hold. Unless stated otherwise, the constants may depend on
parameters that are fixed in the relevant argument.

\section{The operator \texorpdfstring{\(T_\alpha\)}{T-alpha} and its continuous spectrum}

We begin by passing from \(T_\alpha\) to an weighted integral Hankel operator.

\subsection{Boundedness, non-negativity, and the kernel}

Let \(\mathcal P_0\) denote the Dirichlet polynomials with zero constant
coefficient.  We shall use the standard unitary equivalence of \(N^*N\)
and \(NN^*\) away from their kernels; see, for example,
\cite[Section~2.5, p.~128]{FedelePushnitski2017}.

\begin{theorem}\label{thm:Ta-unitary-Ka}
Let \(\alpha>1\).  Then \(T_\alpha\geq0\) and
\(\ker T_\alpha=\{0\}\).  Moreover, \(T_\alpha\) is unitarily
equivalent to
\[ K_\alpha\big|_{(\ker K_\alpha)^\perp}, \]
where \(K_\alpha\) is the bounded non-negative self-adjoint operator on
\(L^2(\mathbb R_+)\) given by
\[ (K_\alpha F)(u) = \int_0^\infty \bigl(\zeta(1+u+v)^\alpha-1\bigr) (uv)^{(\alpha-1)/2}F(v)\dd v. \]
\end{theorem}

\begin{proof}
Define
\[ N_\alpha:\mathcal P_0\subset\A^2_{\alpha,0} \longrightarrow L^2(\mathbb R_+) \]
by
\[ (N_\alpha f)(u) = u^{(\alpha-1)/2} f\left(\frac12+u\right), \qquad u>0. \]
We first prove on \(\mathcal P_0\) that
\[ T_\alpha=N_\alpha^*N_\alpha. \]
Let
\[ f(s)=\sum_{m\geq2}a_mm^{-s}, \qquad g(s)=\sum_{n\geq2}b_nn^{-s}. \]
Then
\[
\begin{split}
  \langle N_\alpha f,N_\alpha g\rangle_{L^2(\mathbb R_+)}
  &=
  \int_0^\infty
  f\left(\frac12+u\right)
  \overline{g\left(\frac12+u\right)}
  u^{\alpha-1}\dd u\\
  &=
  \sum_{m,n\geq2}
  \frac{a_m\overline{b_n}}{\sqrt{mn}}
  \int_0^\infty
  (mn)^{-u}u^{\alpha-1}\dd u\\
  &=
  \sum_{m,n\geq2}
  \frac{\Gamma(\alpha)a_m\overline{b_n}}
       {\sqrt{mn}\,(\log(mn))^\alpha}\\
  &=
  \langle T_\alpha f,g\rangle_{\A^2_{\alpha,0}}.
\end{split}
\]
The boundedness of \(T_\alpha\) on \(\A^2_{\alpha,0}\) gives
\[
\begin{aligned}
\|N_\alpha f\|_{L^2(\mathbb R_+)}^2
&=\langle T_\alpha f,f\rangle_{\A^2_{\alpha,0}}\\
&\leq\|T_\alpha\|\,\|f\|_{\A^2_{\alpha,0}}^2.
\end{aligned}
\]
Thus \(N_\alpha\) extends uniquely to a bounded operator
\[ N_\alpha:\A^2_{\alpha,0}\longrightarrow L^2(\mathbb R_+), \]
and
\[ T_\alpha=N_\alpha^*N_\alpha \]
on \(\A^2_{\alpha,0}\).  In particular, \(T_\alpha\) is self-adjoint
and non-negative.

We next compute \(N_\alpha N_\alpha^*\).  If
\(F\in C_c^\infty(\mathbb R_+)\), then
\[
(N_\alpha^*F)(s)
=\sum_{n\geq2}d_\alpha(n)n^{-s}
\left(\int_0^\infty n^{-1/2-v}v^{(\alpha-1)/2}F(v)\dd v\right).
\]
Therefore
\[
\begin{split}
  (N_\alpha N_\alpha^*F)(u)
  &=
  u^{(\alpha-1)/2}
  \int_0^\infty
  \left(
    \sum_{n\geq2}d_\alpha(n)n^{-1-u-v}
  \right)
  v^{(\alpha-1)/2}F(v)\dd v\\
  &=
  u^{(\alpha-1)/2}
  \int_0^\infty
  \bigl(\zeta(1+u+v)^\alpha-1\bigr)
  v^{(\alpha-1)/2}F(v)\dd v.
\end{split}
\]
Hence
\[ N_\alpha N_\alpha^*=K_\alpha. \]
It follows that \(K_\alpha\) is bounded, self-adjoint, and
non-negative.

The identity \(T_\alpha=N_\alpha^*N_\alpha\) implies
\[ \ker T_\alpha=\ker N_\alpha. \]
If \(f\in\ker N_\alpha\), then
\[ f\left(\frac12+u\right)=0 \]
for almost every \(u>0\).  Every function in
\(\A^2_{\alpha,0}\) is analytic in the half-plane
\(\Re s>1/2\), and the identity theorem therefore gives \(f\equiv0\).
Thus
\[ \ker T_\alpha=\{0\}. \]

For every bounded operator \(N\) between Hilbert spaces, the
restrictions of \(N^*N\) and \(NN^*\) to the orthogonal complements of
their kernels are unitarily equivalent.  Applying this fact to
\(N_\alpha\) proves
\[ T_\alpha \simeq K_\alpha\big|_{(\ker K_\alpha)^\perp}. \]
\end{proof}

\subsection{Essential, absolutely continuous, and singular continuous spectra}

We follow Howland \cite{Howland1992} in the choice of the conjugate
operator and the construction of the models at zero and infinity.
For weighted Hankel operators, the commutator estimates must also
control derivatives of the weight and the logarithmic factors
introduced by the conjugate operator. Under the assumptions of
Theorem~\ref{thm:weighted-Hankel-no-sc}, we prove that the perturbation
and its first commutator are compact, while its second commutator
is bounded. No quantitative rate of convergence is required for
the normalized kernel or weight themselves: their finite limits,
together with the bounds on their second logarithmic derivatives,
give the necessary estimates. We can then transfer the Mourre
estimate from the models to the weighted operator and exclude
singular continuous spectrum.

\begin{theorem}[Mourre {\cite[Theorem~2.1, p.~417]{Howland1992}}]
\label{thm:Mourre}
Let \(H\) and \(A\) be self-adjoint operators on a Hilbert space, with
\(H\) bounded.  Assume that the quadratic-form commutators
\([A,H]\) and \([A,[A,H]]\), initially defined on
\(\mathcal D(A)\), extend to bounded operators.  Let
\(\Delta\subset\mathbb R\) be an open interval and let
\(E_\Delta(H)\) be the corresponding spectral projection.  Suppose
that there are \(\delta>0\) and a compact self-adjoint operator \(K\)
such that
\begin{equation}\label{eq:Mourre-estimate} E_\Delta(H)i[A,H]E_\Delta(H) \geq \delta E_\Delta(H)+K.\end{equation}
Then \(H\) has no singular continuous spectrum in \(\Delta\). Moreover, $H$ has at most finitely many eigenvalues in \(\Delta\), and every such eigenvalue has finite multiplicity.
\end{theorem}

We now verify the hypotheses of this theorem for weighted Hankel operators whose normalized kernel and weight have limits at both endpoints.

Let \(\rho>-1/2\), let \(b:(0,\infty)\to\mathbb R\) and \(w:(0,\infty)\to\mathbb C\), and consider
\[ (Hf)(x)=\int_0^\infty w(x)b(x+y)\overline{w(y)}f(y)\dd y \]
on \(L^2(\mathbb R_+)\).  Define
\[ a(t)=t^{1+2\rho}b(t), \qquad \omega(t)=t^{-\rho}w(t). \]
For a function $f$ of one variable, write
\[
(Df)(t)=tf'(t),
\qquad
(D^2f)(t)=t^2f''(t)+tf'(t),
\]
with $D^0f=f$.  We also write \(\langle s\rangle=1+|s|\).

The next theorem expresses the essential and singular continuous spectra in terms of the limits at zero and infinity.

\begin{theorem}\label{thm:weighted-Hankel-no-sc}
Assume that \(a,\omega\in C^2(0,\infty)\), and that there are
\(a_0,a_\infty\in\mathbb R\) and
\(\omega_0,\omega_\infty\in\mathbb C\) such that
\[ a(t)\to a_0, \qquad \omega(t)\to\omega_0, \qquad t\to0, \]
and
\[ a(t)\to a_\infty, \qquad \omega(t)\to\omega_\infty, \qquad t\to\infty. \]
Assume, in addition, that, for \(f=a,\omega\),
\[ |(D^2f)(t)|\leq C\langle\log t\rangle^{-2}, \qquad t>0. \]
Set
\[
\begin{gathered}
q_0=a_0|\omega_0|^2,\qquad q_\infty=a_\infty|\omega_\infty|^2,\\
\pi_\rho=\frac{\Gamma(\rho+1/2)^2}{\Gamma(1+2\rho)}.
\end{gathered}
\]
Then \(H\) extends to a bounded self-adjoint operator and
\[
\sigma_{\mathrm{ess}}(H)
=q_0[0,\pi_\rho]\cup q_\infty[0,\pi_\rho],
\qquad
\sigma_{\mathrm{sc}}(H)=\varnothing.
\]
For $q<0$, we write $q[0,\pi_\rho]=[q\pi_\rho,0]$.

Moreove, if \(\Delta\) is a non-empty bounded open interval such that
\[ \overline\Delta \subset \bigl(q_0(0,\pi_\rho)\cup q_\infty(0,\pi_\rho)\bigr) \setminus\{q_0\pi_\rho,q_\infty\pi_\rho\}, \]
then \(H\) has no singular continuous spectrum in \(\Delta\), and
there are only finitely many eigenvalues in \(\Delta\), each of finite
multiplicity.
\end{theorem}

\begin{proof}
Put
\[
  h_\rho(x,y)=\frac{x^\rho y^\rho}{(x+y)^{1+2\rho}},
\]
and let \(C_\rho\) be the integral operator on \(L^2(\mathbb R_+)\)
with kernel \(h_\rho\).  We organize the proof into five steps.

\smallskip
\noindent\emph{Step 1: the model operator and its commutators.}
The Mellin transform
\[
  (\mathcal Mf)(\xi)
  =\frac{1}{\sqrt{2\pi}}
   \int_0^\infty x^{-1/2+i\xi}f(x)\dd x
\]
is unitary from \(L^2(\mathbb R_+)\) to \(L^2(\mathbb R)\) and satisfies
\begin{equation}\label{eq:Mellin-model}
  \mathcal MC_\rho\mathcal M^*=M_{m_\rho},
  \qquad
  m_\rho(\xi)
  =\frac{|\Gamma(\rho+1/2+i\xi)|^2}{\Gamma(1+2\rho)}.
\end{equation}
The Gamma product formula \cite[Eq.~(5.8.3)]{NIST2010} gives, with
\(s=\rho+1/2\),
\[
  |\Gamma(s+i\xi)|^2
  =\Gamma(s)^2
   \prod_{n=0}^\infty
   \frac{(n+s)^2}{(n+s)^2+\xi^2}.
\]
Hence
\[
  \frac{m_\rho'(\xi)}{m_\rho(\xi)}
  =-2\xi\sum_{n=0}^\infty
   \frac{1}{(n+s)^2+\xi^2},
\]
so
\begin{equation}\label{eq:model-monotonicity}
  -\xi m_\rho'(\xi)>0,
  \qquad \xi\neq0.
\end{equation}
Thus \(m_\rho\) is positive and even, strictly decreasing on
\((0,\infty)\), satisfies \(m_\rho(0)=\pi_\rho\), and tends to zero
as \(|\xi|\to\infty\).  On each of \((0,\infty)\) and
\(( -\infty,0)\), the substitution \(\lambda=m_\rho(\xi)\)
represents \(M_{m_\rho}\) as multiplication by \(\lambda\) on a
weighted \(L^2\)-space over \((0,\pi_\rho)\).  Therefore
\[
  \sigma(C_\rho)=[0,\pi_\rho],
\]
the spectrum is absolutely continuous, and its multiplicity on
\((0,\pi_\rho)\) equals two.

Let \(p=-i\dd/\dd x\), and define
\[
  A=\frac12(x\log x\,p+p\,x\log x)
   =-i\left(
      x\log x\,\frac{\dd}{\dd x}
      +\frac{\log x+1}{2}
    \right).
\]
In the Mellin representation,
\[
  \mathcal M i[A,C_\rho]\mathcal M^*
  =M_{-\xi m_\rho'(\xi)}
\]
and
\[
  \mathcal M[iA,[iA,C_\rho]]\mathcal M^*
  =M_{\xi(\xi m_\rho'(\xi))'}.
\]
Since $m_\rho\in C^\infty(\mathbb R)$, the functions
\[
\xi m'_\rho(\xi)
\qquad\text{and}\qquad
\xi(\xi m'_\rho(\xi))'
\]
are bounded on compact intervals. For large $|\xi|$, Stirling's formula gives
\[
m_\rho(\xi)
=
O\bigl(|\xi|^{2\rho}e^{-\pi|\xi|}\bigr),
\qquad |\xi|\to\infty,
\]
and the same exponential decay, up to polynomial factors, holds after one or
two differentiations. Hence
\[
\xi m'_\rho(\xi),\qquad
\xi(\xi m'_\rho(\xi))'
\]
are bounded on $\mathbb R$. Therefore $[A,C_\rho]$ and
$[A,[A,C_\rho]]$, initially defined as quadratic forms on $\mathcal D(A)$, extend to
bounded operators.

\smallskip
\noindent\emph{Step 2: the two-channel model and its strict commutator estimate.}
The kernel of \(H\) is
\[
  k(x,y)
  =h_\rho(x,y)\omega(x)a(x+y)\overline{\omega(y)}.
\]
The boundedness of \(a\) and \(\omega\) gives
\[
  |k(x,y)|
  \leq
  \|a\|_\infty\|\omega\|_\infty^2 h_\rho(x,y).
\]
The boundedness of \(C_\rho\) therefore gives
\[
  |(Hf)(x)|
  \leq
  \|a\|_\infty\|\omega\|_\infty^2(C_\rho|f|)(x).
\]
Thus \(H\) is bounded.  Its kernel satisfies
\(k(x,y)=\overline{k(y,x)}\), so \(H\) is self-adjoint.

Define
\[
  (Jf)(x)=x^{-1}f(x^{-1}).
\]
The operator \(J\) is unitary, and \(\widetilde H=JHJ^*\) has kernel
\[
  \widetilde k(x,y)
  =h_\rho(x,y)\omega(x^{-1})a(x^{-1}+y^{-1})
   \overline{\omega(y^{-1})}.
\]
On
\[
  \mathcal H=L^2(\mathbb R_+)\oplus L^2(\mathbb R_+),
\]
put \(\mathbf H=H\oplus\widetilde H\).  Choose
\(R\in C^\infty((0,\infty);U(2))\) such that
\[
  R(x)=I\quad(0<x\leq1/2),
  \qquad
  R(x)=S:=
  \begin{pmatrix}
    0&-1\\
    1&0
  \end{pmatrix}
  \quad(x\geq2),
\]
and let \(\mathcal R\) be multiplication by \(R(x)\).  Define
\[
  T=\mathcal R\mathbf H\mathcal R^*,
  \qquad
  H_0=q_0C_\rho\oplus q_\infty C_\rho,
  \qquad
  K=T-H_0.
\]
Set
\[
  \sigma_1=
  \begin{cases}
    \operatorname{sgn}(q_0),&q_0\neq0,\\
    1,&q_0=0,
  \end{cases}
  \qquad
  \sigma_2=
  \begin{cases}
    \operatorname{sgn}(q_\infty),&q_\infty\neq0,\\
    1,&q_\infty=0,
  \end{cases}
\]
and
\[
  \mathbf A=\sigma_1A\oplus\sigma_2A.
\]
Then
\[
  i[\mathbf A,H_0]
  \simeq
  M_{|q_0|(-\xi m_\rho'(\xi))}
  \oplus
  M_{|q_\infty|(-\xi m_\rho'(\xi))}.
\]

Let \(\Delta\) be a non-empty bounded open interval satisfying the
condition in the theorem, and define
\[
  \Xi_0
  =\{\xi\in\mathbb R:q_0m_\rho(\xi)\in\overline\Delta\},
  \qquad
  \Xi_\infty
  =\{\xi\in\mathbb R:q_\infty m_\rho(\xi)\in\overline\Delta\}.
\]
If \(q_0=0\), then \(\Xi_0=\varnothing\); if
\(q_\infty=0\), then \(\Xi_\infty=\varnothing\).  Each non-empty
set \(\Xi_0\) or \(\Xi_\infty\) is a compact subset of
\(\mathbb R\setminus\{0\}\).  If \((q_0,q_\infty)\neq(0,0)\),
then \eqref{eq:model-monotonicity} gives
\[
  c_\Delta
  :=\min
  \left(
    \{|q_0|(-\xi m_\rho'(\xi)):\xi\in\Xi_0\}
    \cup
    \{|q_\infty|(-\xi m_\rho'(\xi)):\xi\in\Xi_\infty\}
  \right)>0,
\]
where the empty set is omitted from the union.  The spectral calculus
then gives
\[
  E_\Delta(H_0)i[\mathbf A,H_0]E_\Delta(H_0)
  \geq c_\Delta E_\Delta(H_0).
\]

\smallskip
\noindent\emph{Step 3: estimates for $K$ and its commutator kernels.}
The matrix kernel of $K$ is
\begin{equation}\label{eq:matrix-perturbation-kernel}
  \mathbf K(x,y)
  =R(x)
   \begin{pmatrix}k(x,y)&0\\0&\widetilde k(x,y)\end{pmatrix}
   R(y)^*
   -h_\rho(x,y)
   \begin{pmatrix}q_0&0\\0&q_\infty\end{pmatrix}.
\end{equation}
Write
\[
  B_x=x\log x\,\partial_x+\frac{\log x+1}{2},\qquad
  B_y=y\log y\,\partial_y+\frac{\log y+1}{2},\qquad
  \mathcal D_{jk}=\sigma_jB_x+\sigma_kB_y.
\]
The kernels of $i[\mathbf A,K]$ and
$[i\mathbf A,[i\mathbf A,K]]$ have entries
$\mathcal D_{jk}K_{jk}$ and $\mathcal D_{jk}^2K_{jk}$,
respectively.

Let $f=a$ or $f=\omega$, and set $F(s)=f(e^s)$.
Then $F$ has finite limits at both ends of $\mathbb R$ and
$|F''(s)|\lesssim\langle s\rangle^{-2}$. For $s>0$ and
$0<\delta<1$, Taylor's formula gives
\[
  \delta sF'(s)
  =F((1+\delta)s)-F(s)
   -\int_s^{(1+\delta)s}((1+\delta)s-u)F''(u)\dd u.
\]
For sufficiently large $s$, the absolute value of the integral is at most
$C\delta^2/2$. Since $F((1+\delta)s)-F(s)\to0$, the displayed identity gives
\[
  \limsup_{s\to\infty}s|F'(s)|\lesssim\delta/2.
\]
Letting $\delta\downarrow0$ proves $sF'(s)\to0$.
The same argument applied to $F(-s)$ treats $s\to-\infty$.
Returning to $t=e^s$ and using $F'(s)=(Df)(e^s)$, we obtain
\begin{equation}\label{eq:first-derivative-decay}
  \langle\log t\rangle |(Df)(t)|\longrightarrow0
  \quad(t\to0,\ t\to\infty),\qquad
  \sup_{t>0}\bigl(|f(t)|+|(Df)(t)|+|(D^2f)(t)|\bigr)<\infty.
\end{equation}
For $f=a,\omega$, differentiation of the composite function
$t\mapsto f(t^{-1})$ gives
\[
\begin{aligned}
t\frac{\dd}{\dd t}[f(t^{-1})]
&=-t^{-1}f'(t^{-1})
=-(Df)(t^{-1}),\\
\left(t\frac{\dd}{\dd t}\right)^2[f(t^{-1})]
&=t^{-1}f'(t^{-1})+t^{-2}f''(t^{-1})
=(D^2f)(t^{-1}).
\end{aligned}
\]
Thus, for $f=a,\omega$,
\[
\langle\log t\rangle
\left|t\frac{\dd}{\dd t}[f(t^{-1})]\right|
\longrightarrow0,
\qquad t\to0,\infty,
\]
and
\[
\sup_{t>0}
\sum_{j=0}^2
\left|
\left(t\frac{\dd}{\dd t}\right)^j[f(t^{-1})]
\right|<\infty.
\]

Set
\[
  \theta=\frac{x}{x+y},\qquad
  L(\theta)=1+|\log\theta|+|\log(1-\theta)|,\qquad
  P=x\log x\,\partial_x+y\log y\,\partial_y.
\]
Direct differentiation gives
\[
 \frac{(B_x+B_y)h_\rho}{h_\rho}
 =1+\left(\rho+\frac12\right)(1-2\theta)
       \log\frac{\theta}{1-\theta}
\]
and
\[
 P\left[\frac{(B_x+B_y)h_\rho}{h_\rho}\right]
 =\left(\rho+\frac12\right)
 \left[-2\theta(1-\theta)
       \left(\log\frac{\theta}{1-\theta}\right)^2
       +(1-2\theta)\log\frac{\theta}{1-\theta}\right].
\]
Since
\[
 \left|\log\frac{\theta}{1-\theta}\right|\le L(\theta),
 \qquad
 \sup_{0<\theta<1}\theta(1-\theta)
       \log^2\frac{\theta}{1-\theta}<\infty,
\]
the preceding identities give
\[
 \left|\frac{(B_x+B_y)h_\rho}{h_\rho}\right|
 \lesssim L(\theta),
 \qquad
 \left|P\left[\frac{(B_x+B_y)h_\rho}{h_\rho}\right]\right|
 \lesssim L(\theta).
\]
Moreover, the product rule gives
\[
 \frac{(B_x+B_y)^2h_\rho}{h_\rho}
 =
 P\left[\frac{(B_x+B_y)h_\rho}{h_\rho}\right]
 +
 \left(\frac{(B_x+B_y)h_\rho}{h_\rho}\right)^2.
\]
Combining these estimates with $L(\theta)\ge1$ gives
\[
 |(B_x+B_y)^r h_\rho(x,y)|
 \lesssim h_\rho(x,y)L(\theta)^r,
 \qquad r=0,1,2.
\]
For a $C^2$ function $c=c(x,y)$, the product rule reads
\[
\begin{aligned}
 (B_x+B_y)(h_\rho c)
   &=h_\rho Pc+((B_x+B_y)h_\rho)c,\\
 (B_x+B_y)^2(h_\rho c)
   &=h_\rho P^2c+2((B_x+B_y)h_\rho)Pc
     +((B_x+B_y)^2h_\rho)c.
\end{aligned}
\]
Thus
\begin{equation}\label{eq:diagonal-master-estimate}
 |(B_x+B_y)^r(h_\rho c)|
 \lesssim h_\rho L(\theta)^r\sum_{\ell=0}^r|P^\ell c|,
 \qquad r=0,1,2.
\end{equation}

On $(0,1/2]^2$ the only non-zero entries of $\mathbf K$ are
$K_{11}=h_\rho c_0$ and $K_{22}=h_\rho\widetilde c_0$, where
\[
\begin{aligned}
 c_0&=\omega(x)a(x+y)\overline{\omega(y)}-q_0,\\
 \widetilde c_0
 &=\omega(x^{-1})a(x^{-1}+y^{-1})
   \overline{\omega(y^{-1})}-q_\infty.
\end{aligned}
\]
On $[2,\infty)^2$ they are
$K_{11}=h_\rho\widetilde c_\infty$ and
$K_{22}=h_\rho c_\infty$, where
\[
\begin{aligned}
 c_\infty&=\omega(x)a(x+y)\overline{\omega(y)}-q_\infty,\\
 \widetilde c_\infty
 &=\omega(x^{-1})a(x^{-1}+y^{-1})
   \overline{\omega(y^{-1})}-q_0.
\end{aligned}
\]
For $u=x+y$ and $v=x^{-1}+y^{-1}$,
\[
\begin{aligned}
 P[a(u)]
 &=\bigl(\theta\log x+(1-\theta)\log y\bigr)(Da)(u),\\
 P[a(v)]
 &=\bigl((1-\theta)\log(x^{-1})
          +\theta\log(y^{-1})\bigr)(Da)(v).
\end{aligned}
\]
Using $x=\theta u$ and $y=(1-\theta)u$, we obtain
\[
  \theta\log x+(1-\theta)\log y
  =\log u+\theta\log\theta+(1-\theta)\log(1-\theta).
\]
Similarly, since $x^{-1}=(1-\theta)v$ and $y^{-1}=\theta v$,
\[
  (1-\theta)\log(x^{-1})+\theta\log(y^{-1})
  =\log v+\theta\log\theta+(1-\theta)\log(1-\theta).
\]
Differentiating these expressions gives
\[
\begin{aligned}
 P[\theta\log x+(1-\theta)\log y]
 &=\theta\log x+(1-\theta)\log y
   +\theta(1-\theta)\log^2(x/y),\\
 P[(1-\theta)\log(x^{-1})+\theta\log(y^{-1})]
 &=(1-\theta)\log(x^{-1})+\theta\log(y^{-1})
   +\theta(1-\theta)\log^2(x/y).
\end{aligned}
\]
The term $\theta(1-\theta)\log^2(x/y)$ is bounded.
Consequently,
\[
  |P[a(u)]|
  \lesssim \langle\log u\rangle|(Da)(u)|,
\]
and
\[
  |P^2[a(u)]|
  \lesssim \bigl(
    \langle\log u\rangle|(Da)(u)|
    +\langle\log u\rangle^2|(D^2a)(u)|
  \bigr).
\]
The same estimates hold with $u$ replaced by $v$.
For the factors involving $\omega$,
\[
\begin{aligned}
 P[\omega(x)]&=(\log x)(D\omega)(x),\\
 P[\omega(x^{-1})]
  &=\log(x^{-1})(D\omega)(x^{-1}),\\
 P^2[\omega(x)]
  &=(\log x)(D\omega)(x)
    +(\log x)^2(D^2\omega)(x),\\
 P^2[\omega(x^{-1})]
  &=\log(x^{-1})(D\omega)(x^{-1})
    +\log^2(x^{-1})(D^2\omega)(x^{-1}).
\end{aligned}
\]
The identities with $y$ in place of $x$ hold as well; complex
conjugation does not change these bounds. By
\eqref{eq:first-derivative-decay} and the product rule,
\[
\begin{aligned}
 |c_0|+|Pc_0|+|\widetilde c_0|+|P\widetilde c_0|
 &\longrightarrow0 &&(x,y\to0),\\
 |c_\infty|+|Pc_\infty|
  +|\widetilde c_\infty|+|P\widetilde c_\infty|
 &\longrightarrow0 &&(x,y\to\infty),
\end{aligned}
\]
and $|P^2c_0|+|P^2\widetilde c_0|$ is bounded on
$(0,1/2]^2$, while $|P^2c_\infty|+|P^2\widetilde c_\infty|$
is bounded on $[2,\infty)^2$. Using $\mathcal D_{jj}=\sigma_j(B_x+B_y)$ in
\eqref{eq:diagonal-master-estimate}, we obtain
\begin{equation}\label{eq:diagonal-global-bound}
 \sum_{j=1}^2|\mathcal D_{jj}^rK_{jj}(x,y)|
 \lesssim h_\rho(x,y)L(\theta)^r,
 \qquad r=0,1,2,
\end{equation}
for $(x,y)\in(0,1/2]^2\cup[2,\infty)^2$. For $r=0,1$,
\begin{equation}\label{eq:diagonal-vanishing-estimates}
\begin{aligned}
 \lim_{\delta\downarrow0}\sup_{0<x,y\le\delta}
 \frac{\sum_{j=1}^2|\mathcal D_{jj}^rK_{jj}(x,y)|}
      {h_\rho(x,y)L(\theta)^r}&=0,\\
 \lim_{R\to\infty}\sup_{x,y\ge R}
 \frac{\sum_{j=1}^2|\mathcal D_{jj}^rK_{jj}(x,y)|}
      {h_\rho(x,y)L(\theta)^r}&=0.
\end{aligned}
\end{equation}

For $m=0,1,2$, substitution $y=xt$ gives
\[
 \sup_{x>0}x^{1/2}\int_0^\infty
 h_\rho(x,y)L(\theta)^m y^{-1/2}\dd y
 =\int_0^\infty
 \frac{t^{\rho-1/2}}{(1+t)^{1+2\rho}}
 L\!\left(\frac1{1+t}\right)^m\dd t<\infty.
\]
The corresponding integral with $x$ and $y$ interchanged has
the same value. Both integrals are finite because $\rho>-1/2$. Thus, for $m=0,1,2$, the integral operator with kernel
$h_\rho(x,y)L(\theta)^m$ is bounded on $L^2(\mathbb R_+)$
by Schur's test.

We next estimate derivatives with respect to $x$ and $y$
separately. For $u=x+y$ and $v=x^{-1}+y^{-1}$, the chain rule
and the boundedness of $a$, $Da$, and $D^2a$ give
\[
\begin{aligned}
&\left|(x\partial_x)^p(y\partial_y)^q[a(u)]\right|
+\left|(x\partial_x)^p(y\partial_y)^q[a(v)]\right|
\lesssim1,
\end{aligned}
\]
for all non-negative integers $p,q$ with $p+q\leq2$.
Indeed, the coefficients of $a$, $Da$, and $D^2a$ in these
derivatives are polynomials in $\theta$ and $1-\theta$,
where $0<\theta<1$.
For the weight factors,
\[
\begin{aligned}
\left|(x\partial_x)^j[\omega(x)]\right|
&=|(D^j\omega)(x)|\lesssim1,\\
\left|(x\partial_x)^j[\omega(x^{-1})]\right|
&=|(D^j\omega)(x^{-1})|\lesssim1,
\qquad j=0,1,2.
\end{aligned}
\]
The same bounds hold with $y$ in place of $x$ and for the
complex conjugates. The smoothness of $R$ and its constancy
outside $[1/2,2]$ imply
\[
\sup_{t>0}\sum_{j=0}^2
\left\|(t\partial_t)^j[R(t)]\right\|<\infty.
\]
For $h_\rho$,
\[
\frac{x\partial_xh_\rho}{h_\rho}
=\rho-(1+2\rho)\theta,
\qquad
\frac{y\partial_yh_\rho}{h_\rho}
=\rho-(1+2\rho)(1-\theta).
\]
Together with
\[
x\partial_x\theta=\theta(1-\theta),
\qquad
y\partial_y\theta=-\theta(1-\theta),
\]
these identities give
\[
\left|((x\partial_x)^p(y\partial_y)^qh_\rho)(x,y)\right|
\lesssim h_\rho(x,y),
\qquad p+q\leq2.
\]
Applying the product rule to
\eqref{eq:matrix-perturbation-kernel}, we obtain
\[
\sum_{p+q\leq2}
\left|((x\partial_x)^p(y\partial_y)^qK_{jk})(x,y)\right|
\lesssim h_\rho(x,y),
\qquad j,k\in\{1,2\}.
\]
Expanding $\mathcal D_{jk}^r=(\sigma_jB_x+\sigma_kB_y)^r$
therefore yields, in particular, on
$(0,\infty)^2\setminus\bigl((0,1/2]^2\cup[2,\infty)^2\bigr)$,
\begin{equation}\label{eq:complement-kernel-bound}
|(\mathcal D_{jk}^rK_{jk})(x,y)|
\lesssim
h_\rho(x,y)(1+|\log x|+|\log y|)^r,
\qquad r=0,1,2.
\end{equation}

\smallskip
\noindent\emph{Step 4: compactness of $K$ and $i[\mathbf A,K]$
and boundedness of $[i\mathbf A,[i\mathbf A,K]]$.}
Fix $r\in\{0,1\}$. By \eqref{eq:diagonal-vanishing-estimates} and Schur's test, the
norms of the integral operators with kernels
\[
 \ind_{(0,\delta]}(x)\mathcal D_{jj}^rK_{jj}(x,y)
       \ind_{(0,\delta]}(y),
 \qquad
 \ind_{[R,\infty)}(x)\mathcal D_{jj}^rK_{jj}(x,y)
       \ind_{[R,\infty)}(y)
\]
tend to zero as $\delta\downarrow0$ and $R\to\infty$,
respectively. On $(0,1/2]^2\setminus(0,\delta]^2$ and
$[2,\infty)^2\setminus[R,\infty)^2$,
\[
\begin{aligned}
 &\iint_{\substack{0<x,y\le1/2\\\max(x,y)\ge\delta}}
 h_\rho(x,y)^2L(\theta)^{2r}\dd x\dd y
 \le2\int_\delta^{1/2}\int_0^{1/2}
 h_\rho(x,y)^2L(\theta)^{2r}\dd y\dd x<\infty,\\
 &\iint_{\substack{x,y\ge2\\\min(x,y)\le R}}
 h_\rho(x,y)^2L(\theta)^{2r}\dd x\dd y
 \le2\int_2^R\int_2^\infty
 h_\rho(x,y)^2L(\theta)^{2r}\dd y\dd x<\infty
\end{aligned}
\]
for fixed $0<\delta<1/2$ and $R>2$. In the first integral, with
$x\in[\delta,1/2]$, the integrand is at most
$Cy^{2\rho}(1+|\log y|)^{2r}$ near $y=0$.
In the second, with $x\in[2,R]$, it is at most
$Cy^{-2-2\rho}(1+\log y)^{2r}$ for large $y$.
Both functions are integrable because $\rho>-1/2$.
Therefore, the integral operators on $L^2(\mathbb R_+)$ with kernels
\[
  \mathbf 1_{(0,1/2]^2}\mathcal D_{jj}^rK_{jj},
  \qquad
  \mathbf 1_{[2,\infty)^2}\mathcal D_{jj}^rK_{jj},
  \qquad j=1,2,\quad r=0,1,
\]
are limits of Hilbert--Schmidt operators in the operator norm
and are therefore compact.

Outside these squares, $h_\rho$ is bounded on $[1/2,2]^2$,
and
\[
\begin{array}{ll}
 h_\rho(x,y)\lesssim x^\rho
   & (0<x\le1/2,\ 1/2\le y\le2),\\
 h_\rho(x,y)\lesssim x^\rho y^{-1-\rho}
   & (0<x\le1/2,\ y\ge2),\\
 h_\rho(x,y)\lesssim y^{-1-\rho}
   & (1/2\le x\le2,\ y\ge2).
\end{array}
\]
The same bounds hold with $x$ and $y$ exchanged. Since
\[
 \int_0^{1/2}x^{2\rho}(1+|\log x|)^4\dd x<\infty,
 \qquad
 \int_2^\infty x^{-2-2\rho}(1+\log x)^4\dd x<\infty,
\]
we obtain
\[
 \iint_{\mathbb R_+^2\setminus((0,1/2]^2\cup[2,\infty)^2)}
 h_\rho(x,y)^2(1+|\log x|+|\log y|)^4\dd x\dd y<\infty.
\]
By \eqref{eq:complement-kernel-bound}, the integral
operators with kernels $\mathcal D_{jk}^rK_{jk}$ outside the two
squares are Hilbert--Schmidt for $r=0,1,2$. Combining these operators with
the two diagonal regions proves
\begin{equation}\label{eq:compact-first-commutator}
 K\in\mathcal K(\mathcal H),\qquad
 i[\mathbf A,K]\in\mathcal K(\mathcal H).
\end{equation}
For $r=2$, \eqref{eq:diagonal-global-bound} and Schur's test
bound the operators on $(0,1/2]^2$ and $[2,\infty)^2$.
By the preceding integral estimate, the kernels
\[
  \mathbf 1_{\mathbb R_+^2\setminus
  \bigl((0,1/2]^2\cup[2,\infty)^2\bigr)}
  \mathcal D_{jk}^rK_{jk},
  \qquad j,k\in\{1,2\},\quad r=0,1,2,
\]
belong to $L^2(\mathbb R_+^2)$. The corresponding integral
operators on $L^2(\mathbb R_+)$ are therefore Hilbert--Schmidt. Thus
$[i\mathbf A,[i\mathbf A,K]]$ is bounded. Since the first two
commutators of $H_0$ with $\mathbf A$ are bounded, the
quadratic-form commutators
\[
 [\mathbf A,T]\qquad\text{and}\qquad
 [\mathbf A,[\mathbf A,T]]
\]
extend to bounded operators.

\smallskip
\noindent\emph{Step 5: the Mourre estimate and spectral conclusions.}
Assume first that \((q_0,q_\infty)\neq(0,0)\).  Fix a non-empty
bounded open interval \(\Delta\) satisfying the condition in the
theorem.  Choose a bounded open interval \(\Delta_1\) such that
\[
  \overline\Delta\subset\Delta_1,
  \qquad
  \overline{\Delta_1}
  \subset
  \bigl(q_0(0,\pi_\rho)\cup q_\infty(0,\pi_\rho)\bigr)
  \setminus\{q_0\pi_\rho,q_\infty\pi_\rho\}.
\]
Put
\[
  E=E_\Delta(T),
  \qquad
  E_1=E_{\Delta_1}(H_0).
\]
Step 2 gives a constant \(c_1>0\) such that
\begin{equation}\label{eq:model-Mourre-expanded}
  E_1i[\mathbf A,H_0]E_1\geq c_1E_1.
\end{equation}

Let
\[
  \varkappa:\mathcal B(\mathcal H)
  \longrightarrow
  \mathcal B(\mathcal H)/\mathcal K(\mathcal H)
\]
be the quotient map.  Choose \(\varphi\in C_c(\mathbb R)\) such that
\[
  \varphi=1
  \quad\text{on a neighborhood of }\overline\Delta,
  \qquad
  \operatorname{supp}\varphi\subset\Delta_1.
\]
The compactness of \(T-H_0\) gives
\[
  \varkappa(\varphi(T))=\varkappa(\varphi(H_0)).
\]
Since
\[
  E\varphi(T)=E,
  \qquad
  \varphi(H_0)(1-E_1)=0,
\]
we have
\[
\begin{aligned}
  \varkappa(E(1-E_1))
  &=\varkappa(E\varphi(T)(1-E_1))\\
  &=\varkappa(E(\varphi(T)-\varphi(H_0))(1-E_1))
  =0.
\end{aligned}
\]
Thus
\[
  \varkappa(E)=\varkappa(EE_1)=\varkappa(E_1E).
\]
Furthermore, \(i[\mathbf A,T]-i[\mathbf A,H_0]\) is compact by
\eqref{eq:compact-first-commutator}, and hence
\[
\begin{aligned}
  \varkappa(Ei[\mathbf A,T]E)
  &=\varkappa(Ei[\mathbf A,H_0]E)\\
  &=\varkappa(EE_1i[\mathbf A,H_0]E_1E).
\end{aligned}
\]
Applying the positive map \(\varkappa\) to
\eqref{eq:model-Mourre-expanded} and multiplying by
\(\varkappa(E)\) on both sides gives
\[
  \varkappa(Ei[\mathbf A,T]E)
  \geq c_1\varkappa(E).
\]
Therefore
\[
  \varkappa(Ei[\mathbf A,T]E-c_1E)\geq0.
\]
By the lifting property for positive elements of a quotient
\(C^*\)-algebra, there is a compact self-adjoint operator
\(K_\Delta\) such that
\begin{equation}\label{eq:transferred-Mourre}
  E_\Delta(T)i[\mathbf A,T]E_\Delta(T)
  \geq c_1E_\Delta(T)+K_\Delta.
\end{equation}

The compactness of \(K=T-H_0\) gives
\[
  \sigma_{\mathrm{ess}}(T)
  =\sigma_{\mathrm{ess}}(H_0)
  =q_0[0,\pi_\rho]\cup q_\infty[0,\pi_\rho].
\]
 Step~4 verify the
commutator assumptions in Theorem~\ref{thm:Mourre}, and
\eqref{eq:transferred-Mourre} has the form
\eqref{eq:Mourre-estimate} on \(\Delta\).  Thus
\[
  \sigma_{\mathrm{sc}}(T)\cap\Delta=\varnothing,
\]
and the eigenvalues of \(T\) in \(\Delta\) form a finite set and have
finite multiplicity.

Since the preceding argument applies to every interval compactly
contained in
\[
  \bigl(q_0(0,\pi_\rho)\cup q_\infty(0,\pi_\rho)\bigr)
  \setminus\{q_0\pi_\rho,q_\infty\pi_\rho\},
\]
and the spectrum outside $\sigma_{\mathrm{ess}}(T)$ is discrete,
every singular continuous spectral measure of $T$ is supported on
$\{0,q_0\pi_\rho,q_\infty\pi_\rho\}$. Such a measure has no atoms
and must therefore vanish. Hence
\[
  \sigma_{\mathrm{sc}}(T)=\varnothing.
\]

Finally, \(T\) is unitarily equivalent to \(H\oplus JHJ^*\), and
\(JHJ^*\) is unitarily equivalent to \(H\).  All the assertions
therefore follow for \(H\).  If \(q_0=q_\infty=0\), then \(T\) is
compact and the same conclusions are immediate.
\end{proof}

The existence of the finite limits $a_0$, $a_\infty$,
$\omega_0$, and $\omega_\infty$ implies that $K$ is compact,
so no derivative assumptions are needed to determine the essential
spectrum. The bounds on $D^2a$ and $D^2\omega$ are used only
in the Mourre argument. Together with these limits, they imply
\[
  |\log t|\,Df(t)\to0,
  \qquad t\to0 \text{ and } t\to\infty,
  \qquad f=a,\omega.
\]
These estimates give compactness of $i[\mathbf A,K]$
and boundedness of $[i\mathbf A,[i\mathbf A,K]]$.

\begin{proof}[Proof of part~(2) of Theorem~\ref{thm:complete-spectrum-Ta}]
Fix \(\alpha>1\), set
\[
  \rho=\frac{\alpha-1}{2},
  \qquad
  \pi_\rho=\frac{\Gamma(\alpha/2)^2}{\Gamma(\alpha)}=\Lambda_\alpha.
\]

We next verify the hypotheses needed for the singular continuous spectrum.
Apply Theorem~\ref{thm:weighted-Hankel-no-sc} to
\(K_\alpha\).  In the notation preceding that theorem, take
\[
  b(t)=\zeta(1+t)^\alpha-1,
  \qquad w(t)=t^\rho.
\]
The corresponding normalized functions are
\[
  a(t)=t^\alpha\bigl(\zeta(1+t)^\alpha-1\bigr),
  \qquad \omega(t)=1.
\]
The Laurent expansion of \(\zeta\) at \(1\) gives
\[
  a(t)=(t\zeta(1+t))^\alpha-t^\alpha
  =1+O(t)-t^\alpha,
  \qquad t\to0.
\]
Applying \(D=t\partial_t\) once or twice to the same identity shows that,
for \(j=0,1,2\),
\[
  \bigl(D^j(a-1)\bigr)(t)
  =O(t)+O(t^\alpha),
  \qquad t\to0.
\]
At infinity, \(\zeta(1+t)-1=O(2^{-t})\), and its first two
derivatives satisfy the same exponential estimate up to polynomial factors.
Hence, for \(j=0,1,2\),
\[ (D^j a)(t)=O(t^{\alpha+2}2^{-t}), \qquad t\to\infty. \]
Together with continuity on compact subintervals, these estimates give
\[
  |(D^2a)(t)|\leq C\langle\log t\rangle^{-2},
  \qquad t>0.
\]
For \(\omega\equiv1\), the corresponding condition is immediate.  Therefore
\[
  a_0=1, \qquad a_\infty=0,
  \qquad \omega_0=\omega_\infty=1,
\]
so that \(q_0=1\) and \(q_\infty=0\).  Theorem~\ref{thm:weighted-Hankel-no-sc}
now yields
\[ \sigma_{\mathrm{ess}}(K_\alpha) =[0,\Lambda_\alpha], \qquad \sigma_{\mathrm{sc}}(K_\alpha)=\varnothing. \]
Removing \(\ker K_\alpha\) does not alter either the absolutely
continuous or the singular continuous part.  The unitary equivalence
in Theorem~\ref{thm:Ta-unitary-Ka} therefore gives
\[
  \sigma_{\mathrm{ess}}(T_\alpha)=[0,\Lambda_\alpha],
  \qquad
  \sigma_{\mathrm{sc}}(T_\alpha)=\varnothing.
\]

It remains to determine the absolutely continuous spectrum.  Apply
\cite[Theorem~1.3]{FedelePushnitski2017} to \(K_\alpha\) with $\epsilon=1$,
\[
  \sigma_{\mathrm{ac}}(K_\alpha)
  =[0,\pi_\rho]
  =[0,\Lambda_\alpha],
\]
with multiplicity one.  Theorem~\ref{thm:Ta-unitary-Ka} transfers this conclusion to
\(T_\alpha\).
\end{proof}

\section{Absence of eigenvalues in the essential spectrum}

The proof of part~(3) of Theorem~\ref{thm:complete-spectrum-Ta} is based on the Mellin transform. The eigenvalue relation will be converted into a meromorphic continuation formula for the Mellin transform of an eigenfunction.
\subsection{Derivation of the two alternatives}

The next two lemmas contain the analytic estimates needed for the continuation argument. The first one is designed so that differentiation acts only on the function of \(x+y\) in the kernel. Indeed, for every differentiable function \(F\),
\[
  \left(\frac{\mathrm d}{\mathrm d x}-\frac{\rho}{x}\right)
  \bigl((xy)^\rho F(x+y)\bigr)
  =(xy)^\rho F'(x+y).
\]
Thus, when this operator acts on the \(x\)-variable of the kernel, only the factor \(F(x+y)\) is differentiated.

\begin{lemma}\label{lem:embedded-remainder-regularity}

Let $\alpha>1$ and $\rho=(\alpha-1)/2$. Put
\[
r_\alpha(t)=\zeta(1+t)^\alpha-1-t^{-\alpha},
\]
let $R_\alpha$ be the integral operator with kernel $(xy)^\rho r_\alpha(x+y)$, and define
\[
P_\rho=\frac{\mathrm d}{\mathrm d x}-\frac{\rho}{x}.
\]
For every integer \(N\ge1\) and every \(0\le\delta<\alpha/2\), if
\(u\in L^2(\mathbb R_+)\) satisfies
\[
  x^{-N+1-\delta}u\in L^2(\mathbb R_+),
\]
then
\[
  x^{-\delta}P_\rho^N R_\alpha u\in L^2(\mathbb R_+).
\]
Moreover, $x^NP_\rho^N R_\alpha$ is bounded on $L^2(\mathbb R_+)$.

\end{lemma}

\smallskip
\noindent\textit{Proof.}
Since
\[
t\zeta(1+t)=1+O(t),\qquad t\downarrow0,
\]
the branch of $(t\zeta(1+t))^\alpha$ which equals $1$ at $t=0$ is analytic near $0$, and therefore, for every integer $N\ge0$,
\[
|r_\alpha^{(N)}(t)|\lesssim t^{-\alpha-N+1},\qquad 0<t\le1.
\]
For $t\ge1$, the derivatives of $\zeta(1+t)^\alpha-1$ decay exponentially while $(t^{-\alpha})^{(N)}=O(t^{-\alpha-N})$, so
\[
|r_\alpha^{(N)}(t)|\lesssim t^{-\alpha-N},\qquad t\ge1.
\]

Also put
\[
A=x\frac{\mathrm d}{\mathrm d x}-\rho.
\]
Then
\[
P_\rho^N[x^\rho F(x)]=x^\rho F^{(N)}(x),
\qquad
x^NP_\rho^N=\prod_{j=0}^{N-1}(A-j).
\]
Write
\[
u(y)=y^{N-1+\delta}v(y),\qquad v\in L^2(\mathbb R_+).
\]
we have,
\[
x^{-\delta}P_\rho^N R_\alpha u(x)
=
\int_0^\infty x^{\rho-\delta}y^{\rho+N-1+\delta}r_\alpha^{(N)}(x+y)v(y)\dd y.
\]
Thus it remains to prove that the kernel
\[
x^{\rho-\delta}y^{\rho+N-1+\delta}r_\alpha^{(N)}(x+y)
\]
defines a bounded operator on $L^2(\mathbb R_+)$. Write it as the sum of its restrictions to $x+y\le1$ and $x+y>1$. For the first term,
\[
\ind_{\{x+y\le1\}}x^{\rho-\delta}y^{\rho+N-1+\delta}|r_\alpha^{(N)}(x+y)|
\lesssim
\frac{x^{\rho-\delta}y^{\rho+N-1+\delta}}{(x+y)^{\alpha+N-1}}.
\]
 With $w(x)=x^{-1/2}$,
\[
\frac1{w(x)}\int_0^\infty
\frac{x^{\rho-\delta}y^{\rho+N-1+\delta}}{(x+y)^{\alpha+N-1}}w(y)\dd y
=
B\!\left(\rho+N-\frac12+\delta,\frac\alpha2-\delta\right)<\infty,
\]
and similarly
\[
\frac1{w(y)}\int_0^\infty
\frac{x^{\rho-\delta}y^{\rho+N-1+\delta}}{(x+y)^{\alpha+N-1}}w(x)\dd x
=
B\!\left(\frac\alpha2-\delta,\rho+N-\frac12+\delta\right)<\infty.
\]
Thus the first term defines a bounded operator on $L^2(\mathbb R_+)$.
For the second term,
\[
\ind_{\{x+y>1\}}x^{\rho-\delta}y^{\rho+N-1+\delta}|r_\alpha^{(N)}(x+y)|
\lesssim
\ind_{\{x+y>1\}}\frac{x^{\rho-\delta}y^{\rho+N-1+\delta}}{(x+y)^{\alpha+N}}.
\]
The square of its Hilbert--Schmidt norm is bounded by
\[
\iint_{x+y>1}
\frac{x^{2\rho-2\delta}y^{2\rho+2N-2+2\delta}}{(x+y)^{2\alpha+2N}}\dd x\dd y.
\]
With $t=x+y$ and $z=x/(x+y)$, this becomes
\[
\int_1^\infty t^{-3}\dd t
\int_0^1 z^{2\rho-2\delta}(1-z)^{2\rho+2N-2+2\delta}\dd z<\infty,
\]
because
\[
2\rho-2\delta>-1\iff\delta<\frac\alpha2,
\]
while $2\rho+2N-2+2\delta>-1$ is automatic for $N\ge1$. Hence the second term is Hilbert--Schmidt, and the first assertion follows. Taking $\delta=0$ in the first assertion,
\[
P_\rho^N R_\alpha u\in L^2(\mathbb R_+)
\qquad\text{whenever}\qquad
x^{-N+1}u\in L^2(\mathbb R_+).
\]
Moreover, the estimates for $r_\alpha^{(N)}$ give, for all $t>0$,
\[
|r_\alpha^{(N)}(t)|\lesssim t^{-\alpha-N}.
\]
Hence the kernel of $x^NP_\rho^N R_\alpha$ satisfies
\[
|x^{N+\rho}y^\rho r_\alpha^{(N)}(x+y)|
\lesssim
\frac{x^{N+\rho}y^\rho}{(x+y)^{N+\alpha}}.
\]
The two Schur integrals for this homogeneous kernel are finite, and therefore $x^NP_\rho^N R_\alpha$ is bounded on $L^2(\mathbb R_+)$.
\hfill$\square$

\medskip

Whenever the integral converges, write
\[
  (\mathfrak Mq)(z)=\int_0^\infty q(x)x^{z-1}\,\dd x.
\]

\begin{lemma}\label{lem:embedded-mellin-strip}
Let \(a,b>0\).
If
\[
  x^{-a}q,\ x^bq\in L^2(\mathbb R_+),
\]
then \(\mathfrak Mq\) is analytic in
\[
  \frac12-a<\Re z<\frac12+b
\]
and
\[
  \sup_{-a<\sigma<b}
  \int_{\mathbb R}
  \left|
    (\mathfrak Mq)\left(\frac12+\sigma+i\tau\right)
  \right|^2\dd\tau<\infty.
\]

Conversely, let \(Q\) be analytic in this strip and suppose that
\[
  \sup_{\sigma_1\leq\sigma\leq\sigma_2}
  \int_{\mathbb R}
  \left|
    Q\left(\frac12+\sigma+i\tau\right)
  \right|^2\dd\tau<\infty
\]
whenever
\[
  -a<\sigma_1<\sigma_2<b.
\]
If
\[
  q=\mathcal M^{-1}
  \left[
    \frac1{\sqrt{2\pi}}
    Q\left(\frac12+i\,\cdot\right)
  \right],
\]
then
\[
  x^\sigma q\in L^2(\mathbb R_+),
  \qquad -a<\sigma<b,
\]
and \(Q=\mathfrak Mq\) throughout the strip.
\end{lemma}

\smallskip
\noindent\textit{Proof.}
Recall that 
\[
  \mathfrak Mq\left(\frac12+\sigma+i\tau\right)
  =\sqrt{2\pi}\,\mathcal M(x^\sigma q)(\tau)
\]
If $x^{-a}q,x^{b}q\in L^2(\mathbb R_+)$ for some $a,b>0$, then the Cauchy--Schwarz inequality, applied near $0$ and near infinity, shows that $\mathfrak M q$ is analytic in
\[
\frac12-a<\Re z<\frac12+b.
\]
For $-a<\sigma<b$,
\[
\mathfrak M q\left(\frac12+\sigma+i\tau\right)
=\sqrt{2\pi}\,\mathcal M(x^\sigma q)(\tau),
\]
and hence
\[
\int_{\mathbb R}\left|\mathfrak M q\left(\frac12+\sigma+i\tau\right)\right|^2\dd\tau
=2\pi\int_0^\infty |q(x)|^2x^{2\sigma}\dd x.
\]
Moreover,
\[
\begin{aligned}
&\sup_{-a<\sigma<b}
\int_{\mathbb R}\left|\mathfrak M q\left(\frac12+\sigma+i\tau\right)\right|^2\dd\tau\\
&\qquad\le2\pi\left(
\int_0^1|q(x)|^2x^{-2a}\dd x+
\int_1^\infty|q(x)|^2x^{2b}\dd x
\right)<\infty.
\end{aligned}
\]

We also need the converse implication. Suppose $Q$ is analytic in
\[
\frac12-a<\Re z<\frac12+b
\]
and, for every $-a<\sigma_1<\sigma_2<b$,
\[
\sup_{\sigma_1\le\sigma\le\sigma_2}
\int_{\mathbb R}\left|Q\left(\frac12+\sigma+i\tau\right)\right|^2\dd\tau<\infty.
\]
Define
\[
q=\mathcal M^{-1}\!\left[\frac1{\sqrt{2\pi}}Q\left(\frac12+i\,\cdot\right)\right].
\]
Then
\[
Q\left(\frac12+i\tau\right)=\sqrt{2\pi}\,(\mathcal M q)(\tau)
\]
for almost every $\tau$. Fix $\sigma\in(-a,b)$ and $\varphi\in C_c^\infty(\mathbb R_+)$, and put
\[
\Phi_\sigma(z)=\int_0^\infty\varphi(x)x^{\sigma-z}\dd x.
\]
Writing $x=e^t$ and integrating by parts in $t$ shows that, on every fixed vertical strip and for every integer $m\ge0$,
\[
|\Phi_\sigma(c+i\tau)|\lesssim_m(1+|\tau|)^{-m}.
\]
Assume first $\sigma>0$. The uniform $L^2$ bound implies
\[
\int_{\mathbb R}\int_0^\sigma
\left|Q\left(\frac12+s+i\tau\right)\right|^2ds\,\dd\tau<\infty.
\]
Hence there is a sequence $T_k\to\infty$ such that
\[
\int_0^\sigma\left(
\left|Q\left(\frac12+s+iT_k\right)\right|^2+
\left|Q\left(\frac12+s-iT_k\right)\right|^2
\right)ds\to0.
\]
Cauchy's integral theorem applied to $Q(z)\Phi_\sigma(z)$ on the rectangle with vertical sides $\Re z=1/2$ and $\Re z=1/2+\sigma$, together with the Cauchy--Schwarz inequality and the rapid decay of $\Phi_\sigma$, gives
\[
\int_0^\infty x^\sigma q(x)\varphi(x)\dd x
=\frac1{\sqrt{2\pi}}\int_{\mathbb R}
Q\left(\frac12+\sigma+i\tau\right)(\mathcal M\varphi)(-\tau)\,\dd\tau.
\]
For $\sigma<0$ the same formula follows by reversing the rectangle. Consequently,
\[
\left|\int_0^\infty x^\sigma q(x)\varphi(x)\dd x\right|
\le\frac1{\sqrt{2\pi}}
\left\|Q\left(\frac12+\sigma+i\,\cdot\right)\right\|_2\|\varphi\|_2.
\]
By the density of $C_c^\infty(\mathbb R_+)$ in $L^2(\mathbb R_+)$,
\[
x^\sigma q\in L^2(\mathbb R_+),\qquad -a<\sigma<b.
\]
The Mellin transform satisfies
\[
\int_0^\infty x^\sigma q(x)\varphi(x)\dd x
=\int_{\mathbb R}(\mathcal M (x^\sigma q))(\tau)(\mathcal M\varphi)(-\tau)\,\dd\tau.
\]
Comparing this with the preceding integral identity, and using the density of
\[
\{(\mathcal M\varphi)(-\,\cdot):\varphi\in C_c^\infty(\mathbb R_+)\}
\]
in $L^2(\mathbb R)$, gives
\[
\mathcal M(x^\sigma q)(\tau)=\frac1{\sqrt{2\pi}}Q\left(\frac12+\sigma+i\tau\right)
\]
for almost every $\tau$. As $\sigma$ ranges over $(-a,b)$, this proves $Q(z)=\mathfrak M q(z)$ throughout the strip.
\hfill$\square$

\begin{proof}[Proof of part~(3) of Theorem~\ref{thm:complete-spectrum-Ta}]
Put
\[
\rho=\frac{\alpha-1}{2}.
\]
Assume, to obtain a contradiction, that
\begin{equation}\label{eq:1}
K_\alpha f=\lambda f,
\qquad 0<\lambda\le\Lambda_\alpha,
\qquad 0\ne f\in L^2(\mathbb R_+).
\end{equation}
Since the kernel is real-valued, one of the real and imaginary parts of
$f$ is a non-zero eigenfunction for the same eigenvalue, so we take $f$
real-valued. We use the continuous representative
\[
f(x)=\lambda^{-1}\int_0^\infty (xy)^\rho\phi_\alpha(x+y)f(y)\dd y,
\qquad
\phi_\alpha(t)=\zeta(1+t)^\alpha-1.
\]
For $t\ge1$,
\[
0<\phi_\alpha(t)\lesssim2^{-t}.
\]
Hence, by the Cauchy--Schwarz inequality,
\begin{equation}\label{eq:2}
|f(x)|\lesssim x^\rho2^{-x},
\qquad x\ge1.
\end{equation}

Recall
\[
r_\alpha(t)=\phi_\alpha(t)-t^{-\alpha},
\qquad
R_\alpha=K_\alpha-C_\rho.
\]
Assume for the moment that $x^{-N+1}f\in L^2(\mathbb R_+)$ for some
integer $N\ge1$, and put
\[
v_k=x^kP_\rho^kR_\alpha f,
\qquad0\le k\le N.
\]
By Lemma~\ref{lem:embedded-remainder-regularity} with $\delta=0$,
together with \eqref{eq:2} on $(1,\infty)$,
\[
v_k\in L^2(\mathbb R_+),
\qquad0\le k\le N.
\]
Moreover,
\[
v_{k+1}=(A-k)v_k,
\qquad
A=x\frac{\mathrm d}{\mathrm d x}-\rho,
\]
so $Av_k=v_{k+1}+kv_k\in L^2$. Choose $\chi\in C_c^\infty(\mathbb R)$
with $\chi=1$ on $[-1,1]$ and set
\[
\chi_R(x)=\chi\!\left(\frac{\log x}{R}\right).
\]
Then
\[
\chi_Rv_k\to v_k,
\qquad
A(\chi_Rv_k)\to Av_k
\quad\text{in }L^2(\mathbb R_+).
\]
Integration by parts, with no boundary terms because $\chi_Rv_k$
has compact support in $(0,\infty)$, gives
\[
\mathcal M(A(\chi_Rv_k))(\xi)
=-\left(\frac12+\rho+i\xi\right)\mathcal M(\chi_Rv_k)(\xi).
\]
Passing to the $L^2$ limit and iterating $v_{k+1}=(A-k)v_k$ yields
\begin{equation}\label{eq:4}
\mathcal M(x^NP_\rho^NR_\alpha f)(\xi)
=(-1)^N\prod_{j=0}^{N-1}
\left(\frac12+\rho+j+i\xi\right)\mathcal M(R_\alpha f)(\xi)
\end{equation}
for almost every $\xi\in\mathbb R$.

Set
\[
q_N=x^NP_\rho^NR_\alpha f,
\qquad
\Psi_N=\mathfrak Mq_N.
\]
By Lemma~\ref{lem:embedded-remainder-regularity},
\[
q_N\in L^2(\mathbb R_+),
\qquad
x^{-N}q_N=P_\rho^NR_\alpha f\in L^2(\mathbb R_+).
\]
For $x\ge1$, the estimate
$|r_\alpha^{(N)}(t)|\lesssim t^{-\alpha-N}$ gives
\[
|q_N(x)|
\lesssim x^{N+\rho}\int_0^\infty
y^\rho(x+y)^{-\alpha-N}|f(y)|\dd y
\le x^{\rho-\alpha}\int_0^\infty y^\rho|f(y)|\dd y.
\]
The last integral is finite by the Cauchy--Schwarz inequality on
$(0,1)$ and \eqref{eq:2} on $(1,\infty)$. Using $\rho-\alpha=-1-\rho$, we obtain
\[
|q_N(x)|\lesssim x^{-1-\rho},
\qquad x\ge1,
\]
and hence
\[
\int_1^\infty x|q_N(x)|^2\dd x
\lesssim\int_1^\infty x^{-1-2\rho}\dd x
=\int_1^\infty x^{-\alpha}\dd x<\infty.
\]
Thus $x^{1/2}q_N\in L^2(\mathbb R_+)$, and
Lemma~\ref{lem:embedded-mellin-strip} shows that $\Psi_N$ is analytic in
\[
\frac12-N<\Re z<1.
\]

By \eqref{eq:2},
\[
F(z):=(\mathfrak Mf)(z)=\int_0^\infty f(x)x^{z-1}\dd x
\]
is well defined and analytic for $\Re z>1/2$. Indeed, if
$\sigma>1/2$,
\[
\int_0^1|f(x)|x^{\sigma-1}\dd x
\le\|f\|_{L^2(0,1)}
\left(\int_0^1x^{2\sigma-2}\dd x\right)^{1/2}<\infty,
\]
while the integral over $(1,\infty)$ converges by \eqref{eq:2}.
Moreover, for $s>0$,
\[
F\left(\frac12+s+i\tau\right)
=\sqrt{2\pi}\,\mathcal M(x^sf)(\tau).
\]
As $s\downarrow0$, $x^sf\to f$ in $L^2(\mathbb R_+)$, and therefore
\[
F\left(\frac12+s+i\,\cdot\right)
\longrightarrow\sqrt{2\pi}\,\mathcal Mf
\quad\text{in }L^2(\mathbb R).
\]
We use the same symbol $F$ for this $L^2$ boundary value on
$\Re z=1/2$. Put
\[
\mathfrak m_\rho(z)
=\frac{\Gamma(\rho+z)\Gamma(1+\rho-z)}{\Gamma(2\rho+1)},
\]
hence
\[
\mathfrak m_\rho\left(\frac12+i\xi\right)
=m_\rho(\xi),
\qquad \xi\in\mathbb R.
\]
The Mellin representation of $C_\rho$ and the eigenvalue equation give
\[
\mathcal M(R_\alpha f)(\xi)
=
\left(\lambda-\mathfrak m_\rho\left(\frac12+i\xi\right)\right)
\mathcal Mf(\xi)
\]
for almost every $\xi\in\mathbb R$.
Setting $z=1/2+i\xi$ and using \eqref{eq:4}, we obtain
\[
\Psi_N(z)=(-1)^N
\left(\prod_{j=0}^{N-1}(z+\rho+j)\right)
(\lambda-\mathfrak m_\rho(z))F(z)
\]
for almost every $z$ on the line $\Re z=1/2$.
For this $N$, set
\begin{equation}\label{eq:5}
\widetilde F_N(z)=
\frac{(-1)^{N+1}\Psi_N(z)}
{\left(\prod_{j=0}^{N-1}(z+\rho+j)\right)
(\mathfrak m_\rho(z)-\lambda)}.
\end{equation}
The factor $\Gamma(\rho+z)$ has simple poles at
$z=-\rho-j$, $j=0,1,\ldots$, and the polynomial in \eqref{eq:5}
cancels the first $N$ of them. The next pole is
\[
-\rho-N=\frac12-\frac\alpha2-N<\frac12-N,
\]
while the first pole of $\Gamma(1+\rho-z)$ is $1+\rho>1$. Hence, in
$1/2-N<\Re z<1$, the only possible non-removable singularities of
$\widetilde F_N$ occur at solutions of
\begin{equation}\label{eq:6}
\mathfrak m_\rho(z)-\lambda=0.
\end{equation}
There is no non-removable singularity on $\Re z=1/2$. Indeed, a pole of order $k\ge1$ at $z_0=1/2+i\tau_0$ would give
\[
|\widetilde F_N(1/2+i\tau)|
\asymp|\tau-\tau_0|^{-k},
\qquad 0<|\tau-\tau_0|<\delta,
\]
for some $\delta>0$, and hence
\[
\int_{\tau_0-\delta}^{\tau_0+\delta}
|\widetilde F_N(1/2+i\tau)|^2\dd\tau=\infty.
\]
This contradicts the almost-everywhere equality
$\widetilde F_N(1/2+i\tau)=F(1/2+i\tau)$ and
\[
\int_{\mathbb R}|F(1/2+i\tau)|^2\dd\tau
=2\pi\|f\|_2^2<\infty.
\]
Thus all singularities of $\widetilde F_N$ on $\Re z=1/2$
are removable.

Fix $X>0$ and $0\le Y<1/2$. Uniformly for
\[
\frac12-X\le\sigma\le\frac12+Y,
\]
Stirling's formula gives
\[
|\mathfrak m_\rho(\sigma+i\tau)|
\lesssim_{X,Y}(1+|\tau|)^{\alpha-1}e^{-\pi|\tau|}
\]
for sufficiently large $|\tau|$. Hence every solution of \eqref{eq:6}
in this strip satisfies $|\Im z|\le T$ for some
$T=T(X,Y,\lambda)$. Choosing an integer $N>X$, the function
\[
\left(\prod_{j=0}^{N-1}(z+\rho+j)\right)
(\mathfrak m_\rho(z)-\lambda)
\]
is holomorphic on a neighborhood of the resulting compact rectangle
and is not identically zero. Therefore \eqref{eq:6} has only finitely
many solutions in every strip of finite depth.

For any integers $M,N\geq1$ satisfying
$x^{-M+1}f,x^{-N+1}f\in L^2(\mathbb R_+)$, the preceding
identities give
\[
\widetilde F_M\left(\frac12+i\tau\right)
=
\widetilde F_N\left(\frac12+i\tau\right)
=
\sqrt{2\pi}\,\mathcal Mf(\tau)
\]
for almost every $\tau\in\mathbb R$.
Since the singularities on $\Re z=1/2$ are removable,
the identity theorem gives
\[
\widetilde F_M(z)=\widetilde F_N(z),
\qquad \frac12-\min\{M,N\}<\Re z<1.
\]
We denote this common meromorphic function by $\widetilde F$.

Suppose that $x^{-N+1}f\in L^2(\mathbb R_+)$ and that
$\widetilde F$ has no poles in $1/2-N<\Re z<1/2$.
The removability on $\Re z=1/2$ and the finiteness of the
solutions of \eqref{eq:6} give $0<\eta<1/2$ such that
$\widetilde F$ is analytic in
\[
\frac12-N<\Re z<\frac12+\eta.
\]
Fix
\[
\frac12-N<a<b<\frac12+\eta.
\]
For sufficiently large $|\tau|$, uniformly in $\sigma\in[a,b]$,
\[
|\mathfrak m_\rho(\sigma+i\tau)-\lambda|
\geq\frac{\lambda}{2},
\qquad
\prod_{j=0}^{N-1}|\sigma+\rho+j+i\tau|
\gtrsim(1+|\tau|)^N.
\]
Formula~\eqref{eq:5} therefore gives
\[
|\widetilde F(\sigma+i\tau)|
\lesssim(1+|\tau|)^{-N}|\Psi_N(\sigma+i\tau)|.
\]
Applying Lemma~\ref{lem:embedded-mellin-strip} with
$x^{-N}q_N,x^{1/2}q_N\in L^2(\mathbb R_+)$ gives
\[
\sup_{\sigma\in[a,b]}
\int_{\mathbb R}|\Psi_N(\sigma+i\tau)|^2\dd\tau<\infty.
\]
For each fixed $T>0$, analyticity gives
\[
\sup_{\substack{\sigma\in[a,b]\\|\tau|\leq T}}
|\widetilde F(\sigma+i\tau)|<\infty.
\]
Combining these bounds, we obtain
\[
\sup_{\sigma\in[a,b]}
\int_{\mathbb R}|\widetilde F(\sigma+i\tau)|^2\dd\tau<\infty.
\]
Together with the almost-everywhere equality
\[
\widetilde F\left(\frac12+i\tau\right)
=\sqrt{2\pi}\,\mathcal Mf(\tau),
\]
Lemma~\ref{lem:embedded-mellin-strip} yields
\[
x^{-d}f\in L^2(\mathbb R_+),
\qquad 0\leq d<N.
\]
Taking $d=N-1/2$ and applying
Lemma~\ref{lem:embedded-remainder-regularity} with $\delta=1/2$,
we obtain
\[
x^{-1/2}P_\rho^NR_\alpha f\in L^2(\mathbb R_+),
\qquad
x^{-N-1/2}q_N\in L^2(\mathbb R_+).
\]
Together with $x^{1/2}q_N\in L^2(\mathbb R_+)$,
Lemma~\ref{lem:embedded-mellin-strip} shows that $\Psi_N$
is analytic in $-N<\Re z<1$ and satisfies
\[
\sup_{\sigma\in[a,b]}
\int_{\mathbb R}|\Psi_N(\sigma+i\tau)|^2\dd\tau<\infty,
\qquad -N<a<b<1.
\]
Formula~\eqref{eq:5} extends $\widetilde F$ meromorphically
to $-N<\Re z<1$.
If $\widetilde F$ has no pole on $\Re z=1/2-N$,
the finiteness of the solutions of \eqref{eq:6} allows us
to choose $0<\varepsilon<1/2$ such that $\widetilde F$
is analytic in
\[
\frac12-N-\varepsilon<\Re z<\frac12+\eta.
\]
The same argument as above gives
\[
\sup_{\sigma\in[a,b]}
\int_{\mathbb R}|\widetilde F(\sigma+i\tau)|^2\dd\tau<\infty,
\qquad
\frac12-N-\varepsilon<a<b<\frac12+\eta.
\]
Applying Lemma~\ref{lem:embedded-mellin-strip} on this strip
and taking $\Re z=1/2-N$, we conclude that
\[
x^{-N}f\in L^2(\mathbb R_+).
\]

Starting from $N=1$, exactly one of the following alternatives occurs:
\begin{equation}\label{eq:7}
\begin{aligned}
&\text{(i)}\quad
\widetilde F\text{ has no non-removable pole in }\Re z<\frac12,
\quad x^{-N}f\in L^2(\mathbb R_+)\text{ for every }N\ge0;\\
&\text{(ii)}\quad
\widetilde F\text{ has a non-removable pole at a solution of
\eqref{eq:6} with }\Re z<\frac12.
\end{aligned}
\end{equation}
\subsection{Exclusion of alternative (i)}
Assume~(i) and put
\[
g(x)=x^{-\rho}f(x).
\]
Then \eqref{eq:1} becomes
\begin{equation}\label{eq:alt-i-equation}
\lambda g(x)=\int_0^\infty
\phi_\alpha(x+y)g(y)y^{\alpha-1}\dd y.
\end{equation}
For every integer $k\ge0$, choose $N>\rho+k+1/2$. we have
\[
\int_0^1|g(y)|y^{-k-1}\dd y
=\int_0^1|y^{-N}f(y)|y^{N-\rho-k-1}\dd y<\infty,
\]
and \eqref{eq:2} gives the same integrability at infinity. Hence
\begin{equation}\label{eq:alt-i-moments}
\int_0^\infty|g(y)|y^{-k-1}\dd y<\infty,
\qquad k=0,1,2,\ldots .
\end{equation}
Since
$|\phi_\alpha^{(k)}(t)|\lesssim_k t^{-\alpha-k}$ for $t>0$,
differentiation under the integral in \eqref{eq:alt-i-equation} is
valid for every $k$ and $x\ge0$. Thus $g\in C^\infty[0,\infty)$.
If $g^{(k)}(0)\ne0$ for the first such $k$, then
$|g(x)|\ge cx^k$ near the origin, and
$x^{-N}f=x^{\rho-N}g\notin L^2$ for $N>\rho+k+1/2$. Therefore
\begin{equation}\label{eq:alt-i-flatness}
g^{(k)}(0)=0,
\qquad k=0,1,2,\ldots .
\end{equation}

Assume first $\alpha\notin\N$ and set
\[
\beta=\lfloor\alpha\rfloor+1-\alpha\in(0,1).
\]
Since $z\zeta(1+z)$ extends holomorphically to $z=0$ with value $1$,
\[
(z\zeta(1+z))^\alpha
=\sum_{j=0}^{\lfloor\alpha\rfloor-1}a_jz^j
+z^{\lfloor\alpha\rfloor}B(z),
\qquad a_0=1,
\]
with $B$ holomorphic near $0$. Hence
\[
\phi_\alpha(t)
=\sum_{j=0}^{\lfloor\alpha\rfloor-1}a_jt^{j-\alpha}+R(t),
\qquad t>0,
\]
where, near $0$,
\[
R(t)=-1+t^{\beta-1}B(t).
\]
Substituting this decomposition into \eqref{eq:alt-i-equation},
\[
\lambda g(x)
=\sum_{j=0}^{\lfloor\alpha\rfloor-1}a_j
\int_0^\infty(x+y)^{j-\alpha}g(y)y^{\alpha-1}\dd y
+\int_0^\infty R(x+y)g(y)y^{\alpha-1}\dd y.
\]
Define
\[
(T_jg)(x)=\int_0^\infty
(x+y)^{j-\alpha}g(y)y^{\alpha-1}\dd y,
\qquad0\le j\le\lfloor\alpha\rfloor-1,
\]
and
\[
(T_Rg)(x)=\int_0^\infty R(x+y)g(y)y^{\alpha-1}\dd y.
\]
Then
\begin{equation}\label{eq:alt-i-Jn-equation}
\lambda g
=\sum_{j=0}^{\lfloor\alpha\rfloor-1}a_jJ_nT_jg+J_nT_Rg.
\end{equation}
The integrals are absolutely convergent by
\eqref{eq:alt-i-moments} near $0$ and by the exponential decay of $g$
at infinity.

For $n\ge0$, set
\[
J_nu(x)=u(x)-\sum_{k=0}^n\frac{u^{(k)}(0)}{k!}x^k.
\]
By \eqref{eq:alt-i-flatness}, $J_ng=g$. Moreover,
$\mathfrak Mg(w)$ is absolutely convergent for every $w\in\mathbb C$.
For $0\leq j\leq\lfloor\alpha\rfloor-1$,
\[
(T_jg)^{(n+1)}(t)
=(-1)^{n+1}
\frac{\Gamma(\alpha-j+n+1)}{\Gamma(\alpha-j)}
\int_0^\infty
\frac{g(y)y^{\alpha-1}}
{(t+y)^{\alpha-j+n+1}}\dd y.
\]
Since
\[
\int_0^\infty |g(y)|y^{j-n-2}\dd y<\infty,
\]
this derivative is bounded for $t\geq0$.
Taylor's formula gives
\[
J_n(T_jg)(x)
=\frac1{n!}\int_0^x
(x-t)^n(T_jg)^{(n+1)}(t)\dd t,
\]
and hence, for fixed $n$,
\[
\begin{aligned}
|J_n(T_jg)(x)|
&\leq
\frac{\sup_{0\leq t\leq x}|(T_jg)^{(n+1)}(t)|}{n!}
\int_0^x(x-t)^n\dd t\\
&=
\frac{x^{n+1}}{(n+1)!}
\sup_{0\leq t\leq x}|(T_jg)^{(n+1)}(t)|
=O(x^{n+1}),
\qquad x\downarrow0.
\end{aligned}
\]
Using $j-\alpha<0$, we also obtain
\[
|T_jg(x)|
\leq x^{j-\alpha}
\int_0^\infty |g(y)|y^{\alpha-1}\dd y
=O(x^{j-\alpha}).
\]
The definition
\[
J_n(T_jg)(x)
=T_jg(x)
-\sum_{k=0}^n\frac{(T_jg)^{(k)}(0)}{k!}x^k
\]
therefore gives $J_n(T_jg)(x)=O(x^n)$ as $x\to\infty$.
Writing $\sigma=\Re s$, we obtain
\[
\begin{aligned}
\int_0^\infty |J_n(T_jg)(x)|x^{\sigma-1}\dd x
&\lesssim
\int_0^1x^{n+\sigma}\dd x
+\int_1^\infty x^{n+\sigma-1}\dd x\\
&<\infty,
\qquad -n-1<\sigma<-n.
\end{aligned}
\]
Thus $\mathfrak M(J_nT_jg)(s)$ converges absolutely in
$-n-1<\Re s<-n$.
For such $s$, Fubini's theorem gives
\[
\mathfrak M(J_nT_jg)(s)
=\frac1{n!}\int_0^\infty(T_jg)^{(n+1)}(t)
\left(\int_t^\infty x^{s-1}(x-t)^n\dd x\right)\dd t.
\]
With $x=t/u$,
\[
\int_t^\infty x^{s-1}(x-t)^n\dd x
=t^{s+n}\frac{\Gamma(-s-n)\Gamma(n+1)}{\Gamma(1-s)}.
\]
Substituting the derivative formula and interchanging the $t$- and
$y$-integrals,
\[
\begin{aligned}
\mathfrak M(J_nT_jg)(s)
&=(-1)^{n+1}\frac{\Gamma(-s-n)}{\Gamma(1-s)}
\frac{\Gamma(\alpha-j+n+1)}{\Gamma(\alpha-j)}\\
&\quad\times\int_0^\infty g(y)y^{\alpha-1}
\left(\int_0^\infty
\frac{t^{s+n}}{(t+y)^{\alpha-j+n+1}}\dd t\right)\dd y.
\end{aligned}
\]
With $t=yu$,
\[
\int_0^\infty\frac{t^{s+n}}{(t+y)^{\alpha-j+n+1}}\dd t
=y^{s+j-\alpha}
\frac{\Gamma(s+n+1)\Gamma(\alpha-j-s)}
{\Gamma(\alpha-j+n+1)}.
\]
Using
\[
\Gamma(-s-n)\Gamma(s+n+1)=\frac{(-1)^{n+1}\pi}{\sin\pi s},
\qquad
\Gamma(s)\Gamma(1-s)=\frac\pi{\sin\pi s},
\]
we obtain
\begin{equation}\label{eq:alt-i-Mellin-J}
\mathfrak M(J_nT_jg)(s)
=\frac{\Gamma(s)\Gamma(\alpha-j-s)}{\Gamma(\alpha-j)}
\mathfrak Mg(s+j),
\qquad -n-1<\Re s<-n.
\end{equation}
Applying $J_n$ to \eqref{eq:alt-i-Jn-equation}
and using $J_ng=g$, we obtain
\[
\lambda g
=\sum_{j=0}^{\lfloor\alpha\rfloor-1}a_jJ_nT_jg+J_nT_Rg.
\]
Taking Mellin transforms and using \eqref{eq:alt-i-Mellin-J}
with $a_0=1$ gives, for $-n-1<\Re s<-n$,
\begin{equation}\label{eq:alt-i-recurrence}
\begin{aligned}
\left(\lambda-\frac{\Gamma(s)\Gamma(\alpha-s)}{\Gamma(\alpha)}\right)
\mathfrak Mg(s)
&=\sum_{j=1}^{\lfloor\alpha\rfloor-1}
a_j\frac{\Gamma(s)\Gamma(\alpha-j-s)}{\Gamma(\alpha-j)}
\mathfrak Mg(s+j)+\mathfrak M(J_nT_Rg)(s).
\end{aligned}
\end{equation}

Set
\[
s=-n-\frac12+i\tau,
\qquad\tau\in\mathbb R,
\]
with $n$ tending to infinity through the even integers. By the Gamma
reflection formula~\cite[Eq.~(5.5.3)]{NIST2010},
\[
\frac{\Gamma(s)\Gamma(\alpha-s)}{\Gamma(\alpha)}
=-\frac{\pi}{\Gamma(\alpha)\cosh(\pi\tau)}
\frac{\Gamma(\alpha-s)}{\Gamma(1-s)}.
\]
Choose $A_0>(\alpha-1)/\pi$.
Here $\Re(1-s)=n+3/2$ and $|\arg(1-s)|<\pi/2$, so
\cite[Eq.~(5.11.13)]{NIST2010} gives
\[
\frac{\Gamma(\alpha-s)}{\Gamma(1-s)}
=(1-s)^{\alpha-1}\bigl(1+O_\alpha(n^{-1})\bigr),
\]
uniformly for $\tau\in\mathbb R$ as $n\to\infty$,

For $|\tau|\leq A_0\log n$,
\[
|\arg(1-s)|
=\arctan\frac{|\tau|}{n+3/2}
\leq\frac{A_0\log n}{n+3/2}.
\]
Consequently,
\[
\Re\frac{\Gamma(\alpha-s)}{\Gamma(1-s)}
=
|1-s|^{\alpha-1}
\left(
\cos\bigl((\alpha-1)\arg(1-s)\bigr)
+O_\alpha(n^{-1})
\right)>0
\]
for all sufficiently large $n$, uniformly in this range.
For even $n$, the reflection formula therefore gives
\[
\Re\frac{\Gamma(s)\Gamma(\alpha-s)}{\Gamma(\alpha)}
=
-\frac{\pi}{\Gamma(\alpha)\cosh(\pi\tau)}
\Re\frac{\Gamma(\alpha-s)}{\Gamma(1-s)}
<0,
\qquad |\tau|\leq A_0\log n.
\]
If $|\tau|\ge A_0\log n$, then
\[
\left|\frac{\Gamma(s)\Gamma(\alpha-s)}{\Gamma(\alpha)}\right|
\lesssim(n+|\tau|)^{\alpha-1}e^{-\pi|\tau|},
\]
and the choice $A_0>(\alpha-1)/\pi$ ensures that
\[
\sup_{|\tau|\ge A_0\log n}
(n+|\tau|)^{\alpha-1}e^{-\pi|\tau|}
\lesssim n^{\alpha-1-\pi A_0}\longrightarrow0
\]
as $n\to\infty$ through the even integers. Therefore
\begin{equation}\label{eq:alt-i-gamma-lower}
\inf_{\tau\in\mathbb R}
\left|\frac{\Gamma(s)\Gamma(\alpha-s)}{\Gamma(\alpha)}-\lambda\right|
\ge\frac\lambda2
\end{equation}
for all sufficiently large even $n$.
For $1\le j\le\lfloor\alpha\rfloor-1$,
\[
\frac{a_j\Gamma(s)\Gamma(\alpha-j-s)/\Gamma(\alpha-j)}
{\Gamma(s)\Gamma(\alpha-s)/\Gamma(\alpha)}
=\frac{a_j\Gamma(\alpha)}{\Gamma(\alpha-j)}
\prod_{\ell=1}^j\frac1{\alpha-s-\ell},
\]
and $|\alpha-s-\ell|\ge n/2$ for large $n$, uniformly in $\tau$.
If
\[
q=\frac{\Gamma(s)\Gamma(\alpha-s)}{\Gamma(\alpha)},
\]
then \eqref{eq:alt-i-gamma-lower} gives
\[
\left|\frac q{q-\lambda}\right|
=\left|1+\frac\lambda{q-\lambda}\right|\le3.
\]
Consequently,
\begin{equation}\label{eq:alt-i-coefficient}
\sup_{\tau\in\mathbb R}
\left|
\frac{a_j\Gamma(s)\Gamma(\alpha-j-s)/\Gamma(\alpha-j)}
{\Gamma(s)\Gamma(\alpha-s)/\Gamma(\alpha)-\lambda}
\right|\lesssim_jn^{-j}.
\end{equation}
Dividing \eqref{eq:alt-i-recurrence} by its left-hand coefficient,
using \eqref{eq:alt-i-coefficient}, taking $L^2(\dd\tau)$ norms, and
using the Mellin--Plancherel theorem,
\begin{equation}\label{eq:alt-i-L2-recurrence}
\|x^{-n-1}g\|_2
\lesssim\sum_{j=1}^{\lfloor\alpha\rfloor-1}
n^{-j}\|x^{-n+j-1}g\|_2
+\|x^{-n-1}J_nT_Rg\|_2.
\end{equation}

It remains to estimate $\|x^{-n-1}J_nT_Rg\|_2$. We first obtain a
bound for $R^{(n+1)}$. Choose $r>0$ so small that $B$ is holomorphic
in $|w|<4r$, and
choose $\eta>0$ such that
\[
\eta<\frac r{2(1+r)}.
\]
For $t\ge r$, on the circle $|w-t|=\eta(1+t)$,
\[
\Re w\ge t-\eta(1+t)>\frac t2\ge\frac r2.
\]
Since $j-\alpha<0$ for
$0\le j\le\lfloor\alpha\rfloor-1$, on these circles we have
\[
\begin{aligned}
|R(w)|
&\le |\zeta(1+w)^\alpha-1|
+\sum_{j=0}^{\lfloor\alpha\rfloor-1}|a_j|\,|w|^{j-\alpha}\\
&\le 1+\zeta(1+r/2)^\alpha
+\sum_{j=0}^{\lfloor\alpha\rfloor-1}|a_j|(r/2)^{j-\alpha}
\lesssim1.
\end{aligned}
\]
Cauchy's formula therefore gives
\[
|R^{(n+1)}(t)|
\lesssim(n+1)!\eta^{-n-1}(1+t)^{-n-1}
\lesssim(n+1)!\eta^{-n-1}(1+t)^{-2},
\qquad t\ge r.
\]
For $0<t<r$,
\[
R(t)=-1+t^{\beta-1}B(t),
\]
and, since $|w-t|=2r$ lies in $|w|<3r$,
\[
|B^{(k)}(t)|\lesssim\frac{k!}{(2r)^k},
\qquad k\ge0.
\]
For $m\geq1$, differentiation gives
\[
\left|\frac{\mathrm d^m}{\mathrm d t^m}t^{\beta-1}\right|
=
\frac{\Gamma(m+1-\beta)}{\Gamma(1-\beta)}
t^{\beta-1-m}.
\]
By \cite[Eq.~(5.11.13)]{NIST2010},
\[
\frac{\Gamma(m+1-\beta)}{\Gamma(m+1)}
=m^{-\beta}\bigl(1+O_\beta(m^{-1})\bigr),
\qquad m\to\infty.
\]
Thus
\[
\left|\frac{\mathrm d^m}{\mathrm d t^m}t^{\beta-1}\right|
\lesssim\ m!m^{-\beta}t^{\beta-1-m},
\qquad m\geq1,\quad t>0.
\]
Leibniz' formula yields
\[
R^{(n+1)}(t)
=\sum_{m=0}^{n+1}\binom{n+1}{m}
\frac{\mathrm d^m}{\mathrm d t^m}t^{\beta-1}
B^{(n+1-m)}(t).
\]
The term $m=0$ satisfies
\[
t^{\beta-1}|B^{(n+1)}(t)|
\lesssim(n+1)!t^{\beta-n-2}
\left(\frac t{2r}\right)^{n+1}.
\]
For $1\le m\le n+1$,
\[
\binom{n+1}{m}
\left|\frac{\mathrm d^m}{\mathrm d t^m}t^{\beta-1}\right|
|B^{(n+1-m)}(t)|
\lesssim(n+1)!t^{\beta-n-2}m^{-\beta}
\left(\frac t{2r}\right)^{n+1-m}.
\]
Using $t/(2r)<1/2$, we estimate
\[
\begin{aligned}
\sum_{m=1}^{n+1}m^{-\beta}
\left(\frac t{2r}\right)^{n+1-m}
&=\sum_{k=0}^n(n+1-k)^{-\beta}
\left(\frac t{2r}\right)^k\\
&\le(n+1)^{-\beta}
\sum_{k=0}^n(k+1)^\beta2^{-k}
\lesssim n^{-\beta}.
\end{aligned}
\]
Also $2^{-n-1}\lesssim n^{-\beta}$. Therefore
\[
|R^{(n+1)}(t)|
\lesssim(n+1)!n^{-\beta}t^{\beta-n-2},
\qquad0<t<r.
\]
With a fixed $C_*>1$, $C_*\ge\eta^{-1}$,
\[
|R^{(n+1)}(t)|
\lesssim(n+1)!\left(
n^{-\beta}t^{\beta-n-2}\ind_{(0,r)}(t)
+C_*^{n+1}(1+t)^{-2}
\right).
\]
Differentiation under $T_R$ is justified by
\eqref{eq:alt-i-moments} and the exponential decay of $g$. Since
\[
y^{\alpha-1}|g(y)|
=y^{n+1-\beta}|y^{-n+\lfloor\alpha\rfloor-1}g(y)|,
\]
we obtain
\[
\begin{aligned}
|(T_Rg)^{(n+1)}(x)|
&\lesssim(n+1)!n^{-\beta}
\int_0^\infty\frac{y^{n+1-\beta}}
{(x+y)^{n+2-\beta}}
|y^{-n+\lfloor\alpha\rfloor-1}g(y)|\dd y\\
&\quad+(n+1)!C_*^{n+1}
\int_0^\infty\frac{|g(y)|y^{\alpha-1}}{(1+x+y)^2}\dd y.
\end{aligned}
\]
For
\[
K_n(x,y)=\frac{y^{n+1-\beta}}{(x+y)^{n+2-\beta}},
\]
Schur's test with weight $x^{-1/2}$ gives
\[
x^{1/2}\int_0^\infty K_n(x,y)y^{-1/2}\dd y
=B\left(n+\frac32-\beta,\frac12\right),
\]
and
\[
y^{1/2}\int_0^\infty K_n(x,y)x^{-1/2}\dd x
=B\left(\frac12,n+\frac32-\beta\right),
\]
both $O(n^{-1/2})$. Moreover,
\[
H(x):=\int_0^\infty\frac{|g(y)|y^{\alpha-1}}{(1+x+y)^2}\dd y
\lesssim(1+x)^{-2}\in L^2(\mathbb R_+).
\]
Hence
\[
\|(T_Rg)^{(n+1)}\|_2
\lesssim(n+1)!\left(
n^{-\beta-1/2}\|x^{-n+\lfloor\alpha\rfloor-1}g\|_2
+C_*^{n+1}\right).
\]
Taylor's formula gives
\[
J_n(T_Rg)(x)=\frac1{n!}\int_0^x
(x-t)^n(T_Rg)^{(n+1)}(t)\dd t.
\]
Thus
\[
x^{-n-1}J_n(T_Rg)(x)
=\frac1{n!}\int_0^\infty
x^{-n-1}(x-t)^n\ind_{\{0<t<x\}}(T_Rg)^{(n+1)}(t)\dd t.
\]
We apply Schur's test to the kernel
\[
\frac1{n!}x^{-n-1}(x-t)^n\mathbf1_{\{0<t<x\}}
\]
with the weight $x^{-1/2}$.
The substitution $t=xu$ gives
\[
\begin{aligned}
x^{1/2}\frac1{n!}\int_0^x
x^{-n-1}(x-t)^nt^{-1/2}\dd t
&=\frac1{n!}\int_0^1u^{-1/2}(1-u)^n\dd u\\
&=\frac1{n!}B\left(\frac12,n+1\right).
\end{aligned}
\]
The substitution $x=t/u$ gives
\[
\begin{aligned}
t^{1/2}\frac1{n!}\int_t^\infty
x^{-n-3/2}(x-t)^n\dd x
&=\frac1{n!}\int_0^1u^{-1/2}(1-u)^n\dd u\\
&=\frac1{n!}B\left(\frac12,n+1\right).
\end{aligned}
\]
Therefore
\[
\|x^{-n-1}J_n(T_Rg)\|_2
\le\frac{B(1/2,n+1)}{n!}\|(T_Rg)^{(n+1)}\|_2.
\]
Substituting $B(1/2,n+1)=O(n^{-1/2})$ into the preceding estimate gives
\[
\|x^{-n-1}J_nT_Rg\|_2
\lesssim n^{-\beta}\|x^{-n+\lfloor\alpha\rfloor-1}g\|_2
+n^{1/2}C_*^{n+1}.
\]
Choose a fixed $C_1>C_*$. Then
$n^{1/2}C_*^{n+1}\lesssim C_1^{n+1}$, and hence
\begin{equation}\label{eq:alt-i-remainder}
\|x^{-n-1}J_nT_Rg\|_2
\lesssim n^{-\beta}\|x^{-n+\lfloor\alpha\rfloor-1}g\|_2
+C_1^{n+1}
\end{equation}
Substituting
\eqref{eq:alt-i-remainder} into \eqref{eq:alt-i-L2-recurrence},
\[
\begin{aligned}
\|x^{-n-1}g\|_2
&\lesssim\sum_{j=1}^{\lfloor\alpha\rfloor-1}
n^{-j}\|x^{-n+j-1}g\|_2+n^{-\beta}\|x^{-n+\lfloor\alpha\rfloor-1}g\|_2
+C_1^{n+1}.
\end{aligned}
\]
Choose $c>0$ so small that $cC_1<1$. For every fixed $j\ge1$ and
sufficiently large $n$,
\[
\begin{aligned}
c^{2(n-j+1)}\|x^{-n+j-1}g\|_2^2
&=\int_0^c\left(\frac cx\right)^{2(n-j+1)}|g(x)|^2\dd x\\
&\quad+\int_c^\infty\left(\frac cx\right)^{2(n-j+1)}|g(x)|^2\dd x\\
&\le c^{2n+2}\|x^{-n-1}g\|_2^2+\|g\|_2^2.
\end{aligned}
\]
Hence
\begin{equation}\label{eq:alt-i-scaling}
c^{n-j+1}\|x^{-n+j-1}g\|_2
\le c^{n+1}\|x^{-n-1}g\|_2+\|g\|_2.
\end{equation}
Apply \eqref{eq:alt-i-scaling} for
$1\le j\le\lfloor\alpha\rfloor$. Multiplying the preceding estimate
by $c^{n+1}$,
\[
\begin{aligned}
c^{n+1}\|x^{-n-1}g\|_2
&\lesssim
\left(\sum_{j=1}^{\lfloor\alpha\rfloor-1}n^{-j}+n^{-\beta}\right)
\left(c^{n+1}\|x^{-n-1}g\|_2+\|g\|_2\right)+(cC_1)^{n+1}.
\end{aligned}
\]
The coefficient in parentheses tends to zero, so it can be absorbed
for all sufficiently large even $n$. Thus
\[
c^{n+1}\|x^{-n-1}g\|_2\longrightarrow0
\qquad(n\to\infty,\ n\text{ even}).
\]
Since
\[
c^{2n+2}\|x^{-n-1}g\|_2^2
=\int_0^\infty\left(\frac cx\right)^{2n+2}|g(x)|^2\dd x
\ge\int_0^c|g(x)|^2\dd x,
\]
we have $g=0$ on $(0,c)$. The right-hand side of
\eqref{eq:alt-i-equation} is real analytic for $x>0$, hence $g$ is
real analytic on $(0,\infty)$ and $g\equiv0$, contradicting
\eqref{eq:1}.

If $\alpha\in\N$, expand $(t\zeta(1+t))^\alpha$ through degree
$\alpha-1$ and write
\[
\phi_\alpha(t)=\sum_{j=0}^{\alpha-1}a_jt^{j-\alpha}+R(t),
\]
where $R$ is holomorphic at $0$. The same argument gives
\eqref{eq:alt-i-Mellin-J}--\eqref{eq:alt-i-L2-recurrence}, now with
$1\le j\le\alpha-1$, while Cauchy's formula gives
\[
\|x^{-n-1}J_nT_Rg\|_2\lesssim C_1^{n+1}.
\]
The proof of \eqref{eq:alt-i-scaling} and the same absorption then give
$c^{n+1}\|x^{-n-1}g\|_2\to0$, and the preceding argument again yields
$g\equiv0$. Thus alternative~(i) is impossible.
\subsection{Exclusion of alternative (ii)}
Assume now that alternative~(ii) occurs. Let \(z\) be a non-removable
pole. Since there are no such poles on the line \(\Re z=\frac12\), write
\[
  z=\frac12-x+i\eta,
  \qquad x>0.
\]
so that \eqref{eq:6} becomes
\begin{equation}\label{eq:23}
\frac{\Gamma(\alpha/2-x+i\eta)\Gamma(\alpha/2+x-i\eta)}{\Gamma(\alpha)}=\lambda.
\end{equation}
 For $0<\nu\le\Lambda_\alpha$, define the unique $\xi_\nu\ge0$ by
\[
m_\rho(\xi_\nu)=\nu.
\]

We have to prove $|\eta|>\xi_\lambda$,   and we need the following strict inequality.

\begin{lemma}\label{lem:embedded-gamma-modulus}

Let $\alpha>1$ , if $x>0$, $y\in\mathbb R$, and
\[
\Gamma\left(\frac\alpha2+x+iy\right)
\Gamma\left(\frac\alpha2-x-iy\right)
\]
is finite, then
\[
\left|\Gamma\left(\frac\alpha2+x+iy\right)
\Gamma\left(\frac\alpha2-x-iy\right)\right|
>
\left|\Gamma\left(\frac\alpha2+iy\right)\right|^2.
\]

\end{lemma}

\smallskip
\noindent\textit{Proof.}
First, for $x\ge1/2$, and put $\psi=\Gamma'/\Gamma$
\[
\Re\psi'(x+iy)>0.
\]
Indeed, the reflection formula for the digamma function~\cite[Eq.~(5.5.4)]{NIST2010} is
\[
  \psi(1-z)-\psi(z)=\pi\cot(\pi z).
\]
Differentiating with respect to \(z\) gives
\[
  \psi'(1-z)+\psi'(z)=\pi^2\csc^2(\pi z).
\]
Taking \(z=\frac12+iy\), using
\[
  \psi'\left(\frac12-iy\right)
  =\overline{\psi'\left(\frac12+iy\right)}
\]
and \(\sin(\pi(\frac12+iy))=\cosh(\pi y)\), we obtain
\[
2\Re\psi'\left(\frac12+iy\right)
=\frac{\pi^2}{\cosh^2(\pi y)}>0.
\]
The function $\psi'$ has no pole in $\Re z\ge1/2$, and
\[
\psi'(z)=z^{-1}+O(|z|^{-2})
\]
uniformly in this half-plane as $|z|\to\infty$. Hence $\Re\psi'$ is bounded and harmonic in $\Re z>1/2$, continuous on the boundary, and tends to $0$ at infinity. The Poisson integral formula gives
\[
\Re\psi'(x+iy)
=\frac1\pi\int_{\mathbb R}
\frac{x-1/2}{(y-t)^2+(x-1/2)^2}
\Re\psi'\left(\frac12+it\right)dt>0.
\]
We claim that
\[
\Re\psi(v+iy)>\Re\psi(u+iy),
\qquad0<u<v,
\qquad u+v>1.
\]
If $u\ge1/2$, then $v>u\ge1/2$, and the claim follows by integrating $\Re\psi'(s+iy)>0$ from $u$ to $v$. If $0<u<1/2$, then $1-u>1/2$ and $v>1-u$. Hence
\[
\Re\psi(v+iy)>\Re\psi(1-u+iy).
\]
Applying the same reflection formula at \(z=u+iy\), we obtain
\[
  \psi(1-u-iy)-\psi(u+iy)
  =\pi\cot(\pi(u+iy)).
\]
Since \(\psi(\overline z)=\overline{\psi(z)}\) and
\[
  \cot(\pi(u+iy))
  =\frac{\sin(2\pi u)-i\sinh(2\pi y)}
  {\cosh(2\pi y)-\cos(2\pi u)},
\]
and since \(0<u<1/2\) implies
\[
  \sin(2\pi u)>0,
  \qquad
  \cosh(2\pi y)-\cos(2\pi u)>0,
\]
taking real parts gives
\[
\Re\psi(1-u+iy)-\Re\psi(u+iy)
=
\frac{\pi\sin(2\pi u)}{\cosh(2\pi y)-\cos(2\pi u)}>0.
\]
Thus the claim holds in both cases.

For $0<x<\alpha/2$, set
\[
u=\frac\alpha2-x,
\qquad
v=\frac\alpha2+x.
\]
Then $0<u<v$ and $u+v=\alpha>1$, so
\[
\begin{aligned}
\frac{\mathrm d}{\mathrm d x}\log\left|
\Gamma\left(\frac\alpha2+x+iy\right)
\Gamma\left(\frac\alpha2-x-iy\right)
\right|
&=\Re\psi\left(\frac\alpha2+x+iy\right)-\Re\psi\left(\frac\alpha2-x+iy\right)>0.
\end{aligned}
\]
Integrating from $0$ to $x$,
\[
\left|
\Gamma\left(\frac\alpha2+x+iy\right)
\Gamma\left(\frac\alpha2-x-iy\right)
\right|
>
\left|\Gamma\left(\frac\alpha2+iy\right)\right|^2,
\qquad0<x<\frac\alpha2.
\]
Replacing $x$ by $-x$ gives the same inequality. Hence
\[
\left|
\Gamma\left(\frac\alpha2+x+iy\right)
\Gamma\left(\frac\alpha2-x-iy\right)
\right|
>
\left|\Gamma\left(\frac\alpha2+iy\right)\right|^2,
\qquad0<|x|<\frac\alpha2.
\]

To extend the inequality to every $x>0$, write
\[
x=k+r,
\qquad k\in\mathbb N_0,
\qquad-\frac12<r\le\frac12.
\]
For $j=0,\ldots,k-1$, the Gamma recurrence formula gives
\[
\frac{
\Gamma(\alpha/2+r+j+1+iy)\Gamma(\alpha/2-r-j-1-iy)}
{\Gamma(\alpha/2+r+j+iy)\Gamma(\alpha/2-r-j-iy)}
=
\frac{\alpha/2+r+j+iy}{\alpha/2-r-j-1-iy}.
\]
Assume that
\[
\Gamma\left(\frac\alpha2+x+iy\right)
\Gamma\left(\frac\alpha2-x-iy\right)
\]
is finite. Then
\[
\frac\alpha2-r-j-1-iy\ne0,
\qquad0\le j\le k-1.
\]
Indeed, if this quantity were zero for some $j$, then $y=0$ and $\alpha/2-r=j+1$. Since $x=k+r$,
\[
\frac\alpha2-x=j+1-k\in\{0,-1,\ldots\},
\]
so $\Gamma(\alpha/2-x)$ would have a pole, contradicting the assumed finiteness. Furthermore,
\[
\left|\frac\alpha2+r+j+iy\right|^2
-
\left|\frac\alpha2-r-j-1-iy\right|^2
=(\alpha-1)(2r+2j+1)>0.
\]
Hence, if $k\ge1$,
\[
\frac{
|\Gamma(\alpha/2+x+iy)\Gamma(\alpha/2-x-iy)|}
{|\Gamma(\alpha/2+r+iy)\Gamma(\alpha/2-r-iy)|}
=
\prod_{j=0}^{k-1}
\left|
\frac{\alpha/2+r+j+iy}{\alpha/2-r-j-1-iy}
\right|>1.
\]
If $r\ne0$, the denominator is strictly larger than $|\Gamma(\alpha/2+iy)|^2$ by the local inequality. If $r=0$, then $k\ge1$ and the product is strictly larger than $1$; if $k=0$, the local inequality applies directly. Hence, whenever the left-hand side is finite,
\[
|\Gamma(\alpha/2+x+iy)\Gamma(\alpha/2-x-iy)|
>|\Gamma(\alpha/2+iy)|^2,
\qquad x>0.
\]

\hfill$\square$

\bigskip

Applying Lemma~\ref{lem:embedded-gamma-modulus} to a solution of \eqref{eq:23} gives
\[
\lambda>m_\rho(\eta),
\qquad |\eta|>\xi_\lambda.
\]
In particular $\eta\ne0$.

We now prove that every solution of \eqref{eq:6} with $\Re z<1/2$ is simple. Differentiating explicitly,
\[
\begin{aligned}
&\frac{\mathrm d}{\mathrm d z}\left[
\frac{\Gamma(\rho+z)\Gamma(1+\rho-z)}{\Gamma(1+2\rho)}-\lambda
\right]=
\frac{\Gamma(\rho+z)\Gamma(1+\rho-z)}{\Gamma(1+2\rho)}
\bigl[\psi(\rho+z)-\psi(1+\rho-z)\bigr].
\end{aligned}
\]
At a solution
\[
z=\frac12-x+i\eta
\]
of \eqref{eq:23}, the derivative equals
\[
\lambda\left[
\psi\left(\frac\alpha2-x+i\eta\right)
-
\psi\left(\frac\alpha2+x-i\eta\right)
\right].
\]
To prove that the derivative is non-zero, it suffices to show that
the bracket is non-zero, as $\lambda>0$.
For $\eta>0$, choose a positive integer $M$ such that
$\alpha/2-x+M>0$. By \cite[Eq.~(5.5.2)]{NIST2010},
\[
\psi\left(\frac\alpha2-x+i\eta\right)
=
\psi\left(\frac\alpha2-x+M+i\eta\right)
-\sum_{k=0}^{M-1}\frac1{\alpha/2-x+k+i\eta}.
\]
Taking imaginary parts and applying
\cite[Eq.~(5.7.6)]{NIST2010} to both terms in the bracket gives
\[
\begin{aligned}
&\Im\left[
\psi\left(\frac\alpha2-x+i\eta\right)
-\psi\left(\frac\alpha2+x-i\eta\right)
\right]\\
&\quad=
\sum_{m=0}^\infty
\frac{\eta}{(m+\alpha/2-x+M)^2+\eta^2}
+\sum_{k=0}^{M-1}
\frac{\eta}{(\alpha/2-x+k)^2+\eta^2}\\
&\qquad+
\sum_{m=0}^\infty
\frac{\eta}{(m+\alpha/2+x)^2+\eta^2}
>0.
\end{aligned}
\]
For $\eta<0$, complex conjugation gives a strictly negative
imaginary part. Thus the bracket is non-zero, and every solution
of \eqref{eq:6} with $\Re z<1/2$ is simple.

Equation~\eqref{eq:6} has only finitely many solutions in each
closed vertical strip of finite width. Under alternative~(ii), there is
therefore a smallest $x_*>0$ for which $\widetilde F$ has a
non-removable pole on $\Re z=1/2-x_*$. Write these poles as
\begin{equation}\label{eq:25}
z_j^\pm=\frac12-x_*\pm i\eta_j,
\qquad0<\eta_1<\cdots<\eta_r.
\end{equation}
The poles occur in conjugate pairs. Since $q_N$ is real-valued,
\[
\Psi_N(\overline z)=\overline{\Psi_N(z)}
\]
throughout its strip, while
\[
\mathfrak m_\rho(\overline z)-\lambda
=
\overline{\mathfrak m_\rho(z)-\lambda}.
\]
The polynomial factor in the denominator of \eqref{eq:5} has only real zeros. Thus the non-removable poles in \eqref{eq:25} are paired, and they are simple by the preceding argument. Define
\[
A_j=\lim_{z\to z_j^+}(z-z_j^+)\widetilde F(z)\ne0.
\]
Then
\[
\lim_{z\to z_j^-}(z-z_j^-)\widetilde F(z)=\overline{A_j}.
\]
By the inequality obtained above,
\[
\eta_1>\xi_\lambda.
\]
Choose $0<\delta<1/2$ so that
\[
(1-2\delta)\eta_1>\xi_\lambda.
\]
By continuity of $\nu\mapsto\xi_\nu$ on $(0,\Lambda_\alpha]$, we may
choose a fixed $0<\varepsilon<\lambda$ such that
\begin{equation}\label{eq:26}
\xi_{\lambda-\varepsilon}<(1-2\delta)\eta_1.
\end{equation}

 By the definition of $x_*$, every possible singularity at a solution of \eqref{eq:6} with depth $d<x_*$ is removable, so the preceding induction yields
\[
x^{-d}f\in L^2(\mathbb R_+),
\qquad0\le d<x_*.
\]
If $x_*\notin\N$, set $N=\lceil x_*\rceil$. Then
\[
N-1<x_*<N,
\]
and the preceding induction gives
\[
x^{-N+1}f\in L^2(\mathbb R_+).
\]
Hence $q_N$ and $\Psi_N$ are defined. The weighted estimates for
$q_N$ give
\[
x^{-N}q_N\in L^2(\mathbb R_+),
\qquad
x^{1/2}q_N\in L^2(\mathbb R_+),
\]
so $\Psi_N$ is analytic in
\[
\frac12-N<\Re z<1.
\]
Since
\[
\frac12-N<\frac12-x_*<\frac32-N<1,
\]
the line $\Re z=1/2-x_*$ lies strictly inside this Mellin strip.

If $x_*=N\in\N$, the preceding induction gives
\[
x^{-d}f\in L^2(\mathbb R_+),
\qquad 0\leq d<N.
\]
Taking $d=N-1/2$, we obtain
\[
x^{-N+1/2}f\in L^2(\mathbb R_+).
\]
It then follows from Lemma~\ref{lem:embedded-remainder-regularity}, with $\delta=1/2$, that
\[
x^{-N-1/2}q_N\in L^2(\mathbb R_+).
\]
Together with
\[
x^{1/2}q_N\in L^2(\mathbb R_+),
\]
this shows that $\Psi_N$ is analytic in
\[
-N<\Re z<1.
\]
Since
\[
-N<\frac12-N=\frac12-x_*<1,
\]
the line $\Re z=1/2-x_*$ again lies strictly inside the Mellin strip
of $\Psi_N$.
Thus, in both cases, the line
\[
\Re z=\frac12-x_*
\]
lies strictly inside the Mellin strip of $\Psi_N$. Since the
solutions of \eqref{eq:6} in a neighborhood of this line are finite
and isolated, we may choose $\kappa>0$ so small that the closed strip
\[
\sigma_-:=\frac12-x_*-\kappa
\leq\Re z\leq
\sigma_+:=\frac12-x_*+\kappa
\]
is contained in the Mellin strip of $\Psi_N$, formula \eqref{eq:5}
defines $\widetilde F$ meromorphically on a neighborhood of this
closed strip, and $z_j^\pm$ are its only non-removable poles there.
Put
\[
H(z)=\widetilde F(z)-\sum_{j=1}^r\left(
\frac{A_j}{z-z_j^+}
+
\frac{\overline{A_j}}{z-z_j^-}
\right).
\]
Then $H$ is analytic on a neighborhood of this closed strip.
The closed strip $\sigma_-\leq\Re z\leq\sigma_+$ lies inside the
Mellin strip of $\Psi_N$. Lemma~\ref{lem:embedded-mellin-strip} gives
\[
\sup_{\sigma_-\leq\sigma\leq\sigma_+}
\int_{\mathbb R}|\Psi_N(\sigma+i\tau)|^2\dd\tau
<\infty.
\]
By the preceding estimate for $\widetilde F$, there exists $T>0$
such that
\[
|\widetilde F(\sigma+i\tau)|
\lesssim(1+|\tau|)^{-N}|\Psi_N(\sigma+i\tau)|,
\qquad
\sigma_-\leq\sigma\leq\sigma_+,\quad |\tau|\geq T.
\]
Increasing $T$ if necessary, we also have
\[
\sum_{j=1}^r\left(
\frac{|A_j|}{|\sigma+i\tau-z_j^+|}
+
\frac{|A_j|}{|\sigma+i\tau-z_j^-|}
\right)
\lesssim(1+|\tau|)^{-1}
\]
uniformly in this range. By the definition of $H$,
\[
|H(\sigma+i\tau)|^2
\lesssim
(1+|\tau|)^{-2N}|\Psi_N(\sigma+i\tau)|^2
+(1+|\tau|)^{-2}.
\]
The uniform vertical $L^2$ bound for $\Psi_N$ therefore gives
\[
\sup_{\sigma_-\leq\sigma\leq\sigma_+}
\int_{|\tau|\geq T}|H(\sigma+i\tau)|^2\dd\tau<\infty.
\]
The analyticity of $H$ in a neighborhood of the closed strip also
implies that it is bounded on
$\{\sigma+i\tau:\sigma_-\leq\sigma\leq\sigma_+,\ |\tau|\leq T\}$.
Hence
\begin{equation}\label{eq:27}
\sup_{\sigma_-\leq\sigma\leq\sigma_+}
\int_{\mathbb R}|H(\sigma+i\tau)|^2\dd\tau<\infty.
\end{equation}

Define
\[
p(x)=\mathbf 1_{(0,1)}(x)
\sum_{j=1}^r\left(A_jx^{-z_j^+}+\overline{A_j}x^{-z_j^-}\right).
\]
For $\Re z>1/2-x_*$, direct integration gives
\[
  (\mathfrak Mp)(z)
  =
  \sum_{j=1}^r
  \left(
    \frac{A_j}{z-z_j^+}
    +\frac{\overline{A_j}}{z-z_j^-}
  \right).
\]
Thus, by the definition of $H$,
\[
  H(z)=(\mathfrak M(f-p))(z)
\]
on this half-plane wherever $H$ is defined.
By \eqref{eq:27} and the converse part of
Lemma~\ref{lem:embedded-mellin-strip}, we obtain
\begin{equation}\label{eq:28}
  x^{\sigma_--1/2}(f-p)\in L^2(\mathbb R_+).
\end{equation}

To express \eqref{eq:28} in logarithmic variables, set $x=e^t$ and put
\[
  u(t)=e^{t/2}f(e^t),
  \qquad
  Q(t)=e^{-x_*t}u(t),
  \qquad t\in\mathbb R.
\]
The change of variables is unitary from $L^2(\mathbb R_+)$
onto $L^2(\mathbb R)$, since
\[
  \int_{\mathbb R}|u(t)|^2\dd t
  =\int_0^\infty|f(x)|^2\dd x.
\]
Define
\[
  P(t)=\sum_{j=1}^r
  \left(A_je^{-i\eta_jt}
  +\overline{A_j}e^{i\eta_jt}\right).
\]
By the definition of $p$ and
$z_j^\pm=\frac12-x_*\pm i\eta_j$,
\[
  e^{-x_*t}e^{t/2}\bigl(f(e^t)-p(e^t)\bigr)
  =Q(t)-P(t),
  \qquad t<0.
\]

Changing variables $x=e^t$
in \eqref{eq:28} gives
\[
\begin{aligned}
  &\int_{-\infty}^0
  |Q(t)-P(t)|^2e^{-2\kappa t}\dd t\\
  &\quad=
  \int_0^1
  x^{2\sigma_--1}|f(x)-p(x)|^2\dd x
  <\infty.
\end{aligned}
\]
Hence there exists $W\in L^2(-\infty,0)$ such that
\begin{equation}\label{eq:30}
  Q(t)-P(t)=e^{\kappa t}W(t),
  \qquad t<0.
\end{equation}
Under the same unitary change of variables, \eqref{eq:1} becomes
\[
\int_{\mathbb R}\widetilde k_\alpha(t,s)u(s)\dd s=\lambda u(t),
\]
where
\[
\widetilde k_\alpha(t,s)=e^{\alpha(t+s)/2}
\left(\zeta(1+e^t+e^s)^\alpha-1\right).
\]
The bound $\phi_\alpha(q)\lesssim q^{-\alpha}$ for $q>0$ gives
\begin{equation}\label{eq:31}
0<\widetilde k_\alpha(t,s)
\lesssim\left(2\cosh\frac{t-s}{2}\right)^{-\alpha}
\le e^{-\alpha|t-s|/2}.
\end{equation}

We now pass to a finite interval. Fix $C_0>1$ and, for $L>1$, put
\[
  I_L=[-L,C_0\log L].
\]
Let $K_{\alpha,L}$ be the integral operator on $L^2(I_L)$ with kernel
$\widetilde k_\alpha$, and put $u_L=u|_{I_L}$. Restricting the
eigenvalue equation to $I_L$ gives
\[
  (K_{\alpha,L}-\lambda)u_L=r_L,
\]
where, for $t\in I_L$,
\[
  r_L(t)=
  -\int_{-\infty}^{-L}\widetilde k_\alpha(t,s)u(s)\dd s
  -\int_{C_0\log L}^{\infty}\widetilde k_\alpha(t,s)u(s)\dd s.
\]
The operator $K_{\alpha,L}$ is positive. Indeed, on $I_L\times I_L$ the
Dirichlet expansion gives
\[
  \widetilde k_\alpha(t,s)
  =
  \sum_{n\ge2}\frac{d_\alpha(n)}n
  \left(e^{\alpha t/2}e^{-(\log n)e^t}\right)
  \left(e^{\alpha s/2}e^{-(\log n)e^s}\right),
\]
and hence $\langle K_{\alpha,L}g,g\rangle\ge0$.
To use $u_L$ as an approximate eigenfunction of $K_{\alpha,L}$, we estimate
the truncation error. We claim that
\begin{equation}\label{eq:32}
  \|r_L\|_{L^2(I_L)}\lesssim e^{-x_*L}.
\end{equation}
Indeed, for $s<0$, \eqref{eq:30} gives
\[
u(s)=e^{x_*s}P(s)+e^{(x_*+\kappa)s}W(s).
\]
Write $t=-L+r$, where $0\le r\le L+C_0\log L$. If $s\le-L$, then $|t-s|=t-s$. By \eqref{eq:31} and the boundedness of $P$,
\[
\int_{-\infty}^{-L}\widetilde k_\alpha(t,s)e^{x_*s}|P(s)|\dd s
\lesssim e^{-x_*L}e^{-\alpha r/2}.
\]
For the second term, the Cauchy--Schwarz inequality gives
\[
\begin{aligned}
&\int_{-\infty}^{-L}\widetilde k_\alpha(t,s)e^{(x_*+\kappa)s}|W(s)|\dd s\\
&\qquad\lesssim e^{-\alpha t/2}
\left(\int_{-\infty}^{-L}e^{2(\alpha/2+x_*+\kappa)s}\dd s\right)^{1/2}
\|W\|_{L^2(-\infty,0)}\\
&\qquad\lesssim e^{-(x_*+\kappa)L}e^{-\alpha r/2}
\|W\|_{L^2(-\infty,0)}.
\end{aligned}
\]
Combining the preceding two estimates, for
$t=-L+r\in I_L$ we have
\[
\begin{aligned}
\left|
\int_{-\infty}^{-L}
\widetilde k_\alpha(t,s)u(s)\dd s
\right|
&\lesssim
e^{-x_*L}e^{-\alpha r/2}+
e^{-(x_*+\kappa)L}e^{-\alpha r/2}
\|W\|_{L^2(-\infty,0)}\\
&\lesssim
e^{-x_*L}e^{-\alpha r/2}.
\end{aligned}
\]
Integrating with respect to $t$, using $\dd t=\dd r$ and
$0\leq r\leq L+C_0\log L$, gives
\[
\begin{aligned}
&\left\|
\int_{-\infty}^{-L}
\widetilde k_\alpha(\,\cdot\,,s)u(s)\dd s
\right\|_{L^2(I_L)}^2\\
&\quad\lesssim
e^{-2x_*L}
\int_0^{L+C_0\log L}e^{-\alpha r}\dd r
\leq
\frac1\alpha e^{-2x_*L}.
\end{aligned}
\]
Consequently,
\[
\left\|
\int_{-\infty}^{-L}
\widetilde k_\alpha(\,\cdot\,,s)u(s)\dd s
\right\|_{L^2(I_L)}
\lesssim e^{-x_*L}.
\]
For the integral over $[C_0\log L,\infty)$, the estimate for $f$ at infinity gives
\[
|u(s)|
=
e^{s/2}|f(e^s)|
\lesssim
e^{\alpha s/2}e^{-(\log 2)e^s},
\qquad s\geq0.
\]
If $t\in I_L$ and $s\geq C_0\log L$, then
\[
t\leq C_0\log L\leq s.
\]
Hence \eqref{eq:31} gives
\[
\widetilde k_\alpha(t,s)
\lesssim
e^{-\alpha(s-t)/2}
=
e^{\alpha t/2}e^{-\alpha s/2}.
\]
It follows that
\[
\begin{aligned}
&\int_{C_0\log L}^\infty
\widetilde k_\alpha(t,s)|u(s)|\dd s\\
&\quad\lesssim
e^{\alpha t/2}
\int_{C_0\log L}^\infty
e^{-\alpha s/2}
e^{\alpha s/2}
e^{-(\log 2)e^s}\dd s\\
&\quad=
e^{\alpha t/2}
\int_{C_0\log L}^\infty
e^{-(\log 2)e^s}\dd s.
\end{aligned}
\]
For $t\leq C_0\log L$, we have
\[
e^{\alpha t/2}
\leq
L^{\alpha C_0/2}.
\]
Making the change of variables $y=e^s$, we obtain
\[
\begin{aligned}
&\int_{C_0\log L}^\infty
\widetilde k_\alpha(t,s)|u(s)|\dd s\\
&\quad\lesssim
L^{\alpha C_0/2}
\int_{L^{C_0}}^\infty
e^{-(\log 2)y}\frac{\dd y}{y}.
\end{aligned}
\]
Because $L^{C_0}\geq1$,
\[
\begin{aligned}
\int_{L^{C_0}}^\infty
e^{-(\log 2)y}\frac{\dd y}{y}
&\leq
\int_{L^{C_0}}^\infty
e^{-(\log 2)y}\dd y\\
&=
\frac1{\log 2}
e^{-(\log 2)L^{C_0}}\\
&\lesssim
e^{-(\log 2)L^{C_0}/2}.
\end{aligned}
\]
Therefore, uniformly for $t\in I_L$,
\[
\int_{C_0\log L}^\infty
\widetilde k_\alpha(t,s)|u(s)|\dd s
\lesssim
L^{\alpha C_0/2}
e^{-(\log 2)L^{C_0}/2}.
\]
Since
\[
|I_L|
=
L+C_0\log L
\lesssim L,
\]
we have
\[
\begin{aligned}
&\left\|
\int_{C_0\log L}^\infty
\widetilde k_\alpha(\,\cdot\,,s)u(s)\dd s
\right\|_{L^2(I_L)}\\
&\quad\leq
|I_L|^{1/2}
\sup_{t\in I_L}
\int_{C_0\log L}^\infty
\widetilde k_\alpha(t,s)|u(s)|\dd s\\
&\quad\lesssim
L^{(1+\alpha C_0)/2}
e^{-(\log 2)L^{C_0}/2}.
\end{aligned}
\]
Since $C_0>1$,
\[
-\frac{\log 2}{2}L^{C_0}
+x_*L
+\frac{1+\alpha C_0}{2}\log L
\longrightarrow-\infty.
\]
Consequently,
\[
L^{(1+\alpha C_0)/2}
e^{-(\log 2)L^{C_0}/2}
=
o(e^{-x_*L}),
\]
and hence
\[
\left\|
\int_{C_0\log L}^\infty
\widetilde k_\alpha(\,\cdot\,,s)u(s)\dd s
\right\|_{L^2(I_L)}
=
o(e^{-x_*L}).
\]
Combining this estimate with the estimate for the left tail, we obtain
\[
\|r_L\|_{L^2(I_L)}
\lesssim e^{-x_*L},
\]
which proves \eqref{eq:32}. Thus $u_L$ satisfies the eigenvalue equation for $K_{\alpha,L}$ with an error
that tends to zero exponentially as $L\to\infty$.

Set
\[
E_L=E_{[\lambda-\varepsilon/2,\lambda+\varepsilon/2]}(K_{\alpha,L}),
\qquad
v_L=E_Lu_L.
\]
For $g\in\operatorname{Ran}(I-E_L)$,
\[
\|(K_{\alpha,L}-\lambda)g\|_2\ge\frac\varepsilon2\|g\|_2.
\]
Taking $g=(I-E_L)u_L$ and using \eqref{eq:32}, we obtain
\begin{equation}\label{eq:33}
\|v_L-u_L\|_2\le\frac2\varepsilon\|r_L\|_2
\lesssim_\varepsilon e^{-x_*L}.
\end{equation}
Together with $\|u_L\|_2\to\|u\|_2>0$, this implies that
$v_L\ne0$ for all sufficiently large $L$.

For a real function $v$, let $S(v)$ be the supremum of the integers $m$ for which there exist $t_0<\cdots<t_m$ such that all $v(t_j)\ne0$ and their signs alternate. 
For $h>0$, put
\[
G_h(t)=\frac{1}{\sqrt{4\pi h}}e^{-t^2/(4h)}.
\]
The Gaussian $G_h$ is a Pólya frequency function. Schoenberg's theorem 
\cite[Theorem~1 and Section~2, p.~99]{Schoenberg1950} gives
\begin{equation}\label{eq:36}
S(G_h*V)\leq S(V)
\end{equation}
for every  real-valued bounded measurable
function $V$.
We first compare $G_h*Q$ and $G_h*P$ as $t\to-\infty$. On the negative half-line, \eqref{eq:30} gives
\[
\begin{aligned}
&\left|\int_{-\infty}^0G_h(t-s)(Q(s)-P(s))\dd s\right|\\
&\qquad=e^{\kappa t}
\left|\int_{-\infty}^0G_h(t-s)e^{-\kappa(t-s)}W(s)\dd s\right|\\
&\qquad\le e^{\kappa t}
\left(\int_{\mathbb R}G_h(r)^2e^{-2\kappa r}\dd r\right)^{1/2}
\|W\|_{L^2(-\infty,0)}
=O(e^{\kappa t}).
\end{aligned}
\]
Moreover, $Q\in L^2(0,\infty)$, because $Q(t)=e^{-x_*t}u(t)$ and $u\in L^2(\mathbb R)$. Hence
\[
\left|\int_0^\infty G_h(t-s)Q(s)\dd s\right|
\le\|Q\|_{L^2(0,\infty)}
\left(\int_{-\infty}^tG_h(r)^2\dd r\right)^{1/2}
\longrightarrow0
\]
as $t\to-\infty$. The boundedness of $P$ also gives
\[
\left|\int_0^\infty G_h(t-s)P(s)\dd s\right|
\le\|P\|_\infty\int_{-\infty}^tG_h(r)\dd r
\longrightarrow0.
\]
Therefore
\begin{equation}\label{eq:34}
(G_h*Q)(t)-(G_h*P)(t)\longrightarrow0,
\qquad t\to-\infty.
\end{equation}

The convolution with $P$ is
\[
(G_h*P)(t)=
\sum_{j=1}^re^{-h\eta_j^2}
\left(A_je^{-i\eta_jt}+\overline{A_j}e^{i\eta_jt}\right).
\]
Since $0<\eta_1<\cdots<\eta_r$ and $A_1\ne0$, choose $h$ so large that
\[
2\sum_{j=2}^r|A_j|e^{-h(\eta_j^2-\eta_1^2)}<|A_1|.
\]
At every point satisfying
\[
\eta_1t-\arg A_1=k\pi,
\qquad k\in\mathbb Z,
\]
we then have
\[
\operatorname{sgn}(G_h*P)(t)=(-1)^k,
\qquad
|(G_h*P)(t)|\ge e^{-h\eta_1^2}|A_1|.
\]
By \eqref{eq:34}, choose $T>0$ so that, for $t\le-T$,
\[
|(G_h*Q)(t)-(G_h*P)(t)|
<\frac12e^{-h\eta_1^2}|A_1|.
\]
Thus, at all the preceding phase points with $t\le-T$,
\[
\operatorname{sgn}(G_h*Q)(t)=(-1)^k,
\qquad
|(G_h*Q)(t)|\ge\frac12e^{-h\eta_1^2}|A_1|.
\]

Define on $\mathbb R$
\[
V_L(t)=e^{-x_*t}v_L(t)\mathbf 1_{I_L}(t).
\]
For every $w\in L^2(I_L)$, \eqref{eq:31} and the
Cauchy--Schwarz inequality give
\[
\begin{aligned}
  |(K_{\alpha,L}w)(t)|
  &\leq
  \left(
    \int_{I_L}|\widetilde k_\alpha(t,s)|^2\dd s
  \right)^{1/2}
  \|w\|_2\\
  &\lesssim
  \left(
    \int_{\mathbb R}e^{-\alpha|t-s|}\dd s
  \right)^{1/2}
  \|w\|_2
  =
  \left(\frac{2}{\alpha}\right)^{1/2}\|w\|_2.
\end{aligned}
\]
Hence
\[
  \|K_{\alpha,L}w\|_\infty\lesssim\|w\|_2
\]
with a constant independent of $L$.
Since the spectrum of $K_{\alpha,L}|_{\operatorname{Ran}E_L}$ is contained in
$[\lambda-\varepsilon/2,\lambda+\varepsilon/2]$ and
$\lambda-\varepsilon/2>0$, this restriction is invertible and
\[
  \left\|
    (K_{\alpha,L}|_{\operatorname{Ran}E_L})^{-1}
  \right\|_{2\to2}
  \leq
  (\lambda-\varepsilon/2)^{-1}.
\]
For $v_L\in\operatorname{Ran}E_L$, these bounds give
\[
\begin{aligned}
  \|v_L\|_\infty
  &=
  \left\|
    K_{\alpha,L}(K_{\alpha,L}|_{\operatorname{Ran}E_L})^{-1}v_L
  \right\|_\infty\\
  &\lesssim
  \left\|
    (K_{\alpha,L}|_{\operatorname{Ran}E_L})^{-1}v_L
  \right\|_2\\
  &\leq
  \frac{\|v_L\|_2}{\lambda-\varepsilon/2}
  \leq
  \frac{\|u\|_2}{\lambda-\varepsilon/2}.
\end{aligned}
\]
Thus $v_L$ is bounded and therefore
\[
  V_L(t)=e^{-x_*t}v_L(t)\mathbf 1_{I_L}(t)
\]
is a compactly supported real bounded measurable function, as
required in \eqref{eq:36}.

For all sufficiently large \(L\), put
\[
J_L=[-(1-2\delta)L,-T].
\]
We claim that
\begin{equation}\label{eq:35}
  \sup_{t\in J_L}
  |(G_h*V_L)(t)-(G_h*Q)(t)|
  \longrightarrow 0,
  \qquad L\to\infty.
\end{equation}
By \eqref{eq:33}, uniformly in $t$,
\[
\begin{aligned}
&\int_{-(1-\delta)L}^{C_0\log L}
G_h(t-s)|V_L(s)-Q(s)|\dd s\\
&\qquad\le
e^{x_*(1-\delta)L}\|G_h\|_2\|v_L-u_L\|_2
\lesssim_\varepsilon e^{-\delta x_*L}.
\end{aligned}
\]
If $t\in J_L$ and $s\le-(1-\delta)L$, then
\[
t-s\ge -(1-2\delta)L+(1-\delta)L=\delta L.
\]
Using $\|v_L\|_2\le\|u_L\|_2\le\|u\|_2$, we obtain
\[
\begin{aligned}
&\int_{-L}^{-(1-\delta)L}
G_h(t-s)|V_L(s)|\dd s\\
&\qquad\le
e^{x_*L}\|v_L\|_2
\left(\int_{|r|\ge\delta L}G_h(r)^2\dd r\right)^{1/2}
\longrightarrow0
\end{aligned}
\]
uniformly for $t\in J_L$, by the Gaussian decay of $G_h$.
Similarly, \eqref{eq:30} gives
\[
\begin{aligned}
&\int_{-\infty}^{-(1-\delta)L}
G_h(t-s)|Q(s)|\dd s\\
&\qquad\le
\|P\|_\infty\int_{|r|\ge\delta L}G_h(r)\dd r\\
&\qquad\quad+
e^{-\kappa(1-\delta)L}
\|G_h\|_2\|W\|_{L^2(-\infty,0)},
\end{aligned}
\]
which tends to zero uniformly on $J_L$.
Finally, $s\ge C_0\log L$ and $t\in J_L$ imply
$s-t\ge C_0\log L+T$, so
\[
\begin{aligned}
&\int_{C_0\log L}^\infty G_h(t-s)|Q(s)|\dd s\\
&\qquad\le
\|Q\|_{L^2(0,\infty)}
\left(
\int_{|r|\ge C_0\log L+T}G_h(r)^2\dd r
\right)^{1/2}
\longrightarrow0
\end{aligned}
\]
uniformly on $J_L$. These estimates prove \eqref{eq:35}.

Combining \eqref{eq:34} and \eqref{eq:35}, choose $L$ so large that
\[
\sup_{t\in J_L}|(G_h*V_L)(t)-(G_h*Q)(t)|
<\frac14e^{-h\eta_1^2}|A_1|.
\]
Then $G_h*V_L$ has alternating signs at every consecutive point in $J_L$ satisfying
\[
\eta_1t-\arg A_1=k\pi.
\]
The spacing of these points is $\pi/\eta_1$, while $|J_L|=(1-2\delta)L-T$. Hence
\[
S(G_h*V_L)\ge\frac{(1-2\delta)\eta_1}{\pi}L+O(1).
\]
Multiplication by $e^{-x_*t}$ preserves signs on $I_L$, and zero
values outside $I_L$ cannot occur in a strict sign-alternating
sequence. Thus
\[
S(V_L)=S(v_L).
\]
Together with \eqref{eq:36}, this gives
\begin{equation}\label{eq:37}
S(v_L)\ge\frac{(1-2\delta)\eta_1}{\pi}L+O(1).
\end{equation}

We now estimate $S(v_L)$ from above. The kernel of $K_\alpha$ is
\[
  k_\alpha(x,y)=(xy)^\rho\phi_\alpha(x+y),
  \qquad x,y>0.
\]
Its strict total positivity is established in
\eqref{eq:strict-total-positivity-kalpha} below. The identity
\[
  \widetilde k_\alpha(t,s)
  =e^{(t+s)/2}k_\alpha(e^t,e^s),
\]
gives
\[
\begin{aligned}
  \det\bigl(\widetilde k_\alpha(t_i,s_j)\bigr)_{i,j=1}^m
  &=
  e^{\frac12\sum_{i=1}^m t_i+\frac12\sum_{j=1}^m s_j}
  \det\bigl(k_\alpha(e^{t_i},e^{s_j})\bigr)_{i,j=1}^m\\
  &>0
\end{aligned}
\]
whenever $t_1<\cdots<t_m$ and $s_1<\cdots<s_m$ belong to $I_L$.

Let $e_0,e_1,\ldots$ be real continuous eigenfunctions of $K_{\alpha,L}$,
ordered by decreasing positive eigenvalue, and let $Z$ denote
the number of distinct zeros in the interior of $I_L$.
The kernel is continuous, symmetric, and strictly totally positive,
and hence is a symmetric oscillating kernel. Therefore
\cite[Chapter~1, Theorem~6.2(b), p.~35]{Karlin1968} gives
\[
  m\le S\left(\sum_{j=m}^n b_j e_j\right)
  \le Z\left(\sum_{j=m}^n b_j e_j\right)\le n
\]
for $0\le m\le n$ and real coefficients $b_m,\ldots,b_n$
that are not all zero.
 Since
\[
  \operatorname{Ran}E_L
  \subset\operatorname{span}
  \{e_0,\ldots,e_{N_L(\lambda-\varepsilon)-1}\}
\]
and $v_L\ne0$, the preceding inequality with $m=0$ gives
\begin{equation}\label{eq:39}
  S(v_L)\le N((\lambda-\epsilon,\infty);K_{\alpha,L})-1.
\end{equation}

It remains to establish the asymptotic formula for $N((\nu,\infty);K_{\alpha,L})$.
Let $C_L$ be the integral operator on $L^2(I_L)$ with kernel
\[
  \kappa_\alpha(t-s)
  =\left(2\cosh\frac{t-s}{2}\right)^{-\alpha}.
\]
Its Fourier transform is
\[
  \int_{\mathbb R}\kappa_\alpha(t)e^{-i\xi t}\dd t
  =\frac{|\Gamma(\alpha/2+i\xi)|^2}{\Gamma(\alpha)}
  =m_\rho(\xi).
\]
Let
\[
  \ell_L=|I_L|=L+C_0\log L.
\] For every fixed integer $k\ge2$, the operator $C_L^k$ is trace class and
\[
\operatorname{Tr}C_L^k
=\int_{I_L^k}\kappa_\alpha(t_1-t_2)\cdots
\kappa_\alpha(t_k-t_1)\dd t_1\cdots\dd t_k.
\]
By translation, replace $I_L$ with $(0,\ell_L)$. Put
\[
u_j=t_{j+1}-t_j,
\qquad1\le j\le k-1,
\]
and
\[
s_0=0,
\qquad
s_j=u_1+\cdots+u_j,
\qquad1\le j\le k-1.
\]
For fixed $u=(u_1,\ldots,u_{k-1})$, the conditions
\[
0<t_1+s_j<\ell_L,
\qquad0\le j\le k-1,
\]
leave an interval of admissible values of $t_1$ of length
\[
(\ell_L-R(u))_+,
\qquad
R(u)=\max_{0\le j\le k-1}s_j-
\min_{0\le j\le k-1}s_j.
\]
Using the evenness of $\kappa_\alpha$, we obtain
\[
\operatorname{Tr}C_L^k
=\int_{\mathbb R^{k-1}}(\ell_L-R(u))_+
\left(\prod_{j=1}^{k-1}\kappa_\alpha(u_j)\right)
\kappa_\alpha(s_{k-1})\dd u.
\]
The same symmetry gives the following expression for the $k$-fold
convolution of $\kappa_\alpha$:
\[
  (\kappa_\alpha^{*k})(0)
  =
  \int_{\mathbb R^{k-1}}
  \left(
    \prod_{j=1}^{k-1}\kappa_\alpha(u_j)
  \right)
  \kappa_\alpha(s_{k-1})\dd u.
\]
Moreover,
\[
  0\leq
  \ell_L-(\ell_L-R(u))_+
  \leq R(u).
\]
It follows that
\[
\begin{aligned}
0
&\leq
\ell_L(\kappa_\alpha^{*k})(0)
-\operatorname{Tr}C_L^k\\
&\leq
\int_{\mathbb R^{k-1}}
R(u)
\left(
  \prod_{j=1}^{k-1}\kappa_\alpha(u_j)
\right)
\kappa_\alpha(s_{k-1})\dd u.
\end{aligned}
\]
Since
\[
  R(u)\leq\sum_{j=1}^{k-1}|u_j|,
  \qquad
  \kappa_\alpha(t)\leq e^{-\alpha|t|/2}\leq1,
\]
the last integral is bounded by
\[
\begin{aligned}
&\int_{\mathbb R^{k-1}}
\left(\sum_{j=1}^{k-1}|u_j|\right)
\prod_{j=1}^{k-1}e^{-\alpha|u_j|/2}\dd u\\
&\quad=
(k-1)
\left(
  \int_{\mathbb R}|u|e^{-\alpha|u|/2}\dd u
\right)
\left(
  \int_{\mathbb R}e^{-\alpha|u|/2}\dd u
\right)^{k-2}
<\infty.
\end{aligned}
\]
This bound is independent of $L$. Therefore
\[
  \operatorname{Tr}C_L^k
  =
  \ell_L(\kappa_\alpha^{*k})(0)+O(1).
\]
Fourier inversion and $\ell_L/L\to1$ give
\begin{equation}\label{eq:40}
\frac1L\operatorname{Tr}C_L^k
\longrightarrow\frac1{2\pi}\int_{\mathbb R}m_\rho(\xi)^k\dd\xi,
\qquad k\ge2.
\end{equation}

The convolution operator on $L^2(\mathbb R)$ with kernel
$\kappa_\alpha(t-s)$ has Fourier multiplier $m_\rho$, so its
compression $C_L$ satisfies
\[
0\leq C_L\leq\Lambda_\alpha I.
\]
Let $\varphi\in C([0,\Lambda_\alpha])$ vanish on a neighborhood of
$0$. Given $\eta>0$, polynomial approximation of
$\varphi(x)/x^2$, extended by $0$ at $x=0$, gives a polynomial
$q_\eta$ such that
\[
|\varphi(x)-x^2q_\eta(x)|\leq\eta x^2,
\qquad 0\leq x\leq\Lambda_\alpha.
\]
The spectral theorem gives
\[
\left|
\operatorname{Tr}\varphi(C_L)
-\operatorname{Tr}\bigl(C_L^2q_\eta(C_L)\bigr)
\right|
\leq\eta\operatorname{Tr}C_L^2.
\]
By \eqref{eq:40}, $\operatorname{Tr}C_L^2=O(L)$ and
\[
\lim_{L\to\infty}
\frac1L\operatorname{Tr}\bigl(C_L^2q_\eta(C_L)\bigr)
=
\frac1{2\pi}\int_{\mathbb R}
m_\rho(\xi)^2q_\eta(m_\rho(\xi))\dd\xi.
\]
Also,
\[
\begin{aligned}
&\int_{\mathbb R}
\left|\varphi(m_\rho(\xi))
-m_\rho(\xi)^2q_\eta(m_\rho(\xi))\right|\dd\xi\\
&\qquad\leq\eta\int_{\mathbb R}m_\rho(\xi)^2\dd\xi.
\end{aligned}
\]
Letting first $L\to\infty$ and then $\eta\downarrow0$, we obtain
\begin{equation}\label{eq:42}
\frac1L\operatorname{Tr}\varphi(C_L)
\longrightarrow
\frac1{2\pi}\int_{\mathbb R}\varphi(m_\rho(\xi))\dd\xi.
\end{equation}

Fix $0<\nu<\Lambda_\alpha$ and
$0<d<\min\{\nu/2,\Lambda_\alpha-\nu\}$.
Choose continuous functions $\varphi_d^-,\varphi_d^+$ on
$[0,\Lambda_\alpha]$ such that
\[
0\leq\varphi_d^-\leq\mathbf1_{(\nu,\infty)}
\leq\varphi_d^+\leq1
\quad\text{on }[0,\Lambda_\alpha],
\]
where $\varphi_d^-=0$ on $[0,\nu]$,
$\varphi_d^-=1$ on $[\nu+d,\Lambda_\alpha]$,
$\varphi_d^+=0$ on $[0,\nu-d]$, and
$\varphi_d^+=1$ on $[\nu,\Lambda_\alpha]$. Then
\[
\operatorname{Tr}\varphi_d^-(C_L)
\leq N((\nu,\infty);C_L)
\leq\operatorname{Tr}\varphi_d^+(C_L).
\]
By \eqref{eq:42},
\[
\begin{aligned}
\frac1{2\pi}\int_{\mathbb R}\varphi_d^-(m_\rho(\xi))\dd\xi
&\leq\liminf_{L\to\infty}
\frac{N((\nu,\infty);C_L)}L\\
&\leq\limsup_{L\to\infty}
\frac{N((\nu,\infty);C_L)}L\\
&\leq\frac1{2\pi}\int_{\mathbb R}
\varphi_d^+(m_\rho(\xi))\dd\xi.
\end{aligned}
\]
The functions $\varphi_d^\pm(m_\rho(\xi))$ are bounded by
$\mathbf1_{\{m_\rho(\xi)>\nu/2\}}$, which has bounded support.
The evenness of $m_\rho$ and its strict decrease on $(0,\infty)$
give
\[
\{\xi:m_\rho(\xi)>\nu\}=(-\xi_\nu,\xi_\nu),
\qquad
\{\xi:m_\rho(\xi)=\nu\}=\{-\xi_\nu,\xi_\nu\}.
\]
Dominated convergence as $d\downarrow0$ therefore yields
\begin{equation}\label{eq:model-counting-asymptotics}
N((\nu,\infty);C_L)
=
\frac{\xi_\nu}{\pi}L+o(L),
\qquad 0<\nu<\Lambda_\alpha.
\end{equation}

The kernel of $K_{\alpha,L}-C_L$ is
\[
e^{\alpha(t+s)/2}r_\alpha(e^t+e^s).
\]
The estimates established earlier give
\[
|r_\alpha(q)|\lesssim q^{1-\alpha},
\qquad0<q\le1,
\]
and
\[
|r_\alpha(q)|\lesssim q^{-\alpha},
\qquad q\ge1.
\]
Put
\[
\sigma=\frac{t+s}{2},
\qquad
\tau=t-s.
\]
Then
\[
e^t+e^s=2e^\sigma\cosh(\tau/2),
\]
and, since $t,s\in I_L$,
\[
|\tau|\le L+C_0\log L,
\qquad
-L+\frac{|\tau|}{2}\le\sigma\le
C_0\log L-\frac{|\tau|}{2}.
\]
Consequently,
\[
\begin{aligned}
\|K_{\alpha,L}-C_L\|_{\mathbf S_2}^2
&=\int_{|\tau|\le L+C_0\log L}
\int_{-L+|\tau|/2}^{C_0\log L-|\tau|/2}
e^{2\alpha\sigma}
\left|r_\alpha\left(2e^\sigma\cosh(\tau/2)\right)\right|^2
\dd\sigma\dd\tau.
\end{aligned}
\]
For fixed $\tau$, the condition
\[
  2e^\sigma\cosh(\tau/2)\leq1
\]
is equivalent to
\[
  \sigma\leq-\log\left(2\cosh\frac{\tau}{2}\right).
\]
On this interval,
\[
\begin{aligned}
&e^{2\alpha\sigma}
\left|
  r_\alpha\left(2e^\sigma\cosh(\tau/2)\right)
\right|^2\\
&\qquad\lesssim
e^{2\sigma}
\bigl(\cosh(\tau/2)\bigr)^{2-2\alpha}.
\end{aligned}
\]
Therefore
\[
\begin{aligned}
&\int_{\substack{
-L+|\tau|/2\leq\sigma\leq C_0\log L-|\tau|/2\\
2e^\sigma\cosh(\tau/2)\leq1}}
e^{2\alpha\sigma}
\left|
  r_\alpha\left(2e^\sigma\cosh(\tau/2)\right)
\right|^2\dd\sigma\\
&\qquad\lesssim
\int_{-\infty}^{-\log(2\cosh(\tau/2))}
e^{2\sigma}
\bigl(\cosh(\tau/2)\bigr)^{2-2\alpha}\dd\sigma\\
&\qquad\lesssim
\bigl(\cosh(\tau/2)\bigr)^{-2\alpha}.
\end{aligned}
\]
On the complementary interval,
\[
  2e^\sigma\cosh(\tau/2)\geq1,
\]
we have
\[
e^{2\alpha\sigma}
\left|
  r_\alpha\left(2e^\sigma\cosh(\tau/2)\right)
\right|^2
\lesssim
\bigl(\cosh(\tau/2)\bigr)^{-2\alpha}.
\]
If this interval is non-empty, its length is bounded by
\[
\begin{aligned}
&\left(C_0\log L-\frac{|\tau|}{2}\right)
+\log\left(2\cosh\frac{\tau}{2}\right)\\
&\qquad=
C_0\log L+\log(1+e^{-|\tau|})
\leq C_0\log L+\log2,
\end{aligned}
\]
where we used
\[
  2\cosh(\tau/2)
  =e^{|\tau|/2}(1+e^{-|\tau|}).
\]
Consequently, for every
$|\tau|\leq L+C_0\log L$,
\[
\begin{aligned}
&\int_{-L+|\tau|/2}^{C_0\log L-|\tau|/2}
e^{2\alpha\sigma}
\left|
  r_\alpha\left(2e^\sigma\cosh(\tau/2)\right)
\right|^2\dd\sigma\\
&\qquad\lesssim
(1+\log L)
\bigl(\cosh(\tau/2)\bigr)^{-2\alpha}.
\end{aligned}
\]
Since
\[
  \int_{\mathbb R}
  \bigl(\cosh(\tau/2)\bigr)^{-2\alpha}\dd\tau<\infty,
\]
we may extend the non-negative $\tau$-integral to $\mathbb R$ and
obtain
\begin{equation}\label{eq:44}
\begin{aligned}
  \|K_{\alpha,L}-C_L\|_{\mathbf S_2}^2
  &\lesssim
  (1+\log L)
  \int_{\mathbb R}
  \bigl(\cosh(\tau/2)\bigr)^{-2\alpha}\dd\tau\\
  &\lesssim1+\log L=o(L).
\end{aligned}
\end{equation}

Put $D_L=K_{\alpha,L}-C_L$. For every fixed $\gamma>0$,
\eqref{eq:44} gives
\[
N((\gamma,\infty);|D_L|)
\leq\gamma^{-2}\|D_L\|_{\mathbf S_2}^2
=O(1+\log L)=o(L).
\]
Fix $0<\nu<\Lambda_\alpha$ and
$0<\gamma<\min\{\nu,\Lambda_\alpha-\nu\}$.
For every
\[
f\in
\operatorname{Ran}E_{(-\infty,\nu-\gamma]}(C_L)
\cap
\operatorname{Ran}E_{[0,\gamma]}(|D_L|),
\]
we have
\[
\begin{aligned}
\langle K_{\alpha,L}f,f\rangle
&=\langle C_Lf,f\rangle+\langle D_Lf,f\rangle\\
&\leq(\nu-\gamma)\|f\|_2^2
+\langle |D_L|f,f\rangle\\
&\leq\nu\|f\|_2^2.
\end{aligned}
\]
The codimension of this intersection in $L^2(I_L)$ is at most
\[
N((\nu-\gamma,\infty);C_L)
+N((\gamma,\infty);|D_L|).
\]
The min--max principle therefore gives
\[
N((\nu,\infty);K_{\alpha,L})
\leq N((\nu-\gamma,\infty);C_L)
+N((\gamma,\infty);|D_L|).
\]
Applying the same argument to $C_L=K_{\alpha,L}-D_L$
at the level $\nu+\gamma$ gives
\[
N((\nu+\gamma,\infty);C_L)
\leq N((\nu,\infty);K_{\alpha,L})+N((\gamma,\infty);|D_L|).
\]
Dividing by $L$ and using
\eqref{eq:model-counting-asymptotics}, we obtain
\[
\begin{aligned}
\frac{\xi_{\nu+\gamma}}{\pi}
&\leq\liminf_{L\to\infty}\frac{N((\nu,\infty);K_{\alpha,L})}L\\
&\leq\limsup_{L\to\infty}\frac{N((\nu,\infty);K_{\alpha,L})}L
\leq\frac{\xi_{\nu-\gamma}}{\pi}.
\end{aligned}
\]
Letting $\gamma\downarrow0$ and using the continuity of
$\nu\mapsto\xi_\nu$, we conclude that
\begin{equation}\label{eq:43}
N((\nu,\infty);K_{\alpha,L})
=
\frac{\xi_\nu}{\pi}L+o(L),
\qquad L\to\infty.
\end{equation}

Combining \eqref{eq:37}, \eqref{eq:39}, and \eqref{eq:43}
with $\nu=\lambda-\varepsilon$ gives
\[
\begin{aligned}
\frac{(1-2\delta)\eta_1}{\pi}L+O(1)
&\leq S(v_L)\\
&\leq N((\lambda-\epsilon,\infty);K_{\alpha,L})-1\\
&=\frac{\xi_{\lambda-\varepsilon}}{\pi}L+o(L),
\end{aligned}
\]
which contradicts \eqref{eq:26} as $L\to\infty$.
Thus alternative (ii) is impossible.

Both alternatives in \eqref{eq:7} have been excluded,
contradicting \eqref{eq:1}. Therefore
\[
\sigma_{\mathrm p}(K_\alpha)\cap(0,\Lambda_\alpha]
=\varnothing.
\]
By Theorem~\ref{thm:Ta-unitary-Ka}, the same conclusion holds
for $T_\alpha$, and $\ker T_\alpha=\{0\}$. Hence
\[
\sigma_{\mathrm p}(T_\alpha)\cap[0,\Lambda_\alpha]
=\varnothing.
\]
\end{proof}

\section{Eigenvalues above the essential spectrum}

For the multiplicative Hilbert operator \(H\) in Theorem A,
\(\|H\|=\sup\sigma_{\mathrm{ess}}(H)=\pi\), so there are no eigenvalues
above its essential spectrum. For general \(\alpha>1\), the corresponding
equality need not hold, as the following example shows.

\begin{example}\label{ex:E-nonempty}
For $\alpha=2$, the operator $K_2$ satisfies
\[
\|K_2\|>\Lambda_2=1.
\]
Thus the set
\[
E:=\{\alpha>1:\|K_\alpha\|>\Lambda_\alpha\}
\]
is non-empty.
\end{example}

\begin{proof}
For $\varepsilon>0$, set
\[
f_\varepsilon(x)
=
x^{-1/2+\varepsilon}e^{-2x}.
\]
Then
\[
\|f_\varepsilon\|_2^2
=
\Gamma(2\varepsilon)4^{-2\varepsilon}.
\]
Since
\[
k_2(x,y)
=
\sqrt{xy}\,
\bigl(\zeta(1+x+y)^2-1\bigr),
\]
the change of variables
\[
t=x+y,
\qquad
v=\frac{x}{x+y}
\]
gives
\[
\langle K_2f_\varepsilon,f_\varepsilon\rangle
=
\frac{\Gamma(1+\varepsilon)^2}
{\Gamma(2+2\varepsilon)}
\int_0^\infty
t^{1+2\varepsilon}
\bigl(\zeta(1+t)^2-1\bigr)e^{-2t}\dd t.
\]
Put
\[
J_\varepsilon
=
\int_0^\infty
t^{2\varepsilon}
\left[
t\bigl(\zeta(1+t)^2-1\bigr)-\frac1t
\right]e^{-2t}\dd t.
\]
Then
\[
\int_0^\infty
t^{1+2\varepsilon}
\bigl(\zeta(1+t)^2-1\bigr)e^{-2t}\dd t
=
\Gamma(2\varepsilon)2^{-2\varepsilon}
+
J_\varepsilon,
\]
and hence
\[
\frac{\langle K_2f_\varepsilon,f_\varepsilon\rangle}
{\|f_\varepsilon\|_2^2}
=
\frac{\Gamma(1+\varepsilon)^2}
{\Gamma(2+2\varepsilon)}
\left[
2^{2\varepsilon}
+
\frac{4^{2\varepsilon}}{\Gamma(2\varepsilon)}
J_\varepsilon
\right].
\]

Moreover,
\[
t\bigl(\zeta(1+t)^2-1\bigr)-\frac1t
=
O(1),
\qquad t\downarrow0,
\]
and
\[
t\bigl(\zeta(1+t)^2-1\bigr)-\frac1t
=
-\frac1t+O(t2^{-t}),
\qquad t\to\infty.
\]
Thus, by the dominated convergence theorem,
\[
J_\varepsilon
=
J_0+o(1),
\qquad
\varepsilon\downarrow0.
\]

For $t>0$, the function $x\mapsto x^{-1-t}$ is strictly convex.
Therefore
\[
\int_n^{n+1}x^{-1-t}\dd x
<
\frac12
\left(
n^{-1-t}+(n+1)^{-1-t}
\right),
\qquad n\ge1.
\]
Summing over $n$ yields
\[
\frac1t
<
\zeta(1+t)-\frac12.
\]
Consequently,
\[
t\bigl(\zeta(1+t)^2-1\bigr)-\frac1t
>
1-\frac34t,
\]
and hence
\[
J_0
>
\int_0^\infty
\left(1-\frac34t\right)e^{-2t}\dd t
=
\frac5{16}.
\]

Finally, as $\varepsilon\downarrow0$
\[
\frac{\Gamma(1+\varepsilon)^2}
{\Gamma(2+2\varepsilon)}
=
1-2\varepsilon+O(\varepsilon^2),
\]
\[
2^{2\varepsilon}
=
1+2\varepsilon\log2+O(\varepsilon^2),
\]
and
\[
\frac{4^{2\varepsilon}}{\Gamma(2\varepsilon)}
J_\varepsilon
=
2\varepsilon J_0+o(\varepsilon).
\]
Therefore
\[
\frac{\langle K_2f_\varepsilon,f_\varepsilon\rangle}
{\|f_\varepsilon\|_2^2}
=
1+
2\varepsilon
\bigl(\log2-1+J_0\bigr)
+o(\varepsilon).
\]
Since
\[
\log2-1+J_0
>
\log2-\frac{11}{16}
>0,
\]
we obtain, for all sufficiently small $\varepsilon>0$,
\[
\frac{\langle K_2f_\varepsilon,f_\varepsilon\rangle}
{\|f_\varepsilon\|_2^2}
>1=\Lambda_2.
\]
Thus
\[
\|K_2\|>\Lambda_2,
\]
and consequently $2\in E$.
\end{proof}

\subsection{Finiteness above the essential spectrum}

Retain the notation \(\rho\), \(\Lambda_\alpha\), and
\(K_\alpha\) from Section~3. Apply the
construction in the proof of Theorem~\ref{thm:weighted-Hankel-no-sc} to
\(K_\alpha\).  The endpoint values, verified in the proof of
Theorem~\ref{thm:complete-spectrum-Ta}, give
\[
  H_0=C_\rho\oplus0,
  \qquad T=H_0+K.
\]
Moreover, \(T\) is unitarily equivalent to two copies of
\(K_\alpha\).  We use
\[
  \psi_1(z)=\left(\frac{\Gamma'}{\Gamma}\right)'(z)
  =\sum_{n=0}^\infty\frac1{(n+z)^2},
  \qquad \Re z>0.
\]

We first establish the resolvent expansion at the upper endpoint that
is needed for the finiteness argument.

Let
\[ \ell(x)=(1+|\log x|)^{-2}, \]
and let \(L_0\) be multiplication by \(\ell\) on \(L^2(\mathbb R_+)\).  Define
\(L=L_0\oplus L_0\) on \(\mathcal H=L^2(\mathbb R_+)\oplus L^2(\mathbb R_+)\).

\begin{lemma}\label{lem:model-resolvent-threshold}
There is a bounded self-adjoint operator \(G_\alpha\) on \(\mathcal H\)
and a vector
\[ e= \bigl(x^{-1/2}(1+|\log x|)^{-2},0\bigr) \in\mathcal H \]
such that
\begin{equation}\label{eq:model-resolvent-expansion}
\lim_{\lambda\downarrow\Lambda_\alpha}
\left\|
\begin{aligned}
L(\lambda-H_0)^{-1}L
&-\frac{\langle\,\cdot\,,e\rangle e}
{2\sqrt{\nu_\alpha(\lambda-\Lambda_\alpha)}}-G_\alpha
\end{aligned}
\right\|=0,
\end{equation}
where
\[ \nu_\alpha = \Lambda_\alpha\psi_1\left(\frac\alpha2\right). \]
\end{lemma}

\begin{proof}
Logarithmic differentiation gives
\[
  \frac{m_\rho'(\xi)}{m_\rho(\xi)}
  =i\psi(\alpha/2+i\xi)-i\psi(\alpha/2-i\xi),
\]
and hence
\[
  m_\rho'(0)=0,
  \qquad
  m_\rho''(0)=-2\Lambda_\alpha\psi_1(\alpha/2).
\]
Since $m_\rho$ is even and analytic near $0$, Taylor's formula yields
\begin{equation}\label{eq:model-threshold-Taylor}
  m_\rho(\xi)
  =\Lambda_\alpha-\nu_\alpha\xi^2+O(\xi^4),
  \qquad
  \nu_\alpha=\Lambda_\alpha\psi_1(\alpha/2),
  \qquad \xi\to0.
\end{equation}

Fix \(\lambda>\Lambda_\alpha\),  for \(h\in L^1(\mathbb R)\cap L^2(\mathbb R)\), define
\[ \check h(t) = \frac1{2\pi} \int_{\mathbb R}e^{-it\xi}h(\xi)\dd\xi. \]
If \(f\in C_c^\infty(\mathbb R_+)\), the Mellin inversion formula and Fubini's
theorem give
\[
\begin{split}
  (\mathcal M^*M_h\mathcal Mf)(x)
  &=
  \int_0^\infty
  \frac1{2\pi\sqrt{xy}}
  \left(
    \int_{\mathbb R}
    \left(\frac{x}{y}\right)^{-i\xi}h(\xi)\dd\xi
  \right)f(y)\dd y\\
  &=
  \frac1{\sqrt x}
  \int_0^\infty
  \frac1{\sqrt y}
  \check h\left(\log\frac{x}{y}\right)
  f(y)\dd y.
\end{split}
\]
Consequently,
\begin{equation}\label{eq:weighted-Mellin-kernel}
\begin{aligned}
&(L_0\mathcal M^*M_h\mathcal ML_0f)(x)\\
&\quad=\frac{\ell(x)}{\sqrt x}
\int_0^\infty\frac{\ell(y)}{\sqrt y}
\check h\left(\log\frac{x}{y}\right)f(y)\dd y.
\end{aligned}
\end{equation}

By \eqref{eq:Mellin-model},
\[ (\lambda-C_\rho)^{-1}=\mathcal M^*M_{(\lambda-m_\rho)^{-1}}\mathcal M. \]
Moreover, \eqref{eq:model-threshold-Taylor} gives
\[ \lambda-m_\rho(\xi)=\lambda-\Lambda_\alpha+\nu_\alpha\xi^2+O(\xi^4), \qquad \xi\to0. \]
Define
\begin{equation}\label{eq:exact-resolvent-decomposition}
r_\lambda(\xi)
=
\frac1{\lambda-m_\rho(\xi)}
-\frac1{\lambda-\Lambda_\alpha+\nu_\alpha\xi^2}
-\frac1\lambda.
\end{equation}
For \(\xi\ne0\), set
\[ r_0(\xi) = \frac1{\Lambda_\alpha-m_\rho(\xi)} - \frac1{\nu_\alpha\xi^2} - \frac1{\Lambda_\alpha}. \]
Since \(m_\rho(\xi)\to0\) as \(|\xi|\to\infty\), \eqref{eq:exact-resolvent-decomposition} gives
\[ r_\lambda(\xi)\longrightarrow0, \qquad |\xi|\to\infty. \]
The value of \(r_0\) at \(\xi=0\) is irrelevant.
We claim that
\begin{equation}\label{eq:rlambda-L2-limit} \|r_\lambda-r_0\|_{L^2(\mathbb R)} \longrightarrow0, \qquad \lambda\downarrow\Lambda_\alpha.\end{equation}
For $0<|\xi|\leq1$, expansion
\eqref{eq:model-threshold-Taylor} gives
\[
|r_\lambda(\xi)|+|r_0(\xi)|
\lesssim
\frac{\xi^4}{(\lambda-\Lambda_\alpha+\xi^2)^2}+1
\lesssim1,
\]
uniformly for $0<\lambda-\Lambda_\alpha\leq1$.
For \(|\xi|\geq1\),
\[ r_\lambda(\xi) = \frac{m_\rho(\xi)} {\lambda(\lambda-m_\rho(\xi))} - \frac1{\lambda-\Lambda_\alpha+\nu_\alpha\xi^2}, \]
so, uniformly for \(0<\lambda-\Lambda_\alpha\leq1\),
\[ |r_\lambda(\xi)|+|r_0(\xi)| \lesssim m_\rho(\xi)+\xi^{-2}\bigr. \]
Stirling's formula yields
\[ m_\rho(\xi) = O\bigl(|\xi|^{\alpha-1}e^{-\pi|\xi|}\bigr), \qquad |\xi|\to\infty. \]
Pointwise convergence and the dominated convergence theorem prove
\eqref{eq:rlambda-L2-limit}.  By \eqref{eq:weighted-Mellin-kernel},
the integral kernel of
\[
  L_0\mathcal M^*M_{r_\lambda}\mathcal ML_0
  -
  L_0\mathcal M^*M_{r_0}\mathcal ML_0
\]
is
\[
  \frac{\ell(x)\ell(y)}{\sqrt{xy}}
  \left(
    \check r_\lambda\left(\log\frac{x}{y}\right)
    -
    \check r_0\left(\log\frac{x}{y}\right)
  \right).
\]
Changing variables $x=e^u$ and $y=e^v$, and then setting $t=u-v$,
we obtain
\[
\begin{aligned}
&\left\|
  L_0\mathcal M^*M_{r_\lambda}\mathcal ML_0
  -
  L_0\mathcal M^*M_{r_0}\mathcal ML_0
 \right\|_{\mathbf S_2}^2                                      \\
&\quad=
\int_{\mathbb R}\int_{\mathbb R}
\frac{
  |\check r_\lambda(u-v)-\check r_0(u-v)|^2
}{
  (1+|u|)^4(1+|v|)^4
}
\,\dd u\,\dd v                                                 \\
&\quad\leq
\left(
  \int_{\mathbb R}\frac{\dd v}{(1+|v|)^4}
\right)
\|\check r_\lambda-\check r_0\|_{L^2(\mathbb R)}^2              \\
&\quad=
\frac1{2\pi}
\left(
  \int_{\mathbb R}\frac{\dd v}{(1+|v|)^4}
\right)
\|r_\lambda-r_0\|_{L^2(\mathbb R)}^2
\longrightarrow0,
\end{aligned}
\]
where the last equality follows from the Plancherel theorem. Hence
\[
\left\|
L_0\mathcal M^*M_{r_\lambda-r_0}\mathcal ML_0
\right\|_{\mathbf S_2}
\longrightarrow0,
\qquad \lambda\downarrow\Lambda_\alpha.
\]

We now apply \eqref{eq:weighted-Mellin-kernel} to \(h(\xi)=(\lambda-\Lambda_\alpha+\nu_\alpha\xi^2)^{-1}\).  Its inverse Fourier transform is
\[ \check h(t) = \frac1{2\sqrt{\nu_\alpha(\lambda-\Lambda_\alpha)}} \exp\left( -\sqrt{\frac{\lambda-\Lambda_\alpha}{\nu_\alpha}}\,|t| \right). \]
Let
\[ e_1(x)=x^{-1/2}\ell(x). \]
Then
\[
\begin{split}
  &\bigl(
  L_0\mathcal M^*M_{(\lambda-\Lambda_\alpha+\nu_\alpha\xi^2)^{-1}}
  \mathcal ML_0f
  \bigr)(x)\\
  &\quad=
  \frac{e_1(x)}{2\sqrt{\nu_\alpha(\lambda-\Lambda_\alpha)}}
  \int_0^\infty
  e_1(y)
  \exp\left(
    -\sqrt{\frac{\lambda-\Lambda_\alpha}{\nu_\alpha}}
    \left|\log\frac{x}{y}\right|
  \right)
  f(y)\dd y.
\end{split}
\]
Define
\[
\begin{split}
  (R_\lambda f)(x)
  &=
  \frac{e_1(x)}{2\nu_\alpha}
  \int_0^\infty
  e_1(y)
  \frac{
    \exp\left(
      -\sqrt{(\lambda-\Lambda_\alpha)/\nu_\alpha}
      |\log(x/y)|
    \right)-1
  }{\sqrt{(\lambda-\Lambda_\alpha)/\nu_\alpha}}
  f(y)\dd y,\\
  (R_0f)(x)
  &=
  -\frac{e_1(x)}{2\nu_\alpha}
  \int_0^\infty
  e_1(y)
  \left|\log\frac{x}{y}\right|
  f(y)\dd y.
\end{split}
\]
It follows that
\begin{equation}\label{eq:quadratic-model-resolvent}
\begin{aligned}
&L_0\mathcal M^*M_{(\lambda-\Lambda_\alpha+\nu_\alpha\xi^2)^{-1}}
\mathcal ML_0=\frac{\langle\,\cdot\,,e_1\rangle e_1}
{2\sqrt{\nu_\alpha(\lambda-\Lambda_\alpha)}}+R_\lambda.
\end{aligned}
\end{equation}

For every \(a,t>0\),
\[ 0 \leq t-\frac{1-e^{-at}}a \leq t, \qquad \lim_{a\downarrow0} \left( t-\frac{1-e^{-at}}a \right)=0. \]
With
\[ a=\sqrt{\frac{\lambda-\Lambda_\alpha}{\nu_\alpha}}, \qquad t=\left|\log\frac{x}{y}\right|, \]
this gives pointwise convergence of the kernel of
\(R_\lambda-R_0\) to zero and the bound
\[
\left|\frac{e^{-at}-1}{a}+t\right|\leq t.
\]
Consequently,
\[
\begin{split}
  \|R_\lambda-R_0\|_{\mathbf S_2}^2
  &\leq
  \frac1{4\nu_\alpha^2}
  \int_0^\infty\int_0^\infty
  e_1(x)^2e_1(y)^2
  \left|\log\frac{x}{y}\right|^2
  \dd x\dd y.
\end{split}
\]
After the substitutions
\[ u=\log x, \qquad v=\log y, \qquad \frac{\dd x}{x}=\dd u, \qquad \frac{\dd y}{y}=\dd v, \]
the integral on the right becomes
\[ \int_{\mathbb R}\int_{\mathbb R} \frac{|u-v|^2} {(1+|u|)^4(1+|v|)^4} \dd u\dd v. \]
Since
\[ |u-v|^2\leq2u^2+2v^2, \]
this double integral is finite.  The dominated convergence theorem
therefore yields
\begin{equation}\label{eq:Rlambda-HS-limit} \lim_{\lambda\downarrow\Lambda_\alpha} \|R_\lambda-R_0\|_{\mathbf S_2}=0.\end{equation}

Combining \eqref{eq:quadratic-model-resolvent} and
\eqref{eq:exact-resolvent-decomposition}, we obtain
\[
\begin{split}
  L_0(\lambda-C_\rho)^{-1}L_0
  &=
  \frac{\langle\,\cdot\,,e_1\rangle e_1}
       {2\sqrt{\nu_\alpha(\lambda-\Lambda_\alpha)}}
  +R_\lambda
  +L_0\mathcal M^*M_{r_\lambda}\mathcal ML_0
  +\lambda^{-1}L_0^2.
\end{split}
\]
Define
\[ G_\alpha = \left( R_0+L_0\mathcal M^*M_{r_0}\mathcal ML_0+\Lambda_\alpha^{-1}L_0^2 \right) \oplus \Lambda_\alpha^{-1}L_0^2. \]
For \(H_0=C_\rho\oplus0\), the difference in
\eqref{eq:model-resolvent-expansion} equals
\[
\begin{split}
  &\Bigl(R_\lambda-R_0
  +
  L_0\mathcal M^*M_{r_\lambda-r_0}\mathcal ML_0
  +
  (\lambda^{-1}-\Lambda_\alpha^{-1})L_0^2\Bigr)\oplus
  \Bigl((\lambda^{-1}-\Lambda_\alpha^{-1})L_0^2\Bigr).
\end{split}
\]
Its norm tends to zero by \eqref{eq:Rlambda-HS-limit},
\eqref{eq:rlambda-L2-limit}, and
\(\lambda^{-1}\to\Lambda_\alpha^{-1}\).  This proves
\eqref{eq:model-resolvent-expansion}.
\end{proof}

With the endpoint expansion available, we first prove that only finitely many eigenvalues can lie above \(\Lambda_\alpha\).

\begin{proof}[Proof of the finiteness assertion in part~(4) of
Theorem~\ref{thm:complete-spectrum-Ta}]
\leavevmode\par
The kernel estimates obtained in the proof of
Theorem~\ref{thm:weighted-Hankel-no-sc}, together with
\[ t^\alpha\bigl(\zeta(1+t)^\alpha-1\bigr) = 1+O(t), \qquad t\downarrow0, \]
and the exponential decay of
\(\zeta(1+t)^\alpha-1\) as \(t\to\infty\), imply
\[ \sum_{j,k=1}^2 \int_0^\infty\int_0^\infty (1+|\log x|)^4(1+|\log y|)^4 |K_{jk}(x,y)|^2 \dd x\dd y <\infty. \]
Define \(W\) on \(\mathcal H\) by the kernel
\[ W_{jk}(x,y) = (1+|\log x|)^2K_{jk}(x,y)(1+|\log y|)^2. \]
Then \(W=W^*\) is Hilbert--Schmidt and
\begin{equation}\label{eq:K-factorization} K=LWL.\end{equation}
In view of \eqref{eq:K-factorization}, put
\[ W_+=\frac{|W|+W}{2}, \qquad T_+=H_0+LW_+L. \]
The inequality \(W\leq W_+\) implies \(T\leq T_+\).  For
\(\lambda>\Lambda_\alpha\), define
\[ B(\lambda) = W_+^{1/2}L(\lambda-H_0)^{-1}LW_+^{1/2}. \]
This operator is non-negative and compact.

For \(u\in\mathcal H\), put
\[ v=(\lambda-H_0)^{1/2}u, \qquad Y=W_+^{1/2}L(\lambda-H_0)^{-1/2}. \]
Then
\[ \langle(T_+-\lambda)u,u\rangle = \langle(Y^*Y-I)v,v\rangle, \]
\[ Y^*Y = (\lambda-H_0)^{-1/2}LW_+L(\lambda-H_0)^{-1/2}, \qquad YY^*=B(\lambda). \]
Since \(\lambda-H_0\) is boundedly invertible, the displayed quadratic-form identity and Glazman's lemma \cite[Theorem~1.25]{FrankLaptevWeidl2023} give
\[ N((\lambda,\infty);T_+)=N((1,\infty);Y^*Y). \]
The non-zero eigenvalues of \(Y^*Y\) and \(YY^*\) agree, including their multiplicities.  Using \(YY^*=B(\lambda)\), we obtain
\begin{equation}\label{eq:Birman-Schwinger-count} N((\lambda,\infty);T_+) = N((1,\infty);B(\lambda)).\end{equation}
Lemma~\ref{lem:model-resolvent-threshold} gives
\[
\lim_{\lambda\downarrow\Lambda_\alpha}
\left\|
\begin{aligned}
B(\lambda)-\frac{\langle\,\cdot\,,W_+^{1/2}e\rangle W_+^{1/2}e}
{2\sqrt{\nu_\alpha(\lambda-\Lambda_\alpha)}}-K_\alpha^{(0)}
\end{aligned}
\right\|=0,
\]
where
\[ K_\alpha^{(0)}=W_+^{1/2}G_\alpha W_+^{1/2}. \]
The operator \(K_\alpha^{(0)}\) is compact and self-adjoint.  Hence there is
\(\lambda_0>\Lambda_\alpha\) such that, whenever
\(\Lambda_\alpha<\lambda<\lambda_0\),
\[
\begin{aligned}
B(\lambda)\leq{}&
\frac{\langle\,\cdot\,,W_+^{1/2}e\rangle W_+^{1/2}e}
{2\sqrt{\nu_\alpha(\lambda-\Lambda_\alpha)}}+K_\alpha^{(0)}+\frac14I.
\end{aligned}
\]
By the rank-one eigenvalue-count estimate \cite[Lemma~3.2]{Widom1966},
\[
\begin{split}
  N((1,\infty);B(\lambda))
  &\leq
  N\left(
    (\frac34,\infty);
    \frac{
      \langle\,\cdot\,,W_+^{1/2}e\rangle
      W_+^{1/2}e
    }{2\sqrt{\nu_\alpha(\lambda-\Lambda_\alpha)}}
    +
    K_\alpha^{(0)}
  \right)\\
  &\leq
  1+N((\frac34,\infty);K_\alpha^{(0)}),
  \qquad
  \Lambda_\alpha<\lambda<\lambda_0.
\end{split}
\]
Together with \eqref{eq:Birman-Schwinger-count}, this yields
\[ N((\lambda,\infty);T_+) \leq 1+N((\frac34,\infty);K_\alpha^{(0)}), \qquad \Lambda_\alpha<\lambda<\lambda_0. \]
The right-hand side is finite because \(K_\alpha^{(0)}\) is compact.  Moreover,
\[ N((\Lambda_\alpha,\infty);T_+) = \sup_{\Lambda_\alpha<\lambda<\lambda_0} N((\lambda,\infty);T_+), \]
and hence
\[ N((\Lambda_\alpha,\infty);T_+) \leq 1+N((\frac34,\infty);K_\alpha^{(0)}) <\infty. \]
From $T\le T_+$, we obtain
\[
  N((\Lambda_\alpha,\infty);T)
  \le N((\Lambda_\alpha,\infty);T_+)<\infty.
\]
As $T$ is unitarily equivalent to $K_\alpha\oplus K_\alpha$,
\[
  N((\Lambda_\alpha,\infty);K_\alpha)
  =\frac12 N((\Lambda_\alpha,\infty);T)<\infty.
\]
The same conclusion holds for $T_\alpha$
by Theorem~\ref{thm:Ta-unitary-Ka}.
\end{proof}

\subsection{Simplicity above the essential spectrum}

Let \((X,\mu)\) be a measure space.  An operator \(S\) on
\(L^2(X,\mu)\) is called \emph{positivity improving} if
\[
  0\leq F\in L^2(X,\mu),\quad F\neq0
  \quad\Longrightarrow\quad
  SF>0\ \text{a.e.}
\]
For continuous symmetric strictly totally positive kernels on
compact intervals, the eigenvalues are positive and simple;
see \cite[Theorem~3.2]{Pinkus1996}.
We shall prove simplicity of the eigenvalues above the essential
spectrum for $K_\alpha$, which is not compact.
We begin with the following lemma.

\begin{lemma} \label{lem:simplicity-at-spectral-top}
Let \((X,\mu)\) be a measure space, and let \(S\) be a bounded
non-negative self-adjoint operator on \(L^2(X,\mu)\).  Assume that
\(S\) is positivity improving and that \(\eta=\|S\|\) is an
eigenvalue of \(S\).  Then
\[
  \dim\ker(S-\eta I)=1.
\]
\end{lemma}

\begin{proof}
If \(g\in L^2(X,\mu)\) is real-valued, then
\(g=g_+-g_-\), and the definition of positivity improving shows that
\(Sg_+\) and \(Sg_-\) are real-valued.  Thus \(Sg=Sg_+-Sg_-\) is
real-valued.  Hence, starting from any non-zero element of
\(\ker(S-\eta I)\) and taking its real or imaginary part, we may choose
a non-zero real-valued eigenfunction \(f\in\ker(S-\eta I)\).  The
lattice inequality gives
\[
  |Sf|\leq S|f|,
\]
and hence
\[
  \eta|f|\leq S|f|.
\]
On the other hand,
\[
  \|S|f|\|\leq\eta\||f|\|=\eta\|f\|.
\]
It follows that equality holds almost everywhere:
\[
  S|f|=\eta|f|.
\]
Applying the positivity-improving property to the non-zero function
\(|f|\geq0\) yields
\[
  \eta|f|=S|f|>0
  \quad\text{a.e.},
\]
and therefore \(|f|>0\) almost everywhere.

Suppose that the eigenspace has dimension at least two.  Choose a
non-zero eigenvector \(h\perp |f|\).  Since \(S\) preserves the real
subspace, one of \(\Re h\) and \(\Im h\) is a non-zero real
eigenvector still orthogonal to \(|f|\); replacing \(h\) by that
vector, we may assume that \(h\) is real.  The relations \(|f|>0\) almost everywhere and
\(\langle h,|f|\rangle=0\) imply that both \(h_+\) and \(h_-\)
are non-zero.  Hence
\[
  Sh_+>0,\qquad Sh_->0
  \quad\text{a.e.}
\]
Consequently,
\[
  |Sh|
  =|Sh_+-Sh_-|
  <Sh_++Sh_-
  =S|h|
  \quad\text{a.e.},
\]
whereas the preceding argument applied to \(h\) gives
\(S|h|=\eta|h|=|Sh|\).  This contradiction proves the assertion.
\end{proof}

For an integer \(r\geq1\), let \(\mathfrak S_r\) be the symmetric
group on \(\{1,\ldots,r\}\).  Define the antisymmetrization operator
\(\mathfrak A_r\) on \(\mathcal H^{\otimes r}\) by
\[
  \mathfrak A_r
  (f_1\otimes\cdots\otimes f_r)
  =
  \frac1{r!}
  \sum_{\sigma\in\mathfrak S_r}
  \operatorname{sgn}(\sigma)\,
  f_{\sigma(1)}\otimes\cdots\otimes f_{\sigma(r)}.
\]
Then \(\mathfrak A_r\) is an orthogonal projection, and the \(r\)-th
exterior power of \(\mathcal H\) is defined by
\[
  \bigwedge^r\mathcal H
  =
  \operatorname{Ran}\mathfrak A_r
  =
  \{F\in\mathcal H^{\otimes r}:
    \mathfrak A_rF=F\},
\]
we use the convention
\[
  \bigwedge^0\mathcal H=\mathbb C.
\]
For \(f_1,\ldots,f_r\in\mathcal H\), we use the normalized exterior
product
\[
\begin{split}
  f_1\wedge\cdots\wedge f_r
  &:=
  \sqrt{r!}\,
  \mathfrak A_r(f_1\otimes\cdots\otimes f_r)\\
  &=
  \frac1{\sqrt{r!}}
  \sum_{\sigma\in\mathfrak S_r}
  \operatorname{sgn}(\sigma)\,
  f_{\sigma(1)}\otimes\cdots\otimes f_{\sigma(r)}.
\end{split}
\]
With this normalization,
\[
  \left\langle
    f_1\wedge\cdots\wedge f_r,\,
    g_1\wedge\cdots\wedge g_r
  \right\rangle
  =
  \det\bigl(\langle f_i,g_j\rangle\bigr)_{i,j=1}^r.
\]

If \(A\) is a bounded operator on \(\mathcal H\), then
\(A^{\otimes r}\) leaves \(\bigwedge^r\mathcal H\) invariant.  Its
restriction
\[
  A^{\wedge r}
  :=
  A^{\otimes r}\big|_{\bigwedge^r\mathcal H}
\]
is called the \(r\)-th exterior power of \(A\), and
\[
  A^{\wedge r}(f_1\wedge\cdots\wedge f_r)
  =
  Af_1\wedge\cdots\wedge Af_r.
\]
We shall use the standard identities
\[
  (AB)^{\wedge r}=A^{\wedge r}B^{\wedge r},
  \qquad
  (A^*)^{\wedge r}=(A^{\wedge r})^*.
\]

For further background on antisymmetric tensor products and exterior powers of operators, see \cite[Section~1.5]{Simon2005}.

\begin{theorem}
\label{thm:exterior-power-simplicity}
Let \(\mathcal H\) be a complex Hilbert space and let \(A\) be a
bounded non-negative self-adjoint operator on \(\mathcal H\).  Put
\[
  \beta=\sup\sigma_{\mathrm{ess}}(A).
\]
Assume that for every integer \(r\geq1\) there exist a measure space
\((X_r,\mu_r)\) and a unitary operator
\[
  U_r:\bigwedge^r\mathcal H\longrightarrow L^2(X_r,\mu_r)
\]
such that
\[
  S_r=U_rA^{\wedge r}U_r^{-1}
\]
is positivity improving.  Then every eigenvalue \(\lambda>\beta\) of
\(A\) is simple.
\end{theorem}

\begin{proof}
Fix an eigenvalue \(\lambda>\beta\).  Since \(\lambda\) lies strictly
above the essential spectrum, there are only finitely many eigenvalues
of \(A\) in \((\lambda,\infty)\), counted with multiplicity.  Put
\[
  q=N((\lambda,\infty);A),
\]
and list them as
\[
  \mu_1\geq\mu_2\geq\cdots\geq\mu_q>\lambda.
\]
 Choose an orthonormal family \(e_1,\ldots,e_q\) such that
\[
  Ae_j=\mu_je_j,\qquad j=1,\ldots,q,
\]
and set
\[
  M=\operatorname{span}\{e_1,\ldots,e_q\},
  \qquad
  \mathcal N=M^\perp.
\]
Both \(M\) and \(\mathcal N\) reduce \(A\).  Since \(\mathcal N\)
contains no spectral point strictly larger than \(\lambda\), while
\(\lambda\) itself is an eigenvalue of \(A|_{\mathcal N}\), and
\(A\geq0\), we have
\[
  \|A|_{\mathcal N}\|=\lambda.
\]

Set \(r=q+1\).  The orthogonal decomposition
\(\mathcal H=M\oplus\mathcal N\) induces
\[
  \bigwedge^r\mathcal H
  =
  \bigoplus_{j=0}^q
  \left(\bigwedge^jM\right)
  \wedge
  \left(\bigwedge^{r-j}\mathcal N\right).
\]
On the summand containing \(j\) factors from \(M\), the norm of
\(A^{\wedge r}\) is bounded by
\[
  \mu_1\cdots\mu_j\lambda^{q+1-j}.
\]
For \(j<q\), this is strictly smaller than
\[
  R=\mu_1\cdots\mu_q\lambda.
\]
 If
\(0\neq f\in\ker(A-\lambda I)\), then
\[
  e_1\wedge\cdots\wedge e_q\wedge f\neq0
\]
and
\[
  A^{\wedge r}
  (e_1\wedge\cdots\wedge e_q\wedge f)
  =
  R(e_1\wedge\cdots\wedge e_q\wedge f).
\]
Hence
\[
  \|A^{\wedge r}\|=R,
\]
and \(R\) is an eigenvalue of \(A^{\wedge r}\).

We also record that \(A^{\wedge r}\geq0\).  Since \(A\geq0\), its
non-negative square root \(A^{1/2}\) is self-adjoint, and exterior
powers respect products and adjoints.  Therefore
\[
  A^{\wedge r}
  =
  \bigl((A^{1/2})^{\wedge r}\bigr)^*
  (A^{1/2})^{\wedge r}
  \geq0.
\]
Thus \(S_r\) is a bounded non-negative self-adjoint operator.  By
assumption it is positivity improving, and by unitary equivalence
\(\|S_r\|=R\) is an eigenvalue.  Lemma~\ref{lem:simplicity-at-spectral-top}
therefore implies that \(R\) is a simple eigenvalue of \(A^{\wedge r}\).

Assume now that \(\lambda\) is not simple. Choose linearly independent
\(f,g\in\ker(A-\lambda I)\). Since \(A\) is self-adjoint and
\(\lambda\neq\mu_j\) for \(j=1,\ldots,q\), we have
\[
  f,g\in \mathcal N=M^\perp.
\]
Hence
\[
  e_1\wedge\cdots\wedge e_q\wedge f,
  \qquad
  e_1\wedge\cdots\wedge e_q\wedge g
\]
are linearly independent eigenvectors of \(A^{\wedge r}\)
corresponding to \(R\), a contradiction.  Therefore
\[
  \dim\ker(A-\lambda I)=1.
\]
\end{proof}

\begin{proof}[Proof of the simplicity assertion in part~(4) of Theorem~\ref{thm:complete-spectrum-Ta}]
It remains to verify the
hypotheses of Theorem~\ref{thm:exterior-power-simplicity} for
\(K_\alpha\).  Write
\[
  \phi_\alpha(t)=\zeta(1+t)^\alpha-1
  =\sum_{n\geq2}\frac{d_\alpha(n)}{n}e^{-t\log n},
  \qquad t>0.
\]
All coefficients \(d_\alpha(n)\) are strictly positive.  Fix \(r\geq1\)
and
\[
  0<x_1<\cdots<x_r,\qquad 0<y_1<\cdots<y_r.
\]
The Cauchy--Binet formula
\cite[Section~0.8.7]{HornJohnson2013}, first for finite partial sums and
then by passage to the limit, gives
\[
\begin{split}
  \det\bigl(\phi_\alpha(x_i+y_j)\bigr)_{i,j=1}^r
  &=
  \sum_{2\leq n_1<\cdots<n_r}
  \left(\prod_{k=1}^r\frac{d_\alpha(n_k)}{n_k}\right)
  \det\bigl(e^{-x_i\log n_k}\bigr)_{i,k=1}^r
  \det\bigl(e^{-y_j\log n_k}\bigr)_{j,k=1}^r.
\end{split}
\]
The exponential kernel \(e^{xy}\) is strictly totally positive of all
orders \cite[p.~2]{JohnsonRichards2025}.  Thus, for increasing positive
\(z_i\) and increasing \(\tau_k\), taking \(x_i=-z_i\) and
reversing the order of the \(\tau_k\)'s gives 
\[
  (-1)^{r(r-1)/2}
  \det\bigl(e^{-z_i\tau_k}\bigr)_{i,k=1}^r>0.
\]
The two determinants in every summand therefore have the same
non-zero sign, so every summand is positive.  It follows that
\[
  \det\bigl(\phi_\alpha(x_i+y_j)\bigr)_{i,j=1}^r>0.
\]
Since
\[
  k_\alpha(x,y)=(xy)^\rho\phi_\alpha(x+y),
\]
multiplication of rows and columns by positive factors yields
\begin{equation}\label{eq:strict-total-positivity-kalpha}
  \det\bigl(k_\alpha(x_i,y_j)\bigr)_{i,j=1}^r>0.
\end{equation}

For every \(r\geq1\), let
\[
  X_r=\{(x_1,\ldots,x_r)\in\mathbb R_+^r:
  0<x_1<\cdots<x_r\},
\]
and use the standard unitary identification
\[
  U_r:\bigwedge^rL^2(\mathbb R_+)\longrightarrow L^2(X_r),
  \qquad
  (U_rF)(x)=\sqrt{r!}\,F(x).
\]
Set
\[
  S_r=U_rK_\alpha^{\wedge r}U_r^{-1},
\]
then we have
\[
  (S_rF)(x)=\int_{X_r}K_r(x,y)F(y)\dd y,
  \qquad
  K_r(x,y)=
  \det\bigl(k_\alpha(x_i,y_j)\bigr)_{i,j=1}^r.
\]
 By
\eqref{eq:strict-total-positivity-kalpha},
\[
  K_r(x,y)>0,\qquad x,y\in X_r.
\]
Hence \(S_r\) is positivity improving on  \(L^2(X_r)\).  
The operator \(K_\alpha\) is bounded, non-negative and
self-adjoint, and
\[
  \sup\sigma_{\mathrm{ess}}(K_\alpha)=\Lambda_\alpha.
\]
All hypotheses of Theorem~\ref{thm:exterior-power-simplicity} are
therefore satisfied with \(A=K_\alpha\) and
\(\beta=\Lambda_\alpha\).  Consequently, every eigenvalue of
\(K_\alpha\) in \((\Lambda_\alpha,\infty)\) is simple.

Finally, Theorem~\ref{thm:Ta-unitary-Ka} transfers simplicity, including
multiplicity, from \(K_\alpha\) to \(T_\alpha\).  Thus every
eigenvalue of \(T_\alpha\) above \(\Lambda_\alpha\) is simple.
\end{proof}

\subsection{Existence of eigenvalues and the critical parameter}

\begin{lemma}\label{lem:zeta-log-derivative-bound}
For every \(s>1\),
\begin{equation*}
  -\frac{\zeta'(s)}{\zeta(s)}
  \leq \zeta(s)-1.
\end{equation*}
\end{lemma}

\begin{proof}
By \cite[Theorem~3.1]{Alzer2012}, the function
\[
  s\longmapsto
  \left(1-\frac1{\zeta(s)}\right)^{1/(s-1)}
\]
is strictly increasing on \((1,\infty)\). Hence
\[
\begin{aligned}
  0
  &\leq
  \frac{\dd}{\dd s}
  \left[
    \frac{\log(1-1/\zeta(s))}{s-1}
  \right]\\
  &=
  \frac1{s-1}
  \left[
    \frac{\zeta'(s)}{\zeta(s)(\zeta(s)-1)}
    -
    \frac{\log(1-1/\zeta(s))}{s-1}
  \right].
\end{aligned}
\]
Moreover, \cite[Theorem~4.1, Eq.~(4.2)]{Alzer2012} gives
\[
  \left(1-\frac1{\zeta(s)}\right)^{1/(s-1)}
  >e^{-1},
  \qquad s>1.
\]
Therefore
\[
  \frac{\zeta'(s)}{\zeta(s)(\zeta(s)-1)}
  \geq
  \frac{\log(1-1/\zeta(s))}{s-1}
  >-1.
\]
Multiplying by \(\zeta(s)-1>0\), we obtain
\[
  -\frac{\zeta'(s)}{\zeta(s)}
  <\zeta(s)-1,
\]
which proves the lemma.
\end{proof}

We now determine exactly for which parameters such eigenvalues exist.

\begin{proof}[Proof of part~(5) of
Theorem~\ref{thm:complete-spectrum-Ta}]
We prove $\|K_\alpha\|=\Lambda_\alpha$ for $\alpha>1$ sufficiently
close to $1$. The inclusion
$\Lambda_\alpha\in\sigma_{\mathrm{ess}}(K_\alpha)$ gives
$\|K_\alpha\|\geq\Lambda_\alpha$. For the upper bound, we apply
Schur's test with $w(x)=x^{-1/2}$. As $k_\alpha$ is non-negative
and symmetric, it suffices to show that
\[
  S_\alpha(x):=
  \frac{1}{w(x)}
  \int_0^\infty k_\alpha(x,y)w(y)\dd y
  \leq\Lambda_\alpha,
  \qquad x>0.
\]

Set
\[
r_\alpha(t)
=
\zeta(1+t)^\alpha-1-t^{-\alpha},
\qquad
I_\alpha(x)
=
\int_0^\infty
y^{\alpha/2-1}r_\alpha(x+y)\dd y.
\]
Since
\[
x^{\alpha/2}
\int_0^\infty
y^{\alpha/2-1}(x+y)^{-\alpha}\dd y
=
B\left(\frac{\alpha}{2},\frac{\alpha}{2}\right)
=
\frac{\Gamma(\alpha/2)^2}{\Gamma(\alpha)},
\]
we obtain
\begin{equation}
S_\alpha(x)
=
\Lambda_\alpha
+
x^{\alpha/2}I_\alpha(x).
\label{eq:critical-Schur-function}
\end{equation}
We shall show that $I_\alpha(x)<0$ for every $x>0$ when $\alpha>1$ is
sufficiently close to $1$.

At $\alpha=1$,
\[
r_1(t)
=
\sum_{n=2}^\infty n^{-1-t}-t^{-1}<0,
\qquad t>0,
\]
because $u\mapsto u^{-1-t}$ is strictly decreasing and
\[
\sum_{n=2}^\infty n^{-1-t}
<
\int_1^\infty u^{-1-t}\dd u
=
\frac1t.
\]
Fix $A\in(1,2)$. Uniformly for $1\leq\alpha\leq A$, the Laurent expansion
of $\zeta$ at $1$ gives
\[
|r_\alpha(t)|
\lesssim t^{1-\alpha},
\qquad 0<t\leq1,
\]
whereas, for $t\geq1$,
\[
|r_\alpha(t)|
\lesssim t^{-\alpha}.
\]
Consequently,
\[
J_\alpha
:=
\int_0^\infty
y^{\alpha/2-1}r_\alpha(y)\dd y
\]
is well defined for $1\leq\alpha\leq A$, and the dominated convergence theorem shows that
$J_\alpha\to J_1$ as $\alpha\to1$. Since $J_1<0$, there exist
$\delta_0>0$ and $\kappa>0$, with $1+\delta_0<A$, such that
\[
J_\alpha\leq-\kappa,
\qquad
1\leq\alpha\leq1+\delta_0.
\]

We next prove that
\begin{equation}
I_\alpha(x)\longrightarrow J_\alpha,
\qquad x\downarrow0,
\label{eq:Ialpha-uniform-zero}
\end{equation}
uniformly for $1\leq\alpha\leq1+\delta_0$. Indeed, after the change of variables
$t=x+y$,
\[
I_\alpha(x)
=
\int_x^\infty
(t-x)^{\alpha/2-1}r_\alpha(t)\dd t.
\]
For $0<x\leq1/2$, we therefore have
\[
\begin{aligned}
|I_\alpha(x)-J_\alpha|
\leq{}&
\int_0^{2x}t^{\alpha/2-1}|r_\alpha(t)|\dd t\\
&+
\int_x^{2x}(t-x)^{\alpha/2-1}|r_\alpha(t)|\dd t\\
&+
\int_{2x}^\infty
\left|
(t-x)^{\alpha/2-1}-t^{\alpha/2-1}
\right|
|r_\alpha(t)|\dd t.
\end{aligned}
\]
Using $|r_\alpha(t)|\lesssim t^{1-\alpha}$ for $0<t\leq1$, the first term is bounded by
\[
C\int_0^{2x}t^{-\alpha/2}\dd t
\lesssim x^{1-\alpha/2}.
\]
On $[x,2x]$, we have $t\asymp x$, and the second term satisfies
\[
\begin{aligned}
\int_x^{2x}(t-x)^{\alpha/2-1}|r_\alpha(t)|\dd t
&\lesssim x^{1-\alpha}
\int_x^{2x}(t-x)^{\alpha/2-1}\dd t\\
&\lesssim x^{1-\alpha/2}.
\end{aligned}
\]
For $t\geq2x$, the mean value theorem gives, uniformly for
$1\leq\alpha\leq A$,
\[
\left|
(t-x)^{\alpha/2-1}-t^{\alpha/2-1}
\right|
\lesssim x\,t^{\alpha/2-2}.
\]
Hence
\[
x\int_{2x}^1
t^{\alpha/2-2}|r_\alpha(t)|\dd t
\lesssim x\int_{2x}^1t^{-\alpha/2-1}\dd t
\lesssim x^{1-\alpha/2},
\]
whereas, using $|r_\alpha(t)|\lesssim t^{-\alpha}$ for $t\geq1$,
\[
x\int_1^\infty
t^{\alpha/2-2}|r_\alpha(t)|\dd t
\lesssim x\int_1^\infty t^{-\alpha/2-2}\dd t
\lesssim x.
\]
Consequently,
\[
|I_\alpha(x)-J_\alpha|
\lesssim \bigl(x^{1-\alpha/2}+x\bigr)
\lesssim x^{1-A/2},
\qquad
1\leq\alpha\leq A,
\]
This proves \eqref{eq:Ialpha-uniform-zero}.
It follows from \eqref{eq:Ialpha-uniform-zero} and $J_\alpha\leq-\kappa$ that there
exists $0<x_0\leq\frac{1}{2}$ such that
\[
I_\alpha(x)<0,
\qquad
0<x\leq x_0,\quad
1\leq\alpha\leq1+\delta_0.
\]

For $x\geq1$, uniformly for
$1\leq\alpha\leq1+\delta_0$,
\[
\zeta(1+x+y)^\alpha-1
\lesssim 2^{-x-y}.
\]
Hence
\[
S_\alpha(x)
\lesssim x^{\alpha/2}2^{-x}
\int_0^\infty y^{\alpha/2-1}2^{-y}\dd y.
\]
The integral equals
\[
\frac{\Gamma(\alpha/2)}{(\log2)^{\alpha/2}}
\]
and is uniformly bounded for
$1\leq\alpha\leq1+\delta_0$. Therefore
\[
S_\alpha(x)
\lesssim x^{(1+\delta_0)/2}2^{-x}
\longrightarrow0,
\qquad x\to\infty,
\]
uniformly for $1\leq\alpha\leq1+\delta_0$. The continuity and strict positivity of $\Lambda_\alpha$ give
\[
\inf_{1\leq\alpha\leq1+\delta_0}\Lambda_\alpha>0.
\]
We may therefore choose once and for all $X>x_0$ such that
\[
S_\alpha(x)<\Lambda_\alpha,
\qquad
x\geq X,\quad
1\leq\alpha\leq1+\delta_0.
\]

It remains to control the fixed compact interval $[x_0,X]$. Since
$r_1(t)<0$ for every $t>0$,
\[
I_1(x)<0,
\qquad x>0,
\]
and therefore
\[
\max_{x_0\leq x\leq X}I_1(x)<0.
\]
Moreover, $(\alpha,x)\mapsto I_\alpha(x)$ is continuous on
$[1,1+\delta_0]\times[x_0,X]$. 
After decreasing $\delta_0$ if necessary, while keeping the already chosen
$X$ fixed, we obtain
\[
I_\alpha(x)<0,
\qquad
x_0\leq x\leq X,\quad
1\leq\alpha\leq1+\delta_0.
\]
The preceding large-$x$ estimate remains valid after this reduction of
$\delta_0$.

Combining the three regions, \eqref{eq:critical-Schur-function} gives
\[
S_\alpha(x)\leq\Lambda_\alpha,
\qquad
x>0,\quad
1<\alpha\leq1+\delta_0.
\]
Schur's test now gives $\|K_\alpha\|\leq\Lambda_\alpha$.
Thus
\begin{equation}\label{eq:subcritical-norm-equality}
\|K_\alpha\|=\Lambda_\alpha,
\qquad 1<\alpha\leq1+\delta_0.
\end{equation}

 Let
\[
U:L^2(\mathbb R_+,dx)\longrightarrow L^2(\mathbb R_+,\dd x/x),
\qquad
(Uf)(x)=x^{1/2}f(x),
\]
and put
\[
\widetilde K_\gamma=UK_\gamma U^{-1}.
\]
Let
\[
0\leq q\in L^2(\mathbb R_+,\dd x/x).
\]
For $\gamma>1$, define
\[
Q_\gamma(t)
=
t^\gamma\bigl(\zeta(1+t)^\gamma-1\bigr),
\]
and
\[
\mathcal A_\gamma(q;t)
=
\frac1{B(\gamma/2,\gamma/2)}
\int_0^1
q(tz)q(t(1-z))
z^{\gamma/2-1}(1-z)^{\gamma/2-1}\dd z.
\]
The change of variables
\[
t=x+y,\qquad z=\frac{x}{x+y}
\]
gives
\begin{equation}
\frac{\langle\widetilde K_\gamma q,q\rangle}
{\Lambda_\gamma}
=
\int_0^\infty
Q_\gamma(t)\mathcal A_\gamma(q;t)\frac{\dd t}{t}.
\label{eq:beta-quadratic-form}
\end{equation}

To compare $\|K_\alpha\|/\Lambda_\alpha$ and
$\|K_\beta\|/\Lambda_\beta$, we relate
$\mathcal A_\beta$ to $\mathcal A_\alpha$ by a transformation
of the test function that does not increase its $L^2$ norm.
For $1<\alpha<\beta$, define
\[
(C_{\alpha,\beta}q)(x)
=
\frac1{B(\alpha/2,(\beta-\alpha)/2)}
\int_0^1
q(rx)
r^{\alpha/2-1}
(1-r)^{(\beta-\alpha)/2-1}\dd r.
\]
By the Cauchy--Schwarz inequality,
\[
|C_{\alpha,\beta}q(x)|^2
\leq
\frac1{B(\alpha/2,(\beta-\alpha)/2)}
\int_0^1
|q(rx)|^2
r^{\alpha/2-1}
(1-r)^{(\beta-\alpha)/2-1}\dd r .
\]
Hence
\begin{equation}
\|C_{\alpha,\beta}q\|_{L^2(\dd x/x)}
\leq
\|q\|_{L^2(\dd x/x)} .
\label{eq:beta-contraction}
\end{equation}
By the definitions of $\mathcal A_\beta$ and $C_{\alpha,\beta}$,
\[
\begin{aligned}
\mathcal A_\beta(C_{\alpha,\beta}q;t)
&={}
\frac{1}{
B(\beta/2,\beta/2)
B(\alpha/2,(\beta-\alpha)/2)^2}\\
&
\int_{(0,1)^3}
q(tzr_1)q(t(1-z)r_2)\,
z^{\beta/2-1}(1-z)^{\beta/2-1}
\prod_{j=1}^2
r_j^{\alpha/2-1}(1-r_j)^{(\beta-\alpha)/2-1}
\dd r_1\dd r_2\dd z .
\end{aligned}
\]
Set
\[
x_1=zr_1,\qquad y_1=z(1-r_1),\qquad
x_2=(1-z)r_2,\qquad y_2=(1-z)(1-r_2).
\]
Since \(x_1+x_2+y_1+y_2=1\), we use \(x_1,x_2,y_1\)
as integration variables and put
\(y_2=1-x_1-x_2-y_1\). The absolute value of the
Jacobian is \(z(1-z)\). Consequently, the integral above becomes
\[
\int_{\Omega}
q(tx_1)q(tx_2)
x_1^{\alpha/2-1}x_2^{\alpha/2-1}
y_1^{(\beta-\alpha)/2-1}
y_2^{(\beta-\alpha)/2-1}
\dd x_1\dd x_2\dd y_1,
\]
where
\[
\Omega=\{(x_1,x_2,y_1):x_1,x_2,y_1>0,\
x_1+x_2+y_1<1\}.
\]
Next set
\[
x_1=su,\qquad x_2=s(1-u),\qquad
y_1=(1-s)v,\qquad y_2=(1-s)(1-v).
\]
The absolute value of this Jacobian is \(s(1-s)\).
The integral over \(\Omega\) is therefore
\[
\begin{aligned}
\int_0^1
&s^{\alpha-1}(1-s)^{\beta-\alpha-1}
\left[
\int_0^1 q(tsu)q(ts(1-u))
u^{\alpha/2-1}(1-u)^{\alpha/2-1}\dd u
\right]\\
&\times
\left[
\int_0^1
v^{(\beta-\alpha)/2-1}
(1-v)^{(\beta-\alpha)/2-1}\dd v
\right]\dd s\\
={}&
B(\alpha/2,\alpha/2)
B((\beta-\alpha)/2,(\beta-\alpha)/2)
\int_0^1
s^{\alpha-1}(1-s)^{\beta-\alpha-1}
\mathcal A_\alpha(q;st)\dd s .
\end{aligned}
\]
Finally,
\[
\begin{aligned}
&B(\beta/2,\beta/2)
B(\alpha/2,(\beta-\alpha)/2)^2\\
&\qquad=
B(\alpha/2,\alpha/2)
B((\beta-\alpha)/2,(\beta-\alpha)/2)
B(\alpha,\beta-\alpha).
\end{aligned}
\]
It follows that
\begin{equation}\label{eq:beta-composition}
\mathcal A_\beta(C_{\alpha,\beta}q;t)
=
\frac1{B(\alpha,\beta-\alpha)}
\int_0^1
s^{\alpha-1}(1-s)^{\beta-\alpha-1}
\mathcal A_\alpha(q;st)\dd s.
\end{equation}
Combining
\eqref{eq:beta-quadratic-form} and
\eqref{eq:beta-composition}, we obtain
\begin{equation}\label{eq:beta-transformed-quadratic-form}
\frac{
\langle
\widetilde K_\beta C_{\alpha,\beta}q,
C_{\alpha,\beta}q
\rangle
}{\Lambda_\beta}
=
\int_0^\infty
\mathcal A_\alpha(q;t)
R_{\alpha,\beta}(t)\frac{\dd t}{t}.
\end{equation}
where
\[
R_{\alpha,\beta}(t)
=
\frac1{B(\alpha,\beta-\alpha)}
\int_0^1
s^{\alpha-1}(1-s)^{\beta-\alpha-1}
Q_\beta(t/s)\dd s .
\]
We next prove that
\begin{equation}\label{eq:Q-comparison}
R_{\alpha,\beta}(t)>Q_\alpha(t),
\qquad t>0.
\end{equation}
The substitution $u=t(1-s)/s$ gives
\[
R_{\alpha,\beta}(t)
=
\frac{t^\alpha}{B(\alpha,\beta-\alpha)}
\int_0^\infty
u^{\beta-\alpha-1}
\bigl(\zeta(1+t+u)^\beta-1\bigr)\dd u.
\]
It therefore suffices to prove
\begin{equation}\label{eq:fractional-zeta-comparison}
\begin{aligned}
&\frac{\Gamma(\beta)}
{\Gamma(\alpha)\Gamma(\beta-\alpha)}
\int_0^\infty
u^{\beta-\alpha-1}
\bigl(\zeta(1+t+u)^\beta-1\bigr)\dd u>\zeta(1+t)^\alpha-1,
\qquad t>0.
\end{aligned}
\end{equation}

Fix $t>0$. Lemma~\ref{lem:zeta-log-derivative-bound} gives
\[
-\frac{\zeta'(1+t+u)}{\zeta(1+t+u)^2}
\leq
1-\frac1{\zeta(1+t+u)},
\qquad u\geq0.
\]
Thus
\[
\frac{\dd}{\dd u}
\left[
e^u\left(1-\frac1{\zeta(1+t+u)}\right)
\right]\geq0,
\]
and hence
\[
1-\frac1{\zeta(1+t+u)}
\geq
\left(1-\frac1{\zeta(1+t)}\right)e^{-u}.
\]
Rearranging, we obtain
\[
\zeta(1+t+u)
\geq
\left[
1-\left(1-\zeta(1+t)^{-1}\right)e^{-u}
\right]^{-1},
\qquad u\geq0.
\]

The binomial expansion gives
\[
\begin{aligned}
&\left[
1-\left(1-\zeta(1+t)^{-1}\right)e^{-u}
\right]^{-\beta}-1\\
&\qquad=
\sum_{k=1}^\infty
\frac{\Gamma(k+\beta)}
{\Gamma(\beta)\Gamma(k+1)}
\left(1-\zeta(1+t)^{-1}\right)^k e^{-ku}.
\end{aligned}
\]
All summands are non-negative. Applying Tonelli's theorem and using
\[
\int_0^\infty u^{\beta-\alpha-1}e^{-ku}\dd u
=
\Gamma(\beta-\alpha)k^{-(\beta-\alpha)},
\]
we obtain
\[
\begin{aligned}
&\frac{\Gamma(\beta)}
{\Gamma(\alpha)\Gamma(\beta-\alpha)}
\int_0^\infty
u^{\beta-\alpha-1}
\bigl(\zeta(1+t+u)^\beta-1\bigr)\dd u\\
&\qquad\geq
\frac1{\Gamma(\alpha)}
\sum_{k=1}^\infty
\frac{\Gamma(k+\beta)}{\Gamma(k+1)}
k^{-(\beta-\alpha)}
\left(1-\zeta(1+t)^{-1}\right)^k.
\end{aligned}
\]

For $k\geq1$ and $x>1$, the strict convexity of $\log\Gamma$ gives
\[
\begin{aligned}
\frac{\Gamma'(k+x)}{\Gamma(k+x)}
&>\log\Gamma(k+x)-\log\Gamma(k+x-1)\\
&=\log(k+x-1)>\log k.
\end{aligned}
\]
Integration over $x\in[\alpha,\beta]$ yields
\[
\Gamma(k+\beta)k^{-(\beta-\alpha)}
>
\Gamma(k+\alpha),
\qquad k\geq1.
\]
As $0<1-\zeta(1+t)^{-1}<1$, comparison of the series gives
\[
\begin{aligned}
&\frac1{\Gamma(\alpha)}
\sum_{k=1}^\infty
\frac{\Gamma(k+\beta)}{\Gamma(k+1)}
k^{-(\beta-\alpha)}
\left(1-\zeta(1+t)^{-1}\right)^k\\
&\qquad>
\frac1{\Gamma(\alpha)}
\sum_{k=1}^\infty
\frac{\Gamma(k+\alpha)}{\Gamma(k+1)}
\left(1-\zeta(1+t)^{-1}\right)^k\\
&\qquad=\zeta(1+t)^\alpha-1.
\end{aligned}
\]
This proves \eqref{eq:fractional-zeta-comparison} and
\eqref{eq:Q-comparison}.

We now prove the monotonicity of
$\|K_\alpha\|/\Lambda_\alpha$
asserted in Theorem~\ref{thm:complete-spectrum-Ta}. Let \(1<\alpha<\beta\) and let
\[
  0\ne q\in L^2(\mathbb R_+,\dd x/x),
  \qquad q\geq0.
\]
Since \(R_{\alpha,\beta}(t)>Q_\alpha(t)\) for \(t>0\) and
\(\mathcal A_\alpha(q;\cdot)\geq0\) is not identically zero,
\eqref{eq:beta-quadratic-form} and
\eqref{eq:beta-transformed-quadratic-form} give
\[
  \frac{\langle
    \widetilde K_\beta C_{\alpha,\beta}q,
    C_{\alpha,\beta}q
  \rangle}
  {\Lambda_\beta}
  >
  \frac{\langle\widetilde K_\alpha q,q\rangle}
  {\Lambda_\alpha}.
\]
Moreover, \(C_{\alpha,\beta}q\ne0\), and
\eqref{eq:beta-contraction} yields
\[
\begin{aligned}
  \frac{\|K_\beta\|}{\Lambda_\beta}
  &\geq
  \frac{\langle
    \widetilde K_\beta C_{\alpha,\beta}q,
    C_{\alpha,\beta}q
  \rangle}
  {\Lambda_\beta\|C_{\alpha,\beta}q\|^2}\\
  &>
  \frac{\langle\widetilde K_\alpha q,q\rangle}
  {\Lambda_\alpha\|C_{\alpha,\beta}q\|^2}\\
  &\geq
  \frac{\langle\widetilde K_\alpha q,q\rangle}
  {\Lambda_\alpha\|q\|^2}.
\end{aligned}
\]
The non-negativity of the kernel of \(K_\alpha\) allows its norm to
be computed using non-negative test functions. Taking the supremum over such
functions gives
\begin{equation}
  \frac{\|K_\beta\|}{\Lambda_\beta}
  \geq
  \frac{\|K_\alpha\|}{\Lambda_\alpha},
  \qquad 1<\alpha<\beta.
  \label{eq:normalized-norm-monotonicity}
\end{equation}
The inequality is strict if \(\|K_\alpha\|>\Lambda_\alpha\).
Indeed, in this case \(\|K_\alpha\|\) is an eigenvalue above the
essential spectrum. Since the kernel of \(K_\alpha\) is strictly
positive, a corresponding normalized eigenfunction \(f_\alpha\) may
be chosen strictly positive. Taking \(q=Uf_\alpha\) in the preceding
estimate gives
\begin{equation}
  \frac{\|K_\beta\|}{\Lambda_\beta}
  >
  \frac{\|K_\alpha\|}{\Lambda_\alpha},
  \qquad
  1<\alpha<\beta,
  \quad
  \|K_\alpha\|>\Lambda_\alpha.
  \label{eq:normalized-norm-strict-monotonicity}
\end{equation}

Recall
\[
E
=
\{\alpha>1:\|K_\alpha\|>\Lambda_\alpha\}.
\]
By \eqref{eq:subcritical-norm-equality}, $E$ is disjoint from a right
neighborhood of $1$, while Example~\ref{ex:E-nonempty} gives
\[
2\in E,
\]
and the set $E$ is upward closed by 
\eqref{eq:normalized-norm-monotonicity}. 
We next show that $E$ is open. Fix $\alpha_0\in E$. Since the kernel of
$K_{\alpha_0}$ is non-negative, there exists
$0\leq f\in L^2(\mathbb R_+,dx)$ such that
\[
\langle K_{\alpha_0}f,f\rangle_{L^2(\dd x)}
>
\Lambda_{\alpha_0}\|f\|_{L^2(\dd x)}^2.
\]
The boundedness of $K_{\alpha_0}$ ensures that the strict inequality
is preserved under sufficiently small $L^2$-perturbations.
By the density of $C_c^\infty(0,\infty)$ in $L^2(\mathbb R_+,dx)$,
we may choose
\[
0\leq f\in C_c^\infty(0,\infty)
\]
with the same strict inequality. Choose $0<a<b<\infty$ such that
\[
\operatorname{supp}f\subset[a,b].
\]
On $[a,b]^2$, the kernel $k_\alpha(x,y)$ depends continuously on
$\alpha$, uniformly for $\alpha$ in a neighborhood of $\alpha_0$.
Therefore
\[
\alpha\longmapsto
\langle K_\alpha f,f\rangle_{L^2(\dd x)}
-
\Lambda_\alpha\|f\|_{L^2(\dd x)}^2
\]
is continuous and remains positive for $\alpha$ sufficiently close to
$\alpha_0$. Hence $E$ is open.

Set
\[
\alpha_*=\inf E.
\]
By \eqref{eq:subcritical-norm-equality} and $[2,\infty)\subset E$,
\[
1<\alpha_*<2.
\]
The openness and upward closure of \(E\) imply
\(E=(\alpha_*,\infty)\). The bound
\(\|K_\alpha\|\geq\Lambda_\alpha\) then gives
\[
\frac{\|K_\alpha\|}{\Lambda_\alpha}=1,
\qquad 1<\alpha\leq\alpha_*.
\]
Moreover,
\[
\frac{\|K_\beta\|}{\Lambda_\beta}>
\frac{\|K_\alpha\|}{\Lambda_\alpha},
\qquad \alpha_*\leq\alpha<\beta:
\]
for \(\alpha>\alpha_*\) this follows from
\eqref{eq:normalized-norm-strict-monotonicity}, while for
\(\alpha=\alpha_*\) it follows from \(\beta\in E\).
Thus $\alpha\mapsto\|K_\alpha\|/\Lambda_\alpha$ is strictly increasing
on $[\alpha_*,\infty)$. Theorem~\ref{thm:Ta-unitary-Ka} transfers
these statements from \(K_\alpha\) to \(T_\alpha\), and proves
part~(5) of Theorem~\ref{thm:complete-spectrum-Ta}.

\end{proof}
  
\section{Norm asymptotics and open problems}
\label{sec:norm-open-problems}

Theorem~\ref{thm:complete-spectrum-Ta} determines the norm of
\(T_\alpha\) for \(1<\alpha\leq\alpha_*\), but does not give an
explicit formula when \(\alpha>\alpha_*\). We first determine its
asymptotic behavior as \(\alpha\to\infty\), and then discuss its
dependence on \(\alpha\) and the question of explicit diagonalization.

\subsection{Asymptotic behavior of the norm}

\begin{theorem}\label{thm:large-alpha-norm}
There is a constant \(c>0\) such that
\begin{equation}\label{eq:large-alpha-norm}
  \|T_\alpha\|
  =\frac{\Gamma(\alpha+1)}{2(2\log2)^\alpha}
   \bigl(1+O(e^{-c\alpha})\bigr),
  \qquad \alpha\to\infty.
\end{equation}
In particular,
\[
  \|T_\alpha\|
  \sim\sqrt{\frac{\pi\alpha}{2}}
  \left(\frac{\alpha}{2e\log2}\right)^\alpha.
\]
\end{theorem}

\begin{proof}
By Theorem~\ref{thm:Ta-unitary-Ka}, it suffices to estimate
\(\|K_\alpha\|\). Let \(P_\alpha\) be the rank-one operator with kernel
\[
  p_\alpha(x,y)
  =\frac{\alpha}{2}(xy)^{(\alpha-1)/2}
    e^{-(\log2)(x+y)}.
\]
Then
\begin{equation}\label{eq:large-alpha-rank-one-norm}
  \|P_\alpha\|
  =\frac{\alpha}{2}\int_0^\infty
    x^{\alpha-1}e^{-2(\log2)x}\dd x
  =\frac{\Gamma(\alpha+1)}{2(2\log2)^\alpha}.
\end{equation}
For \(q>0\), Bernoulli's inequality gives
\[
  \zeta(1+q)^\alpha-1
  \geq\alpha\bigl(\zeta(1+q)-1\bigr)
  \geq\frac{\alpha}{2}2^{-q}.
\]
Thus \(k_\alpha(x,y)\geq p_\alpha(x,y)\). Testing \(K_\alpha\)
on the positive normalized eigenfunction of \(P_\alpha\) yields
\[
  \|K_\alpha\|\geq\|P_\alpha\|.
\]

We next estimate $\|K_\alpha-P_\alpha\|$. For $q\geq1$, first observe that, for \(q\geq1\),
\[
\begin{aligned}
  0\leq\zeta(1+q)-1-\frac12\,2^{-q}
  &=\sum_{n=3}^\infty n^{-1-q}\\
  &\leq3^{-1-q}+\int_3^\infty t^{-1-q}\dd t\\
  &=3^{-q}\left(\frac13+\frac1q\right).
\end{aligned}
\]
Consequently,
\[
  \zeta(1+q)-1
  =\frac12\,2^{-q}\bigl(1+O((2/3)^q)\bigr),
  \qquad q\geq1.
\]
For $\alpha\geq1$ and $0<u\leq1/\alpha$,
\[
\frac{(1+u)^\alpha-1}{\alpha u}
=\int_0^1(1+su)^{\alpha-1}\dd s.
\]
Since
\[
1\leq(1+su)^{\alpha-1}\leq e^{\alpha u},
\qquad 0\leq s\leq1,
\]
we have
\[
0\leq
\frac{(1+u)^\alpha-1}{\alpha u}-1
\leq e^{\alpha u}-1
\leq e\,\alpha u.
\]
Thus
\[
(1+u)^\alpha-1
=\alpha u\bigl(1+O(\alpha u)\bigr),
\]
 Since
\[
  \sup_{q\geq\alpha/4}
  \alpha\bigl(\zeta(1+q)-1\bigr)
  =O(\alpha2^{-\alpha/4})\longrightarrow0,
\]
we obtain
\[
  \zeta(1+q)^\alpha-1
  =\frac{\alpha}{2}2^{-q}
   \left[1+O\left((2/3)^{\alpha/4}
                    +\alpha2^{-\alpha/4}\right)\right]
\]
uniformly for \(q\geq\alpha/4\) as \(\alpha\to\infty\).
It follows that, whenever \(x+y\geq\alpha/4\),
\[
\begin{aligned}
  0\leq k_\alpha(x,y)-p_\alpha(x,y)
  \leq C\left((2/3)^{\alpha/4}
                  +\alpha2^{-\alpha/4}\right)p_\alpha(x,y),
\end{aligned}
\]
where \(C\) is independent of all sufficiently large \(\alpha\).

For \(0<q<\alpha/4\), the inequality
\[
  \zeta(1+q)\leq1+\int_1^\infty t^{-1-q}\dd t
  =1+\frac1q
\]
gives
\[
  \zeta(1+q)^\alpha-1
  \leq q^{-\alpha}(1+q)^\alpha
  \leq q^{-\alpha}(1+\alpha/4)^\alpha.
\]
Combining these estimates, we have, for all \(x,y>0\),
\[
\begin{aligned}
  0\leq k_\alpha(x,y)-p_\alpha(x,y)
  &\leq(1+\alpha/4)^\alpha
       \frac{(xy)^{(\alpha-1)/2}}{(x+y)^\alpha}+C\left((2/3)^{\alpha/4}
                     +\alpha2^{-\alpha/4}\right)p_\alpha(x,y).
\end{aligned}
\]
By \eqref{eq:Mellin-model}, the operator with kernel
\((xy)^{(\alpha-1)/2}(x+y)^{-\alpha}\) has norm \(\Lambda_\alpha\).
For $f\in L^2(\mathbb R_+)$, the preceding kernel estimate gives
\[
\begin{aligned}
|(K_\alpha-P_\alpha)f|
&\leq (1+\alpha/4)^\alpha C_\rho|f|+C\left((2/3)^{\alpha/4}
+\alpha2^{-\alpha/4}\right)P_\alpha|f|,
\end{aligned}
\]
where $\rho=(\alpha-1)/2$. Taking $L^2$ norms and using
$\|C_\rho\|=\Lambda_\alpha$, we obtain
\[
\begin{aligned}
  \|K_\alpha-P_\alpha\|
  &\leq(1+\alpha/4)^\alpha\Lambda_\alpha+C\left((2/3)^{\alpha/4}
                     +\alpha2^{-\alpha/4}\right)\|P_\alpha\|.
\end{aligned}
\]
Stirling's formula \cite[Eq.~(5.11.3)]{NIST2010}, together with
\eqref{eq:large-alpha-rank-one-norm}, gives
\[
  \frac{(1+\alpha/4)^\alpha\Lambda_\alpha}{\|P_\alpha\|}
  =O\left(\frac1\alpha
    \left(\frac{e\log2}{4}\right)^\alpha\right).
\]
Choose
\[
0<c<
\min\left\{
\log\frac{4}{e\log2},
\frac14\log\frac32,
\frac14\log2
\right\}.
\]
Then
\[
\frac1\alpha\left(\frac{e\log2}{4}\right)^\alpha
+(2/3)^{\alpha/4}
+\alpha2^{-\alpha/4}
=O(e^{-c\alpha}),
\]
and hence
\[
\|K_\alpha-P_\alpha\|
=O(e^{-c\alpha})\|P_\alpha\|.
\]
Combining this estimate with the lower bound gives
\[
\|P_\alpha\|
\leq\|K_\alpha\|
\leq\|P_\alpha\|+\|K_\alpha-P_\alpha\|
\leq\bigl(1+Ce^{-c\alpha}\bigr)\|P_\alpha\|.
\]
Substituting \eqref{eq:large-alpha-rank-one-norm} and using
$\|T_\alpha\|=\|K_\alpha\|$ proves \eqref{eq:large-alpha-norm}.
Finally, Stirling's formula gives
\[
\frac{\Gamma(\alpha+1)}{2(2\log2)^\alpha}
\sim
\sqrt{\frac{\pi\alpha}{2}}
\left(\frac{\alpha}{2e\log2}\right)^\alpha.
\]
\end{proof}

\subsection{Dependence of the norm on the parameter}

The asymptotic formula alone does not imply eventual monotonicity.
For this purpose we use the following convexity property.

\begin{lemma}\label{lem:norm-divided-alpha-logconvex}
The function
\[
  \alpha\longmapsto\log\frac{\|T_\alpha\|}{\alpha}
\]
is convex on \((1,\infty)\).
\end{lemma}

\begin{proof}
For $x,y>0$,
\[
  \frac{k_\alpha(x,y)}{\alpha}
  =\frac1{\sqrt{xy}}
   \int_0^{\log\zeta(1+x+y)}
   \exp\left[\alpha\left(\frac12\log(xy)+s\right)\right]\dd s.
\]
Fix $\alpha,\beta>1$, $0<\theta<1$, and put
$\gamma=(1-\theta)\alpha+\theta\beta$.
H\"older's inequality, with exponents $1/(1-\theta)$ and
$1/\theta$, gives
\[
  \frac{k_\gamma(x,y)}{\gamma}
  \leq
  \left(\frac{k_\alpha(x,y)}{\alpha}\right)^{1-\theta}
  \left(\frac{k_\beta(x,y)}{\beta}\right)^\theta.
\]
For every non-negative $f\in L^2(\mathbb R_+)$ with $\|f\|_2=1$,
a second application of H\"older's inequality yields
\[
\begin{aligned}
  \frac{\langle K_\gamma f,f\rangle}{\gamma}
  &\leq
  \int_0^\infty\int_0^\infty
  \left(\frac{k_\alpha(x,y)f(x)f(y)}{\alpha}\right)^{1-\theta}
  \left(\frac{k_\beta(x,y)f(x)f(y)}{\beta}\right)^\theta
  \dd x\dd y\\
  &\leq
  \left(\frac{\langle K_\alpha f,f\rangle}{\alpha}\right)^{1-\theta}
  \left(\frac{\langle K_\beta f,f\rangle}{\beta}\right)^\theta.
\end{aligned}
\]
Each \(K_\alpha\) is non-negative and has a non-negative kernel, so
\[
  \|K_\alpha\|
  =\sup_{\substack{f\geq0\\\|f\|_2=1}}
    \langle K_\alpha f,f\rangle.
\]
Taking the supremum and using Theorem~\ref{thm:Ta-unitary-Ka}, we obtain
\[
  \frac{\|T_\gamma\|}{\gamma}
  \leq
  \left(\frac{\|T_\alpha\|}{\alpha}\right)^{1-\theta}
  \left(\frac{\|T_\beta\|}{\beta}\right)^\theta.
\]
\end{proof}

\begin{corollary}\label{cor:norm-decrease-increase}
The function \(\alpha\mapsto\|T_\alpha\|\) is strictly decreasing
on \((1,\alpha_*]\) and strictly increasing on \((A,\infty)\) for
some \(A>\alpha_*\). It attains its global minimum at a parameter
in \([\alpha_*,\infty)\).
\end{corollary}

\begin{proof}
For \(1<\alpha\leq\alpha_*\),
Theorem~\ref{thm:complete-spectrum-Ta} gives
\[
  \|T_\alpha\|=\Lambda_\alpha
  =\int_0^1[t(1-t)]^{\alpha/2-1}\dd t.
\]
Since \(0<t(1-t)<1\) for \(0<t<1\), this is strictly decreasing
in \(\alpha\).

Put
\[
  F(\alpha)=\log\frac{\|T_\alpha\|}{\alpha}.
\]
By Lemma~\ref{lem:norm-divided-alpha-logconvex}, \(F\) is convex,
and Theorem~\ref{thm:large-alpha-norm} gives
\[
  \frac{F(\alpha)}{\alpha}
  =\log\alpha-1-\log(2\log2)+o(1)
  \longrightarrow\infty.
\]
Choose \(1<a<b\) such that \(F(b)>F(a)\). Convexity implies that,
for \(b\leq x<y\),
\[
  \frac{F(y)-F(x)}{y-x}
  \geq\frac{F(b)-F(a)}{b-a}>0.
\]
Thus \(\|T_\alpha\|=\alpha e^{F(\alpha)}\) is strictly increasing
for all sufficiently large \(\alpha\).
Finally, convexity of \(F\) implies continuity of \(\|T_\alpha\|\)
on \((1,\infty)\). The decrease on \((1,\alpha_*]\), together with
\(\|T_\alpha\|\to\infty\), shows that its minimum is attained in
\([\alpha_*,\infty)\).
\end{proof}

The preceding results leave open the behavior between these two
intervals of monotonicity.

\smallskip
\noindent\textbf{Question 1.}
Does there exist a unique \(\alpha_{\min}>1\) such that
\(\|T_\alpha\|\) is strictly decreasing on \((1,\alpha_{\min})\)
and strictly increasing on \((\alpha_{\min},\infty)\)?

\smallskip
The parameter \(\alpha_*\) describes the appearance of eigenvalues
above the essential spectrum. It does not, by itself, determine
where the norm changes from decreasing to increasing. More generally,
for \(\alpha>\alpha_*\), the norm is the largest eigenvalue of
\(T_\alpha\), and this eigenvalue is simple. However, an explicit formula for $\|T_\alpha\|$ when
$\alpha>\alpha_*$ remains unknown.

\smallskip
\noindent\textbf{Question 2.}
Can \(\|T_\alpha\|\) be computed explicitly for
\(\alpha>\alpha_*\)?

\subsection{Explicit diagonalization}

Perfekt and Pushnitski proved that the multiplicative Hilbert matrix
has purely absolutely continuous spectrum \([0,\pi]\) of multiplicity
one. They also posed the problem of finding an explicit
diagonalization; see \cite[Section~1]{PerfektPushnitski2018Spectrum}.
This is the case \(\alpha=1\) of the matrix considered here, since
\(d_1(n)=1\) and \(\Gamma(1)=1\).

For \(1<\alpha\leq\alpha_*\),
Theorem~\ref{thm:complete-spectrum-Ta} shows that \(T_\alpha\) has
purely absolutely continuous spectrum \([0,\Lambda_\alpha]\) of
multiplicity one. The spectral theorem provides a unitary
representation as multiplication by the spectral variable, but the
proof does not give an explicit formula for the unitary transform.

\smallskip
\noindent\textbf{Question 3.}
For \(1<\alpha\leq\alpha_*\), can one construct explicitly a
unitary operator
\[
  U_\alpha:\A_{\alpha,0}^2
  \longrightarrow L^2((0,\Lambda_\alpha),\dd\lambda)
\]
such that
\[
  (U_\alpha T_\alpha U_\alpha^{-1}g)(\lambda)
  =\lambda g(\lambda)?
\]

\vspace{0.5cm}

\addtocontents{toc}{\protect\setcounter{tocdepth}{0}}

\section*{AI disclosure} During the development of this work, the authors used OpenAI’s ChatGPT (GPT-5.6) as an auxiliary research tool. In the proof excluding non-zero eigenvalues from the essential spectrum, the key idea used to rule out the second alternative was suggested by GPT-5.6 during discussions with the authors. In particular, it suggested using strict total positivity to obtain the contradiction. The authors subsequently studied the relevant theory and verified the argument in detail. This also led to the use of strict total positivity and exterior powers in the proof that the eigenvalues above the essential spectrum are simple. The main ideas underlying the remaining proofs were developed
by the authors. AI was used only to assist with calculations
and verification in those proofs. All mathematical statements, proofs, references, and the final presentation were checked and revised by the authors, who take full responsibility for the content of the paper. 

\section*{Acknowledgements}
Y. Lu was supported by the National Natural Science Foundation of China
(Grant No. 12031002 and No. 12671149). C. Zu was supported by the National Natural Science Foundation of China (Grant No. 12401151).

\section*{Conflict of interest}
The authors have no conflict of interest to declare that are relevant to the content of this article.

\section*{Data availability statement}
Data sharing is not applicable to this article as no new data were created or analyzed in this study.

\addtocontents{toc}
{\protect\setcounter{tocdepth}{2}}\bibliographystyle{amsplain}
\bibliography{ref}

\end{document}